\documentclass[10pt]{amsart}
      \usepackage{amsmath,amsfonts}
\def\rg{\hbox to 30pt{\rightarrowfill}}
\def\lg{\hbox to 30pt{\leftarrowfill}}

      \parskip\smallskipamount
          \newtheorem{theorem}{Theorem}[section]
      \newtheorem{definition}[theorem]{Definition}
      \newtheorem{proposition}[theorem]{Proposition}
      \newtheorem{corollary}[theorem]{Corollary}
      \newtheorem{lemma}[theorem]{Lemma}

      \newtheorem{remark}[theorem]{Remark}
      \makeatletter
      \@addtoreset{equation}{section}
      \makeatother

      \newcommand{\BB}{{\mathbb B}}
      \newcommand{\CC}{{\mathbb C}}
      \newcommand{\NN}{{\mathbb N}}
      
      \newcommand{\ZZ}{{\mathbb Z}}
      \newcommand{\DD}{{\mathbb D}}
      
      \newcommand{\FF}{{\mathbb F}}
      \newcommand{\TT}{{\mathbb T}}

      \newcommand{\cA}{{\mathcal A}}
      \newcommand{\cB}{{\mathcal B}}
      \newcommand{\cC}{{\mathcal C}}
      \newcommand{\cD}{{\mathcal D}}
      \newcommand{\cE}{{\mathcal E}}
      \newcommand{\cF}{{\mathcal F}}
      \newcommand{\cG}{{\mathcal G}}
      \newcommand{\cH}{{\mathcal H}}
      \newcommand{\cK}{{\mathcal K}}
      \newcommand{\cL}{{\mathcal L}}
      \newcommand{\cM}{{\mathcal M}}
      \newcommand{\cN}{{\mathcal N}}
      
      \newcommand{\cQ}{{\mathcal Q}}
      \newcommand{\cP}{{\mathcal P}}
      \newcommand{\cR}{{\mathcal R}}
      
      \newcommand{\cT}{{\mathcal T}}
      
      \newcommand{\cV}{{\mathcal V}}
      \newcommand{\cY}{{\mathcal Y}}
      \newcommand{\cX}{{\mathcal X}}

      \newcommand{\supp}{\hbox{\rm{supp}}\,}
      \newcommand{\rank}{\hbox{\rm{rank}}\,}

      \newdimen\expt
      \def\boxit#1{\setbox0\hbox{$\displaystyle{#1}$}
            \hbox{\lower.4\expt
       \hbox{\lower3\expt\hbox{\lower\dp0
            \hbox{\vbox{\hrule height.4\expt
       \hbox{\vrule width.4\expt\hskip3\expt
            \vbox{\vskip3\expt\box0\vskip2\expt}%
       \hskip3\expt\vrule width.4\expt}\hrule height.4\expt}}}}}}
      
\begin{document}
       \pagestyle{myheadings}
      \markboth{ Gelu Popescu}{  Pluriharmonic functions on noncommutative polyballs  }

      \title [ Free pluriharmonic functions on noncommutative  polyballs   ]
      { Free pluriharmonic functions on noncommutative   polyballs }
        \author{Gelu Popescu}
\date{October 12, 2015}
      \thanks{Research supported in part by  NSF grant DMS 1500922}

      \address{Department of Mathematics, The University of Texas
      at San Antonio \\ San Antonio, TX 78249, USA}
      \email{\tt gelu.popescu@utsa.edu}

\begin{abstract} In this paper, we study   free $k$-pluriharmonic functions on the noncommutative regular polyball   ${\bf B_n}$, ${\bf n}=(n_1,\ldots, n_k)\in \NN^k$,  which  is an  analogue of the scalar polyball \, $(\CC^{n_1})_1\times\cdots \times  (\CC^{n_k})_1$. The regular  polyball   has a universal model $ {\bf S}:=\{{\bf S}_{i,j}\}$ consisting of   left creation
operators acting on the  tensor product $F^2(H_{n_1})\otimes \cdots \otimes F^2(H_{n_k})$ of full Fock spaces. We introduce the class $\boldsymbol{\cT_{\bf n}}$ of
 $k$-multi-Toeplitz operators on  this tensor product  and prove that
  $\boldsymbol{\cT_{\bf n}}=\text{\rm span} \{\boldsymbol\cA_{\bf n}^* \boldsymbol\cA_{\bf n}\}^{- \text{\rm SOT}}$,
 where $\boldsymbol\cA_{\bf n}$ is the noncommutative polyball algebra generated by ${\bf S}$ and the identity. We show that the bounded free $k$-pluriharmonic functions on ${\bf B_n}$ are precisely the noncommutative Berezin transforms of $k$-multi-Toeplitz operators.
 The Dirichlet extension problem on regular polyballs is also solved. It is proved that a free $k$-pluriharmonic function has continuous extension to the closed polyball ${\bf B}_{\bf n}^-$ if and only if it is the noncommutative Berezin transform of a $k$-multi-Toeplitz operator in $\text{\rm span} \{\boldsymbol\cA_{\bf n}^* \boldsymbol\cA_{\bf n}\}^{-\|\cdot\|}$.

We  provide a Naimark type dilation theorem for  direct products
  $\FF_{n_1}^+\times \cdots \times \FF_{n_k}^+$  of unital free semigroups, and use it to obtain a structure theorem which  characterizes the positive free $k$-pluriharmonic functions on the regular polyball with operator-valued coefficients. We define the noncommutative  Berezin (resp.~ Poisson) transform of a  completely bounded linear map on $C^*({\bf S})$, the $C^*$-algebra generated by ${\bf S}_{i,j}$, and give necessary and sufficient conditions for a function to be the   Poisson transform of a completely bounded (resp.~completely positive) map. In the last section of the paper, we obtain Herglotz-Riesz representation theorems for free holomorphic functions on regular polyballs with positive real parts, extending the classical result as well as  Kor\' annyi-Puk\' anszky   version in  scalar polydisks.
\end{abstract}

      \maketitle

\section*{Contents}
{\it

\quad Introduction

\begin{enumerate}
   \item[1.]     $k$-multi-Toeplitz operators on tensor products of full  Fock spaces
\item[2.]  Bounded free $k$-pluriharmonic functions
and  Dirichlet extension problem
\item[3.] Naimark type dilation theorem for direct products of free semigroups
\item[4.] Berezin transforms of completely bounded maps in regular polyballs
\item[5.]   Herglotz-Riesz representations for free holomorphic functions with positive real parts
\end{enumerate}

\quad References

}

\section*{Introduction}

A multivariable  operator model theory and a theory of free holomorphic functions on polydomains which admit universal operator models have been recently developed  in \cite{Po-Berezin3} and \cite{Po-Berezin-poly}. An important feature of these theories is that they  are related, via noncommutative Berezin transforms, to the study of the operator algebras generated by the universal models as well as to the theory of functions in several complex variable. These results played a crucial role in   our work on
  the curvature invariant \cite{Po-curvature-polyball}, the Euler characteristic  \cite{Po-Euler-charact}, and the  group of  free holomorphic automorphisms  on  noncommutative regular polyballs \cite{Po-automorphisms-polyball}.

The main goal of the present paper is to continue our investigation along these lines and to study the class of free $k$-pluriharmonic functions      of  the form
$$
F({\bf X})= \sum_{m_1\in \ZZ}\cdots \sum_{m_k\in \ZZ} \sum_{{\alpha_i,\beta_i\in \FF_{n_i}^+}\atop{|\alpha_i|=m_i^-, |\beta_i|=m_i^+}}a_{(\boldsymbol\alpha;\boldsymbol\beta)}
 {\bf X}_{1,\alpha_1}\cdots {\bf X}_{k,\alpha_k}{\bf X}_{1,\beta_1}^*\cdots {\bf X}_{k,\beta_k}^*,\qquad a_{(\boldsymbol\alpha;\boldsymbol\beta)}\in \CC,
$$
where the series converge in the operator norm topology for any  ${\bf X}=\{X_{i,j}\}$ in the   regular polyball $ {\bf B_n}(\cH)$ and any Hilbert space $\cH$.  The  results of this paper will play an important role   in the hyperbolic geometry of noncommutative  polyballs, a project in preparation \cite{Po-hyperbolic-polyballs}.
To present our results we need some notation and preliminaries on regular polyballs and their universal models.

Throughout this paper, $B(\cH)$ stands for the algebra of all bounded linear operators on a Hilbert space $\cH$. We denote by  $B(\cH)^{n_1}\times_c\cdots \times_c B(\cH)^{n_k}$, where $n_i \in\NN:=\{1,2,\ldots\}$,
   the set of all tuples  ${\bf X}:=({ X}_1,\ldots, { X}_k)$ in $B(\cH)^{n_1}\times\cdots \times B(\cH)^{n_k}$
     with the property that the entries of ${X}_s:=(X_{s,1},\ldots, X_{s,n_s})$  are commuting with the entries of
      ${X}_t:=(X_{t,1},\ldots, X_{t,n_t})$  for any $s,t\in \{1,\ldots, k\}$, $s\neq t$.
  Note that  the operators $X_{s,1},\ldots, X_{s,n_s}$ are not necessarily commuting.
  Let ${\bf n}:=(n_1,\ldots, n_k)$ and define  the polyball
  $${\bf P_n}(\cH):=[B(\cH)^{n_1}]_1\times_c \cdots \times_c [B(\cH)^{n_k}]_1,
  $$
  where
    $$[B(\cH)^{n}]_1:=\{(X_1,\ldots, X_n)\in B(\cH)^{n}:\ \|  X_1X_1^*+\cdots +X_nX_n^*\|<1\}, \quad n\in \NN.
    $$
    If $A$ is  a positive invertible operator, we write $A>0$. The {\it regular polyball} on the Hilbert space $\cH$  is defined by
$$
{\bf B_n}(\cH):=\left\{ {\bf X}\in {\bf P_n}(\cH) : \ {\bf \Delta_{X}}(I)> 0  \right\},
$$
where
 the {\it defect mapping} ${\bf \Delta_{X}}:B(\cH)\to  B(\cH)$ is given by
$$
{\bf \Delta_{X}}:=\left(id -\Phi_{X_1}\right)\circ \cdots \circ\left(id -\Phi_{ X_k}\right),
$$
 and
$\Phi_{X_i}:B(\cH)\to B(\cH)$  is the completely positive linear map defined by
$$\Phi_{X_i}(Y):=\sum_{j=1}^{n_i}   X_{i,j} Y X_{i,j} ^*, \qquad Y\in B(\cH).
$$
 Note that if $k=1$, then ${\bf B_n}(\cH)$ coincides with the noncommutative unit ball $[B(\cH)^{n_1}]_1$.
   We remark that the scalar representation of
  the the ({\it abstract})
 {\it regular polyball} ${\bf B}_{\bf n}:=\{{\bf B_n}(\cH):\ \cH \text{\ is a Hilbert space} \}$ is
   ${\bf B_n}(\CC)={\bf P_n}(\CC)= (\CC^{n_1})_1\times \cdots \times (\CC^{n_k})_1$.

Let $H_{n_i}$ be
an $n_i$-dimensional complex  Hilbert space with orthonormal basis $e^i_1,\ldots, e^i_{n_i}$.
  We consider the {\it full Fock space}  of $H_{n_i}$ defined by
$F^2(H_{n_i}):=\CC 1 \oplus\bigoplus_{p\geq 1} H_{n_i}^{\otimes p},$
where  $H_{n_i}^{\otimes p}$ is the
(Hilbert) tensor product of $p$ copies of $H_{n_i}$. Let $\FF_{n_i}^+$ be the unital free semigroup on $n_i$ generators
$g_{1}^i,\ldots, g_{n_i}^i$ and the identity $g_{0}^i$.
  Set $e_\alpha^i :=
e^i_{j_1}\otimes \cdots \otimes e^i_{j_p}$ if
$\alpha=g^i_{j_1}\cdots g^i_{j_p}\in \FF_{n_i}^+$
 and $e^i_{g^i_0}:= 1\in \CC$.
  The length of $\alpha\in
\FF_{n_i}^+$ is defined by $|\alpha|:=0$ if $\alpha=g_0^i$  and
$|\alpha|:=p$ if
 $\alpha=g_{j_1}^i\cdots g_{j_p}^i$, where $j_1,\ldots, j_p\in \{1,\ldots, n_i\}$.
 We  define
 the {\it left creation  operator} $S_{i,j}$ acting on the  Fock space $F^2(H_{n_i})$  by setting
$
S_{i,j} e_\alpha^i:=  e^i_{g_j^i \alpha}$, $\alpha\in \FF_{n_i}^+,
$
 and
  the operator  ${\bf S}_{i,j}$ acting on the Hilbert tensor  product
$F^2(H_{n_1})\otimes\cdots\otimes F^2(H_{n_k})$ by setting
$${\bf S}_{i,j}:=\underbrace{I\otimes\cdots\otimes I}_{\text{${i-1}$
times}}\otimes \,S_{i,j}\otimes \underbrace{I\otimes\cdots\otimes
I}_{\text{${k-i}$ times}},
$$
where  $i\in\{1,\ldots,k\}$ and  $j\in\{1,\ldots,n_i\}$.  We denote ${\bf S}:=({\bf S}_1,\ldots, {\bf S}_k)$, where  ${\bf S}_i:=({\bf S}_{i,1},\ldots,{\bf S}_{i,n_i})$, or ${\bf S}:=\{{\bf S}_{i,j}\}$.   The noncommutative Hardy algebra ${\bf F}_{\bf n}^\infty$ (resp.~the polyball algebra $\boldsymbol\cA_{\bf n}$) is the weakly closed (resp.~norm closed) non-selfadjoint  algebra generated by $\{{\bf S}_{i,j}\}$ and the identity.
Similarly, we   define
 the {\it right creation  operator} $R_{i,j}:F^2(H_{n_i})\to
F^2(H_{n_i})$     by setting
 $
R_{i,j} e_\alpha^i:=  e^i_ {\alpha {g_j}}$ for $ \alpha\in \FF_{n_i}^+,
$
 and
 its ampliation  ${\bf R}_{i,j}$ acting on
$F^2(H_{n_1})\otimes\cdots\otimes F^2(H_{n_k})$. The polyball algebra $\boldsymbol\cR_{\bf n}$ is the norm closed non-selfadjoint  algebra generated by $\{{\bf R}_{i,j}\}$ and the identity.

 We proved in \cite{Po-Berezin-poly} (in a more general setting)   that
    ${\bf X}\in B(\cH)^{n_1}\times\cdots \times B(\cH)^{n_k}$
    is a {\it pure} element in the regular polyball  ${\bf B_n}(\cH)^-$, i.e.  $\lim_{q_i\to \infty}\Phi_{X_i}^{q_i}(I)=0$ in the weak operator topology,  if and only if
there is a Hilbert space $\cD$ and a subspace $\cM\subset F^2(H_{n_1})\otimes \cdots \otimes  F^2(H_{n_k})\otimes \cD$ invariant under each  operator ${\bf S}_{i,j}\otimes I$ such that
$X_{i,j}^*=({\bf S}_{i,j}^*\otimes I)|_{\cM^\perp}$, under an appropriate identification of $\cH$ with $\cM^\perp$.
The $k$-tuple ${\bf S}:=({\bf S}_1,\ldots, {\bf S}_k)$, where  ${\bf S}_i:=({\bf S}_{i,1},\ldots,{\bf S}_{i,n_i})$, is  an element  in the
regular polyball $ {\bf B_n}(\otimes_{i=1}^kF^2(H_{n_i}))^-$ and  plays the role of  {\it left universal model} for
   the abstract
  polyball ${\bf B}_{\bf n}^-:=\{{\bf B_n}(\cH)^-:\ \cH \text{\ is a Hilbert space} \}$. The existence of the universal model  will play an important role in this paper, since it will make the connection between noncommutative function theory, operator algebras, and complex function theory in several variables.

In \cite{BH}, Brown and Halmos showed that a bounded linear operator $T$ on the Hardy space $H^2(\DD)$ is a Toeplitz operator if and only if $S^*TS=T$, where $S$ is the unilateral shift. Expanding on this idea, a study of noncommutative multi-Toeplitz operators on the full Fock space with $n$ generators $F^2(H_n)$ was initiated in \cite{Po-multi}, \cite{Po-analytic} and has had an important impact in multivariable operator theory  and the structure of free semigroups algebras (see \cite{DP2}, \cite{DKP}, \cite{DLP}, \cite{Po-entropy}, \cite{Po-pluriharmonic}, \cite{Ken1}, \cite{Ken2}).

In Section 1, we introduce and study the class $\boldsymbol{\cT_{\bf n}}$, ${\bf n}:=(n_1,\ldots, n_k)\in \NN^k$,  of $k$-multi-Toeplitz operators.
 A bounded linear operator $T$ on  the tensor product $F^2(H_{n_1})\otimes \cdots \otimes F^2(H_{n_k})$ of full Fock spaces is called $k$-multi-Toeplitz operator with respect to the {\it right universal model} ${\bf R}=\{{\bf R}_{i,j}\}$
  if
 $$
 {\bf R}_{i,s}^*T {\bf R}_{i,t}=\delta_{st}T,\qquad s,t\in \{1,\ldots, n_i\},
 $$
  for every $i\in\{1,\ldots, k\}$.
  We associate with each $k$-multi-Toeplitz operator  $T$  a   formal power series in several variables and show  that we can recapture $T$  from its noncommutative "Fourier series". Moreover, we characterize the noncommutative formal power series which are Fourier series of $k$-multi-Toeplitz operators (see Theorem \ref{structure-Toeplitz} and Theorem \ref{caract2}). Using these results,  we prove that
  the set of all   $k$-multi-Toeplitz operators  on $\bigotimes_{i=1}^k F^2(H_{n_i})$ coincides with
\begin{equation*}\begin{split}
\text{\rm span} \{\boldsymbol\cA_{\bf n}^* \boldsymbol\cA_{\bf n}\}^{- \text{\rm SOT}}
=\text{\rm span} \{\boldsymbol\cA_{\bf n}^* \boldsymbol\cA_{\bf n}\}^{- \text{\rm WOT}},
\end{split}
\end{equation*}
where $\boldsymbol\cA_{\bf n}$ is the noncommutative polyball algebra.

In Section 2, we characterize the bounded free $k$-pluriharmonic functions on regular polyballs.
 We prove that a function $F:{\bf B_n}(\cH)\to B(\cH)$ is a bounded free $k$-pluriharmonic  function if and only if there is a $k$-multi-Toeplitz operator $A\in \boldsymbol{\cT_{\bf n}}$ such that
$
F({\bf X})=\boldsymbol\cB_{\bf X}[A]$  for $ {\bf X}\in {\bf B_n}(\cH),
$
where $\boldsymbol\cB_{\bf X}$ is the noncommutative Berezin transform at ${\bf X}$ (see Section 1 for definition). In this case, $A=\text{\rm SOT-}\lim_{r\to 1}F(r{\bf S})$ and there is a completely isometric isomorphism of operator spaces
$$
\Phi:{\bf PH}^\infty({\bf B_n})\to \boldsymbol{\cT_{\bf n}},\qquad \Phi(F):=A,
$$
where ${\bf PH}^\infty({\bf B_n})$ is the operator space of all bounded free $k$-pluriharmonic functions on the polyball.

   The  Dirichlet \cite{H}  extension problem  on noncommutative regular polyballs is solved. We show that  a mapping $F:{\bf B_n}(\cH)\to B(\cH)$
is a  free $k$-pluriharmonic  function which has continuous extension  (in the operator norm topology) to
the closed polyball ${\bf B}_{\bf n}(\cH)^-$, and write $F\in {\bf PH}^c({\bf B}_{\bf n})$,  if and only
there exists a $k$-multi-Toeplitz operator $A\in \text{\rm span} \{\boldsymbol\cA_{\bf n}^* \boldsymbol\cA_{\bf n}\}^{-\|\cdot\|}$ such
that $F({\bf X})=\boldsymbol{\cB}_{\bf X}[A]$ for  ${\bf X}\in {\bf B}_{\bf n}(\cH).
  $
In this case, $A=\lim\limits_{r\to 1}F(r{\bf S})$, where
the convergence is in the operator norm, and the map
$$
\Phi:{\bf PH}^c({\bf B}_{\bf n}) \to
\text{\rm span} \{\boldsymbol\cA_{\bf n}^* \boldsymbol\cA_{\bf n}\}^{-\|\cdot\|},\qquad \Phi(F):=A
$$ is a  completely   isometric isomorphism of
operator spaces.

In Section 3, we  provide a Naimark \cite{N} type dilation theorem for  direct products
  ${\bf F}^+_{\bf n}:=\FF_{n_1}^+\times \cdots \times \FF_{n_k}^+$  of free semigroups. We show that a map $K:{\bf F}^+_{\bf n}\times {\bf F}^+_{\bf n}\to B(\cE)$  is a  positive semi-definite
 left $k$-multi-Toeplitz kernel on ${\bf F}^+_{\bf n}$
if and only if
 there exists a
$k$-tuple   of  {\it commuting  row isometries} ${\bf V}=(V_1,\ldots, V_k)$, $V_i=(V_{i,1},\ldots, V_{i,n_i})$, on  a Hilbert space $\cK\supset \cE$, i.e  the non-selfadjoint algebra $Alg(V_i)$ commutes with $Alg(V_s)$ for any $i,s\in \{1,\ldots, k\}$ with $i\neq s$, such that
$$K(\boldsymbol\sigma, \boldsymbol\omega)=P_\cE{\bf V}_{\boldsymbol \sigma}^*{\bf V}_{\boldsymbol \omega}|_\cE,\qquad \boldsymbol \sigma, \boldsymbol \omega\in {\bf F}^+_{\bf n},
 $$
 and
  $\cK=\bigvee_{\boldsymbol \omega\in {\bf F}^+_{\bf n}} {\bf V}_{\boldsymbol \omega}\cE$.
 In this case, the minimal dilation is unique up to an isomorphism. Here, we  use the notation ${\bf V}_{\boldsymbol \sigma}:=V_{1,\sigma_1}\cdots V_{k,\sigma_k}$ if
  ${\boldsymbol \sigma}=(\sigma_1,\ldots, \sigma_k)\in {\bf F}_{\bf n}^+$, and $V_{i,\sigma_i}:=V_{i,j_1}\cdots V_{i,j_p}$
  if  $\sigma_i=g_{j_1}^i\cdots g_{j_p}^i\in \FF_{n_i}^+$ and
   $V_{i,g_0^i}:=I$. For more information on kernels in various noncommutative settings we refer the reader to the work of Ball and Vinnikov \cite{BV1} (see also \cite{BV2} and the references there in).

 We prove a Schur  \cite{Sc} type result which states  that a free $k$-pluriharmonic function  $F$ on the polyball ${\bf B_n}$ is positive  if and only if a certain  right $k$-multi-Toeplitz kernel $\Gamma_{F_r}$  associated with the mapping ${\bf S}\mapsto F(r{\bf S})$ is positive semi-definite for any $r\in [0,1)$. Our Naimark type result for  positive semi-definite
 right $k$-multi-Toeplitz kernels on ${\bf F}^+_{\bf n}$ is used to provide a structure theorem for positive free $k$-pluriharmonic functions.
We show that a free $k$-pluriharmonic function $F:{\bf B_n}(\cH)\to B(\cE)\otimes_{min} B(\cH)$, with $F(0)=I$,   is  positive    if and only if   it has the form
 $$
 F({\bf X})=
 \sum_{(\boldsymbol \alpha, \boldsymbol \beta)\in {\bf F}_{\bf n}^+\times {\bf F}_{\bf n}^+}
 P_\cE{\bf V}_{\widetilde{\boldsymbol\alpha}}^*{\bf V}_{\widetilde{\boldsymbol\beta}}|_\cE\otimes
 {\bf X}_{\boldsymbol \alpha} {\bf X}_{\boldsymbol\beta}^*,\qquad {\bf X}\in {\bf B_n}(\cH),
 $$
 where ${\bf V}=(V_1,\ldots, V_k)$ is a
 $k$-tuple   of commuting row isometries on a space $\cK\supset \cE$,
 and
 $\widetilde{\boldsymbol\alpha}=(\widetilde{\alpha}_1,\ldots, \widetilde{\alpha}_k)$ is the reverse of ${\boldsymbol\alpha}=({\alpha_1},\ldots, {\alpha_k})$, i.e.
  $\widetilde \alpha_i= g^i_{i_k}\cdots g^i_{i_1}$ if
   $\alpha_i=g^i_{i_1}\cdots g^i_{i_k}\in\FF_{n_i}^+$. The general case when $F(0)\geq 0$ is also considered. As a consequence of these results, we obtain   a structure theorem for  positive  $k$-harmonic functions on the regular polydisk included in $[B(\cH)]_1\times_c\cdots \times_c[B(\cH)]_1$,
    which extends the corresponding  classical result in scalar polydisks \cite{Ru1}.

  In Section 4, we define  the {\it free pluriharmonic Poisson kernel}   on the regular polyball  ${\bf B_n}$  by setting
\begin{equation*}
 \boldsymbol{\cP}({\bf R}, {\bf X}):=\sum_{\alpha_1,\beta_1\in \FF_{n_1}^+}\cdots \sum_{\alpha_k, \beta_k\in \FF_{n_k}^+} {\bf R}_{1,\widetilde\alpha_1}^*\cdots {\bf R}_{k,\widetilde\alpha_k}^*{\bf R}_{1,\widetilde\beta_1}\cdots {\bf R}_{k,\widetilde\beta_k}
\otimes { X}_{1,\alpha_1}\cdots {X}_{k,\alpha_k}{X}_{1,\beta_1}^*\cdots {X}_{k,\beta_k}^*
\end{equation*}
for any ${\bf X}\in {\bf B_n}(\cH)$,
where the convergence is in the operator norm topology.
Given a completely  bounded linear  map $\mu:\text{\rm span}\{\boldsymbol\cR_{\bf n}^*\boldsymbol\cR_{\bf n}\}\to B(\cE)$, we  introduce   the {\it noncommutative Poisson transform} of $\mu$ to
be the map \ $\boldsymbol\cP\mu : {\bf B}_{\bf n}(\cH)\to B(\cE)\otimes_{min}B(\cH)$
defined by
$$
(\boldsymbol\cP \mu)({\bf X}):=\widehat \mu[\boldsymbol{\cP}({\bf R}, {\bf X})],\qquad
{\bf X} \in {\bf B}_{\bf n}(\cH),
$$
 where  the  completely bounded linear map
 $$\widehat
\mu:=\mu\otimes \text{\rm id} : \text{\rm span}\{\boldsymbol\cR_{\bf n}^*\boldsymbol\cR_{\bf n}\}^{-\|\cdot \|} \otimes_{min} B(\cH)\to
B(\cE)\otimes_{min} B(\cH)
$$
is uniquely defined by  $ \widehat \mu(A\otimes Y):= \mu(A)\otimes Y$ for any $A\in
\text{\rm span}\{\boldsymbol\cR_{\bf n}^*\boldsymbol\cR_{\bf n}\} $ and  $Y\in B(\cH)$.
We remark that, in  the particular case when $n_1=\cdots=n_k=1$,
$\cH=\cK=\CC$, ${\bf X}=(X_1,\ldots, X_k)\in \DD^k$,  and $\mu$ is a complex
Borel measure on $\TT^k$ (which can be seen as a bounded linear
functional on $C(\TT^k)$),  the noncommutative Poisson transform of $\mu$ coincides with the classical Poisson transform of $\mu$ \cite{Ru1}.

 In Section 4, we give necessary and sufficient conditions
 on a function $F:{\bf B_n}(\cH)\to B(\cE)\otimes_{min} B(\cH)$ to be  the noncommutative Poisson transform of
 a completely  bounded linear map $\mu:C^*({\bf R})\to B(\cE)$,  where $C^*({\bf R})$ is the $C^*$-algebra generated by the operators ${\bf R}_{i,j}$.
 In this case, we show that  there exists a
$k$-tuple   ${\bf V}=(V_1,\ldots, V_k)$, $V_i=(V_{i,1},\ldots, V_{i,n_i})$, of {\it doubly commuting row isometries} acting on  a Hilbert space $\cK$, i.e.
$C^*(V_i)$ commutes with $C^*(V_j)$ if $i\neq j$,   and  bounded linear operators $W_1,W_2:\cE\to \cK$ such that
$$
F({\bf X})=  (W_1^*\otimes I)\left[ C_{\bf X}({\bf V})^*
C_{\bf X}({\bf V}) \right] (W_2\otimes I), \qquad {\bf X}\in {\bf B_n}(\cH),
$$
where $$C_{\bf X}({\bf V}):=(I\otimes \boldsymbol\Delta_{\bf X}(I)^{1/2})\prod_{i=1}^k(I-V_{i,1}\otimes
X_{i,1}^*-\cdots -{V}_{i,n_i}\otimes X_{i,n_i}^*)^{-1}.$$
 In particular, we obtain necessary and sufficient conditions  for a function $F:{\bf B_n}(\cH)\to B(\cE)\otimes_{min} B(\cH)$ to be  the noncommutative Poisson transform of
 a completely  positive  linear map $\mu:C^*({\bf R})\to B(\cE)$.
 In this case, we have the representation
 $$F({\bf X})=  (W^*\otimes I)\left[ C_{\bf X}({\bf V})^*
C_{\bf X}({\bf V}) \right](W\otimes I), \qquad {\bf X}\in {\bf B_n}(\cH).$$

In Section 5,
 we introduce the {\it noncommutative Herglotz-Riesz
transform} of  a completely positive linear map $\mu:\text{\rm span}\{\boldsymbol\cR_{\bf n}^*\boldsymbol\cR_{\bf n}\}\to B(\cE)$   to be the map ${\bf H}\mu: {\bf B}_{\bf n}(\cH) \to B(\cE)\otimes_{min}B(\cH)$  defined by
$$
({\bf H}\mu)({\bf X} ):=\widehat\mu\left[ 2\prod_{i=1}^k(I-{\bf R}_{i,1}^*\otimes X_{i,1}-\cdots
-{\bf R}_{i,n_i}^*\otimes X_{i,n_i})^{-1}-I\right]
$$
for ${\bf X}:=(X_1,\ldots, X_n)\in {\bf B}_{\bf n}(\cH)$.
The main result of this section provides necessary and sufficient conditions for a function $f$ from the polyball ${\bf B_n}(\cH)$ to  $B(\cE)\otimes_{min}B(\cH)$ to admit a Herglotz-Riesz type representation \cite{Her}, \cite{Ri}, i.e.
$$f({\bf X})=({\bf H}\mu)({\bf X} )+i\Im f(0), \qquad {\bf X}\in {\bf B_n}(\cH),
$$
where
  $\mu:C^*({\bf R})\to B(\cE)$ is a completely positive linear map with the property that
    $\mu({\bf R}_{\boldsymbol\alpha}^* {\bf R}_{\boldsymbol\beta})=0$   if \
 ${\bf R}_{\boldsymbol\alpha}^* {\bf R}_{\boldsymbol\beta}$ is not equal to ${\bf R}_{\boldsymbol\gamma}$ or ${\bf R}_{\boldsymbol\gamma}^*$ for some ${\boldsymbol\gamma}\in {\bf F}_{\bf n}^+$.
In this case,  we show that there exist a
$k$-tuple   ${\bf V}=(V_1,\ldots, V_k)$, $V_i=(V_{i,1},\ldots, V_{i,n_i})$, of doubly commuting row isometries on  a Hilbert space $\cK$, and a  bounded linear operator $W:\cE\to \cK$, such that
$$
f({\bf X})=(W^*\otimes I)\left[ 2\prod_{i=1}^k(I-V_{i,1}^*\otimes X_{i,1}-\cdots
-V_{i,n_i}^*\otimes X_{i,n_i})^{-1}-I\right](W\otimes I)+i \Im f(0)
$$
and $W^* {\bf V}_{\boldsymbol\alpha}^*{\bf V}_{\boldsymbol\beta }W=0$   if \
${\bf R}_{\boldsymbol\alpha}^* {\bf R}_{\boldsymbol\beta}$ is not equal to ${\bf R}_{\boldsymbol\gamma}$ or ${\bf R}_{\boldsymbol\gamma}^*$ for some $\boldsymbol\gamma\in {\bf F}_{\bf n}^+$.

We remark that, in the particular  case when $n_1=\cdots =n_k=1$,  we obtain an operator-valued extension of Kor\' annyi-Puk\' anszky  \cite{KP} integral representation  for holomorphic functions with positive real parts in polydisks.

\bigskip

\section{  $k$-multi-Toeplitz operators  on tensor products of full  Fock spaces }

In this section, we introduce the class $\boldsymbol{\cT_{\bf n}}$ of
 $k$-multi-Toeplitz operators on  tensor products of full Fock spaces.
 We associate with each $k$-multi-Toeplitz operator $T$  a   formal power series in several variables and show  that we can recapture $T$ from its noncommutative Fourier series. Moreover, we characterize the noncommutative formal power series which are Fourier series of $k$-multi-Toeplitz operators
    and prove that
  $\boldsymbol{\cT_{\bf n}}=\text{\rm span} \{\boldsymbol\cA_{\bf n}^* \boldsymbol\cA_{\bf n}\}^{- \text{\rm SOT}}$,
 where $\boldsymbol\cA_{\bf n}$ is the noncommutative polyball algebra.

First, we recall (see \cite{Po-poisson}, \cite{Po-Berezin-poly}) some basic properties for  a class of  noncommutative Berezin \cite{Be} type  transforms associated  with regular polyballs.
 Let  ${\bf X}=({ X}_1,\ldots, { X}_k)\in {\bf B_n}(\cH)^-$ with $X_i:=(X_{i,1},\ldots, X_{i,n_i})$.
We  use the notation
$X_{i,\alpha_i}:=X_{i,j_1}\cdots X_{i,j_p}$
  if  $\alpha_i=g_{j_1}^i\cdots g_{j_p}^i\in \FF_{n_i}^+$ and
   $X_{i,g_0^i}:=I$.
The {\it noncommutative Berezin kernel} associated with any element
   ${\bf X}$ in the noncommutative polyball ${\bf B_n}(\cH)^-$ is the operator
   $${\bf K_{X}}: \cH \to F^2(H_{n_1})\otimes \cdots \otimes  F^2(H_{n_k}) \otimes  \overline{{\bf \Delta_{X}}(I) (\cH)}$$
   defined by
   $$
   {\bf K_{X}}h:=\sum_{\beta_i\in \FF_{n_i}^+, i=1,\ldots,k}
   e^1_{\beta_1}\otimes \cdots \otimes  e^k_{\beta_k}\otimes {\bf \Delta_{X}}(I)^{1/2} X_{1,\beta_1}^*\cdots X_{k,\beta_k}^*h, \qquad h\in \cH,
   $$
   where the defect operator $ {\bf \Delta_{X}}(I)$ was defined in the introduction.
A very  important property of the Berezin kernel is that
     ${\bf K_{X}} { X}^*_{i,j}= ({\bf S}_{i,j}^*\otimes I)  {\bf K_{X}}$ for any  $i\in \{1,\ldots, k\}$ and $ j\in \{1,\ldots, n_i\}.
    $
    The {\it Berezin transform at} $X\in {\bf B_n}(\cH)$ is the map $ \boldsymbol{\cB_{\bf X}}: B(\otimes_{i=1}^k F^2(H_{n_i}))\to B(\cH)$
 defined by
\begin{equation*}
 {\boldsymbol\cB_{\bf X}}[g]:= {\bf K^*_{\bf X}} (g\otimes I_\cH) {\bf K_{\bf X}},
 \qquad g\in B(\otimes_{i=1}^k F^2(H_{n_i})).
 \end{equation*}
  If $g$ is in   the $C^*$-algebra $C^*({\bf S})$ generated by ${\bf S}_{i,1},\ldots,{\bf S}_{i,n_i}$, where $i\in \{1,\ldots, k\}$, we  define the Berezin transform at  $X\in {\bf B_n}(\cH)^-$  by
  $${\boldsymbol\cB_{\bf X}}[g]:=\lim_{r\to 1} {\bf K^*_{r\bf X}} (g\otimes I_\cH) {\bf K_{r\bf X}},
 \qquad g\in  C^*({\bf S}),
 $$
 where the limit is in the operator norm topology.
In this case, the Berezin transform at $X$ is a unital  completely positive linear  map such that
 $${\boldsymbol\cB_{\bf X}}({\bf S}_{\boldsymbol\alpha} {\bf S}_{\boldsymbol\beta}^*)={\bf X}_{\boldsymbol\alpha} {\bf X}_{\boldsymbol\beta}^*, \qquad \boldsymbol\alpha, \boldsymbol\beta \in \FF_{n_1}^+\times \cdots \times\FF_{n_k}^+,
 $$
 where  ${\bf S}_{\boldsymbol\alpha}:= {\bf S}_{1,\alpha_1}\cdots {\bf S}_{k,\alpha_k}$ if  $\boldsymbol\alpha:=(\alpha_1,\ldots, \alpha_k)\in \FF_{n_1}^+\times \cdots \times\FF_{n_k}^+$.

The  Berezin transform will play an important role in this paper.
   More properties  concerning  noncommutative Berezin transforms and multivariable operator theory on noncommutative balls and  polydomains, can be found in \cite{Po-poisson},     \cite{Po-Berezin3}, and \cite{Po-Berezin-poly}.
For basic results on completely positive (resp.~bounded)  maps  we refer the reader to \cite{Pa-book}, \cite{Pi-book}, and \cite{ER}.

 \begin{definition} Let $\cE$ be a Hilbert space.
 A bounded linear operator $A\in B(\cE\otimes \bigotimes_{i=1}^k F^2(H_{n_i}))$ is called $k$-multi-Toeplitz with respect to the universal model
 ${\bf R}:=({\bf R}_1,\ldots, {\bf R}_k)$,  where  ${\bf R}_i:=({\bf R}_{i,1},\ldots,{\bf R}_{i,n_i})$, if
 $$
 (I_\cE\otimes {\bf R}_{i,s}^*)A(I_\cE\otimes {\bf R}_{i,t})=\delta_{st}A,\qquad s,t\in \{1,\ldots, n_i\},
 $$
  for every $i\in\{1,\ldots, k\}$.
 \end{definition}

 A few more notations are necessary. If $\omega, \gamma\in \FF_n^+$,
we say that $\gamma<_r \omega$ if there is $\sigma\in
\FF_n^+\backslash\{g_0\}$ such that $\omega=\sigma \gamma$. In this
case,  we set $\omega\backslash_r \gamma:=\sigma$. Similarly, we say
that $\gamma
<_{l}\omega$ if there is $\sigma\in
\FF_n^+\backslash\{g_0\}$ such that $\omega= \gamma \sigma$ and set
$\omega\backslash_l \gamma:=\sigma$.
 We denote by
$\widetilde\alpha$  the reverse of $\alpha\in \FF_n^+$, i.e.
  $\widetilde \alpha= g_{i_k}\cdots g_{i_1}$ if
   $\alpha=g_{i_1}\cdots g_{i_k}\in\FF_n^+$.
  Notice that $\gamma<_r \omega$ if and only if $\widetilde \gamma<_l\widetilde \omega$. In this case we have
 $ \widetilde{\omega\backslash_r\gamma}=\widetilde\omega \backslash_l\widetilde\gamma$.
  We say that $\omega$ is {\it right comparable} with $\gamma$, and write $\omega\sim_{rc} \gamma$, if  either one of the conditions  $\omega<_r \gamma$, $\gamma<_r \omega$, or $\omega=\gamma$ holds.  In this case, we define
  $$
  c_r^+(\omega, \gamma):=
  \begin{cases}\omega\backslash_r \gamma&; \quad \text{ if } \gamma<_r \omega\\
   g_0&; \quad \text{ if } \omega<_r \gamma \ \text{ or } \ \omega=\gamma
   \end{cases}\quad \text{ and } \quad
   c_r^-(\omega, \gamma):=
  \begin{cases}\gamma\backslash_r\omega&; \quad \text{ if } \omega<_r \gamma\\
   g_0&; \quad \text{ if } \gamma<_r \omega \ \text{ or } \ \omega=\gamma.
   \end{cases}
   $$
  Let $\boldsymbol\omega=(\omega_1,\ldots, \omega_k)$ and $\boldsymbol\gamma=(\gamma_1,\ldots, \gamma_k)$ be in $\FF_{n_1}^+\times\cdots \times \FF_{n_k}^+$. We say that $\boldsymbol\omega$ and $\boldsymbol\gamma$ are right comparable, and write  $\boldsymbol\omega\sim_{rc} \boldsymbol\gamma$, if for each $i\in \{1,\ldots, k\}$,  either one of the conditions  $\omega_i<_r \gamma_i$, $\gamma_i<_r \omega_i$, or $\omega_i=\gamma_i$ holds.  In this case, we define
\begin{equation}
\label{c+}
c_r^+(\boldsymbol\omega, \boldsymbol\gamma):=(c_r^+(\omega_1, \gamma_1),\ldots, c_r^+(\omega_k, \gamma_k))\quad \text{ and } \quad
c_r^-(\boldsymbol\omega, \boldsymbol\gamma):=(c_r^-(\omega_1, \gamma_1),\ldots, c_r^-(\omega_k, \gamma_k)).
\end{equation}
Similarly, we say that that $\boldsymbol\omega$ and $\boldsymbol\gamma$ are {\it left comparable}, and write  $\boldsymbol\omega\sim_{lc} \boldsymbol\gamma$, if
$\widetilde{\boldsymbol\omega}\sim_{rc} \widetilde{\boldsymbol\gamma}$.
The definitions of $c_l^+(\boldsymbol\omega, \boldsymbol\gamma)$ and
$c_l^-(\boldsymbol\omega, \boldsymbol\gamma)$ are now clear.
Note that
$$
\widetilde{c_r^+(\boldsymbol\omega, \boldsymbol \gamma)}=
c_l^+(\widetilde{\boldsymbol\omega}, \widetilde{\boldsymbol \gamma})\quad \text{ and } \quad
\widetilde{c_r^-(\boldsymbol\omega, \boldsymbol \gamma)}=
c_l^-(\widetilde{\boldsymbol\omega}, \widetilde{\boldsymbol \gamma}).
$$

For each $m\in \ZZ$, we set $m^+:=\max\{m, 0\}$ and $m^-:=\max\{-m, 0\}$.

 \begin{lemma}\label{inner}
  Let $\boldsymbol \alpha=(\alpha_1,\ldots, \alpha_k)$ and  $\boldsymbol\beta=(\beta_1,\ldots, \beta_k)$ be $k$-tuples in $\FF_{n_1}^+\times\cdots \times \FF_{n_k}^+$ such that
$\alpha_i,\beta_i\in \FF_{n_i}^+, i\in \{1,\ldots, k\} $,   $|\alpha_i|=m_i^-$,  $ |\beta_i|=m_i^+$, and $m_i\in \ZZ$.
If $\boldsymbol\gamma=(\gamma_1,\ldots, \gamma_k)$ and $\boldsymbol \omega=(\omega_1,\ldots, \omega_k)$ are   $k$-tuples in $\FF_{n_1}^+\times\cdots \times \FF_{n_k}^+$, then  the inner product
\begin{equation*}
\left<{\bf S}_{1,\alpha_1}\cdots {\bf S}_{k,\alpha_k}{\bf S}_{1,\beta_1}^*\cdots {\bf S}_{k,\beta_k}^*(e_{\gamma_1}^1\otimes \cdots \otimes e_{\gamma_k}^k), e_{\omega_1}^1\otimes \cdots \otimes e_{\omega_k}^k\right>
\end{equation*}
is different from zero if and only if
$\boldsymbol \omega \sim_{rc} \boldsymbol\gamma $ and $(\alpha_1,\ldots, \alpha_k; \beta_1,\ldots, \beta_k)=(c_r^+(\boldsymbol\omega, \boldsymbol\gamma);c_r^-(\boldsymbol\omega, \boldsymbol\gamma))$.
    \end{lemma}
    \begin{proof}  Under the conditions of the lemma, ${\bf S}_{i,\alpha_i}{\bf S}_{j,\beta_j}^*={\bf S}_{j,\beta_j}^*{\bf S}_{i,\alpha_i}$ for any $i,j\in \{1,\ldots,k\}$,
    $\alpha_i\in \FF_{n_i}^+$, and $\beta_j\in \FF_{n_j}^+$.
     Note that the inner product
is different from zero if and only if $\beta_i\omega_i=\alpha_i\gamma_i$ for any $i\in \{1,\ldots, k\}$. Let  $m_i\in \ZZ$ and assume that  $|\alpha_i|=m_i^->0$. Then $\beta_i=g_0^i$ and, consequently, $\omega_i=\alpha_i\gamma_i$. This shows that $\gamma_i<_r \omega_i$, $c_r^+(\omega_i, \gamma_i)=\alpha_i$, and $c_r^-(\omega_i, \gamma_i)=g_0^i$.
In the case, when  $|\beta_i|=m_i^+>0$, we have $\alpha_i=g_0^i$ and $\beta_i\omega_i=\gamma_i$. Consequently, $\omega_i<_r \gamma_i$,
$c_r^+(\omega_i, \gamma_i)=g_0^i$, and $c_r^-(\omega_i, \gamma_i)=\beta_i$. When $\alpha_i=\beta_i=g_0^i$, we have $\omega_i=\gamma_i$. Therefore, the scalar product above is different from zero if and only if
$\boldsymbol \omega \sim_{rc} \boldsymbol\gamma $ and $(\alpha_1,\ldots, \alpha_k; \beta_1,\ldots, \beta_k)=(c_r^+(\boldsymbol\omega, \boldsymbol\gamma);c_r^-(\boldsymbol\omega, \boldsymbol\gamma))$.
    \end{proof}

If $\beta_i, \gamma_i\in \FF_{n_i}^+$ and, for each $i\in \{1,\ldots, k\}$,   $\beta_i<_\ell \gamma_i$ or $\beta_i=\gamma_i$, we write $\boldsymbol\beta\leq_\ell \boldsymbol\gamma$.

 \begin{lemma}
 \label{ortho}
 Given  a $k$-tuple $\boldsymbol\gamma=(\gamma_1,\ldots, \gamma_k)\in \FF_{n_1}^+\times\cdots \times \FF_{n_k}^+$, the sequence
$$\left\{{\bf S}_{1,\alpha_1}\cdots {\bf S}_{k,\alpha_k}{\bf S}_{1,\beta_1}^*\cdots {\bf S}_{k,\beta_k}^*(e_{\gamma_1}^1\otimes \cdots \otimes e_{\gamma_k}^k)\right\}
$$
consists of orthonormal vectors,
 if
$\alpha_i,\beta_i\in \FF_{n_i}^+, i\in \{1,\ldots, k\} $ with $m_i\in \ZZ$, $|\alpha_i|=m_i^-$,  $ |\beta_i|=m_i^+$, and  $\boldsymbol\beta\leq_\ell \boldsymbol\gamma$.
 \end{lemma}

 \begin{proof}
 First, note that ${\bf S}_{1,\alpha_1}\cdots {\bf S}_{k,\alpha_k}{\bf S}_{1,\beta_1}^*\cdots {\bf S}_{k,\beta_k}^*(e_{\gamma_1}^1\otimes \cdots \otimes e_{\gamma_k}^k)\neq 0$ if and only if $S_{i,\beta_i}^*(e^i_{\gamma_i})\neq 0$ for each $i\in \{1,\ldots, k\}$, which is equivalent to $\beta_i<_\ell \gamma_i$ or $\beta_i=\gamma_i$. Therefore, $\boldsymbol\beta\leq_\ell \boldsymbol\gamma$.

Fix $i\in \{1,\ldots,k\}$ and $\gamma_i\in \FF_{n_i}^+$. We prove that the sequence
$\{S_{i,\alpha_i}S_{i,\beta_i}^* e_{\gamma_i}^i\}$ consists of orthonormal vectors, if $\alpha_i, \beta_i \in \FF_{n_i}^+$ have the following properties:
 \begin{enumerate}
 \item[(i)] if $|\alpha_i|>0$ then $\beta_i=g_0^i$, and if  $|\beta_i|>0$ then $\alpha_i=g_0^i$;
 \item[(ii)]  $\beta_i\leq_\ell \gamma_i$.
 \end{enumerate}
 Indeed, let $(\alpha_i, \beta_i)$ and $(\alpha_i', \beta_i')$  be two distinct pairs with the above-mentioned properties.
First, we consider the case when $g_0^i\neq \beta_i<_\ell \gamma_i$. Then $\alpha_i=g_0^i$ and, consequently,
$S_{i,\alpha_i}S_{i,\beta_i}^* e_{\gamma_i}^i=e^i_{\gamma_i\backslash_\ell \beta_i}$.
Similarly, if
$g_0^i\neq \beta_i'<_\ell \gamma_i$, then $\alpha_i'=g_0^i$ and, consequently,
$S_{i,\alpha_i'}S_{i,\beta_i'}^* e_{\gamma_i}^i=e^i_{\gamma_i\backslash_\ell \beta_i'}$.
Since $(\alpha_i, \beta_i)\neq (\alpha_i', \beta_i')$, we must have $\beta_i\neq \beta_i'$, which implies $e^i_{\gamma_i\backslash_\ell \beta_i}\perp e^i_{\gamma_i\backslash_\ell \beta_i'}$.
On the other hand, if $\beta_i'=g_0^i$, then $\alpha_i'\in \FF_{n_i}^+$ and
$S_{i,\alpha_i'}S_{i,\beta_i'}^* e_{\gamma_i}^i=e_{\alpha_i'}e_{\gamma_i}\perp e^i_{\gamma_i\backslash_\ell \beta_i'}$. Therefore, $S_{i,\alpha_i'}S_{i,\beta_i'}^* e_{\gamma_i}^i\perp S_{i,\alpha_i'}S_{i,\beta_i'}^* e_{\gamma_i}^i$.

The second case is when $\beta_i=g_0^i$. Then $\alpha_i\in \FF_{n_i}^+$ and
$S_{i,\alpha_i}S_{i,\beta_i}^* e_{\gamma_i}^i=e_{\alpha_i}e_{\gamma_i}$. As we saw above,
$S_{i,\alpha_i'}S_{i,\beta_i'}^* e_{\gamma_i}^i$   is equal to either $e_{\alpha_i'}e_{\gamma_i}$ (when $\beta_i'=g_0^i$) or $e^i_{\gamma_i\backslash_\ell \beta_i'}$ (when $g_0^i\neq \beta_i'<_\ell \gamma_i$). In each case, we have  $S_{i,\alpha_i'}S_{i,\beta_i'}^* e_{\gamma_i}^i\perp S_{i,\alpha_i'}S_{i,\beta_i'}^* e_{\gamma_i}^i$, which completes the proof of our assertion.
Using this result one can easily complete the proof of the lemma.
\end{proof}

 We associate with  each $k$-multi-Toeplitz operator
$A\in B(\cE\otimes \bigotimes_{i=1}^k F^2(H_{n_i}))$  a formal power series
$$
\varphi_A({\bf S}):=\sum_{m_1\in \ZZ}\cdots \sum_{m_k\in \ZZ}
\sum_{{\alpha_i,\beta_i\in \FF_{n_i}^+, i\in \{1,\ldots, k\}}\atop{|\alpha_i|=m_i^-, |\beta_i|=m_i^+}} A_{(\alpha_1,\ldots,\alpha_k;\beta_1,\ldots, \beta_k)}\otimes {\bf S}_{1,\alpha_1}\cdots {\bf S}_{k,\alpha_k}{\bf S}_{1,\beta_1}^*\cdots {\bf S}_{k,\beta_k}^*,
$$
where the coefficients are given by
\begin{equation} \label{AA}
\left<A_{(\alpha_1,\ldots,\alpha_k;\beta_1,\ldots, \beta_k)}h,\ell\right>:=\left<A(h\otimes x), \ell\otimes y\right>,\qquad  h,\ell\in \cE,
\end{equation}
and $x:=x_1\otimes \cdots \otimes x_k$, $y=y_1\otimes \cdots \otimes y_k$ with \begin{equation}
\label{xy}
\begin{cases} x_i=e^i_{\beta_i} \text{ and } y_i=1&; \quad \text{if } m_i\geq 0\\
 x_i=1 \text{ and } y_i=e^i_{\alpha_i}&; \quad \text{if } m_i<0
 \end{cases}
 \end{equation}
 for every $i\in \{1,\ldots, k\}$.

The next result shows that a $k$-multi-Toeplitz operator is uniquely determined by is Fourier series.
\begin{theorem}\label{Fourier}
If $A,B$ are $k$-multi-Toeplitz operators on $\cE\otimes \bigotimes_{i=1}^k F^2(H_{n_i})$, then $A=B$ if and only if the corresponding formal  Fourier series $\varphi_A({\bf S})$ and $\varphi_B({\bf S})$ are equal.
Moreover,
$Aq=\varphi_A({\bf S})q$ for any vector-valued polynomial
$$q=\sum_{{\omega_i \in \FF_{n_i}^+, i\in \{1,\ldots, k\}}\atop{|\omega_i|\leq p_i }}h_{(\omega_1,\ldots, \omega_k)}\otimes e_{\omega_1}^1\otimes \cdots \otimes e_{\omega_k}^k,
 $$
 where $h_{(\omega_1,\ldots, \omega_k)}\in \cE$ and $ (p_1,\ldots, p_k)\in \NN^k$.
\end{theorem}
\begin{proof}
   Let $\boldsymbol\omega=(\omega_1,\ldots, \omega_k)$ and $\boldsymbol\gamma=(\gamma_1,\ldots, \gamma_k)$ be $k$-tuples in $\FF_{n_1}^+\times\cdots \times \FF_{n_k}^+$, and let $h, h'\in \cE$. Since $A$ is a $k$-multi-Toeplitz operator on $\cE\otimes \bigotimes_{i=1}^kF^2(H_{n_i})$, we have
 \begin{equation*}
 \begin{split}
 &\left< A(h\otimes e_{\gamma_1}^1\otimes \cdots \otimes e_{\gamma_k}^k), h'\otimes e_{\omega_1}^1\otimes \cdots \otimes e_{\omega_k}^k\right>\\
 &=
 \left< A(I_\cE\otimes {\bf R}_{1,\widetilde\gamma_1}\cdots {\bf R}_{k,\widetilde\gamma_k})(h\otimes 1), (I_\cE\otimes {\bf R}_{1,\widetilde\omega_1}\cdots {\bf R}_{k,\widetilde\omega_k})(h'\otimes 1)\right>\\
 &=\begin{cases} \left<A_{(c_r^+(\boldsymbol\omega, \boldsymbol\gamma); c_r^-(\boldsymbol\omega, \boldsymbol\gamma))}h,h'\right>,&\quad \boldsymbol\omega\sim_{rc} \boldsymbol\gamma\\
 0,& \text{otherwise},
 \end{cases}
 \end{split}
 \end{equation*}
 where $c_r^+(\boldsymbol\omega, \boldsymbol\gamma)$ and  $ c_r^-(\boldsymbol\omega, \boldsymbol\gamma)$ are defined by  relation \eqref{c+}.
 Consequently,
 $$
  A(h\otimes e_{\gamma_1}^1\otimes \cdots \otimes e_{\gamma_k}^k)=
   \sum_{{\boldsymbol\omega=(\omega_1,\ldots, \omega_k)\in \FF_{n_1}^+\times\cdots \times \FF_{n_k}^+}\atop{\boldsymbol\omega\sim_{rc} \boldsymbol\gamma}}
    A_{(c_r^+(\boldsymbol\omega, \boldsymbol\gamma); c_r^-(\boldsymbol\omega, \boldsymbol\gamma))}h\otimes  e_{\omega_1}^1\otimes \cdots \otimes e_{\omega_k}^k
  $$
  is a vector in $\cE\otimes \bigotimes_{i=1}^kF^2(H_{n_i})$. Hence, we deduce that, for each $\boldsymbol\gamma=(\gamma_1,\ldots, \gamma_k)\in \FF_{n_1}^+\times\cdots \times \FF_{n_k}^+$, the series
  \begin{equation}
  \label{WOT}
    \sum_{{\boldsymbol\omega\in \FF_{n_1}^+\times\cdots \times \FF_{n_k}^+}\atop{\boldsymbol\omega\sim_{rc} \boldsymbol\gamma}}
    A_{(c_r^+(\boldsymbol\omega, \boldsymbol\gamma);c_r^-(\boldsymbol\omega, \boldsymbol\gamma))}^*A_{(c_r^+(\boldsymbol\omega, \boldsymbol\gamma);c_r^-(\boldsymbol\omega, \boldsymbol\gamma))}
    \end{equation}
    is WOT-convergent.
Due to Lemma \ref{ortho}, given $\boldsymbol\gamma=(\gamma_1,\ldots, \gamma_k)\in \FF_{n_1}^+\times\cdots \times \FF_{n_k}^+$, the sequence
$\left\{{\bf S}_{1,\alpha_1}\cdots {\bf S}_{k,\alpha_k}{\bf S}_{1,\beta_1}^*\cdots {\bf S}_{k,\beta_k}^*(e_{\gamma_1}^1\otimes \cdots \otimes e_{\gamma_k}^k)\right\}$, where
$\alpha_i,\beta_i\in \FF_{n_i}^+, i\in \{1,\ldots, k\} $ with $m_i\in \ZZ$, $|\alpha_i|=m_i^-$,  $ |\beta_i|=m_i^+$, and  $\boldsymbol\beta\leq_\ell \boldsymbol\gamma$,  consists of orthonormal vectors.
Note that, in this case,  we also have $\boldsymbol\alpha\sim_{rc} \boldsymbol \beta$ and, consequently, $A_{(\alpha_1,\ldots,\alpha_k;\beta_1,\ldots, \beta_k)}=
A_{(c_r^+(\boldsymbol\alpha, \boldsymbol\beta);c_r^-(\boldsymbol\alpha, \boldsymbol\beta))}$. Hence and taking into account that the series \eqref{WOT} is WOT-convergent, we deduce that
\begin{equation*}
\begin{split}
 &\varphi_A({\bf S})(h\otimes e_{\gamma_1}^1\otimes \cdots \otimes e_{\gamma_k}^k)\\
 &:=\sum_{m_1\in \ZZ}\cdots \sum_{m_k\in \ZZ} \sum_{{\alpha_i,\beta_i\in \FF_{n_i}^+, i\in \{1,\ldots, k\}}\atop{|\alpha_i| =m_i^-, |\beta_i|=m_i^+}}A_{(\alpha_1,\ldots,\alpha_k;\beta_1,\ldots, \beta_k)}h
\otimes {\bf S}_{1,\alpha_1}\cdots {\bf S}_{k,\alpha_k}{\bf S}_{1,\beta_1}^*\cdots {\bf S}_{k,\beta_k}^*(e_{\gamma_1}^1\otimes \cdots \otimes e_{\gamma_k}^k)
\end{split}
\end{equation*}
is a convergent series  in $\cE\otimes \bigotimes_{i=1}^kF^2(H_{n_i})$.
Let $\boldsymbol\gamma=(\gamma_1,\ldots, \gamma_k)$ and $\boldsymbol \omega=(\omega_1,\ldots, \omega_k)$ be  $k$-tuples in $\FF_{n_1}^+\times\cdots \times \FF_{n_k}^+$. According to Lemma \ref{inner}, the inner product
\begin{equation*}
\left<{\bf S}_{1,\alpha_1}\cdots {\bf S}_{k,\alpha_k}{\bf S}_{1,\beta_1}^*\cdots {\bf S}_{k,\beta_k}^*(e_{\gamma_1}^1\otimes \cdots \otimes e_{\gamma_k}^k), e_{\omega_1}^1\otimes \cdots \otimes e_{\omega_k}^k\right>
\end{equation*}
is different from zero if and only if
$\boldsymbol \omega \sim_{rc} \boldsymbol\gamma $ and $(\alpha_1,\ldots, \alpha_k, \beta_1,\ldots, \beta_k)=(c_r^+(\boldsymbol\omega, \boldsymbol\gamma);c_r^-(\boldsymbol\omega, \boldsymbol\gamma))$.
Now, using relation \eqref{WOT}, one can see that
\begin{equation*}
\begin{split}
 &\left<\varphi_A({\bf S})(h\otimes e_{\gamma_1}^1\otimes \cdots \otimes e_{\gamma_k}^k), h'\otimes e_{\omega_1}^1\otimes \cdots \otimes e_{\omega_k}^k\right>\\
 &\qquad =
 \sum_{m_1\in \ZZ}\cdots \sum_{m_k\in \ZZ} \sum_{{\alpha_i,\beta_i\in \FF_{n_i}^+, i\in \{1,\ldots, k\}}\atop{|\alpha_i| =m_i^-, |\beta_i|=m_i^+}}\left<A_{(\alpha_1,\ldots,\alpha_k;\beta_1,\ldots, \beta_k)}h,h'\right>\\
 &
 \qquad \times\left<{\bf S}_{1,\alpha_1}\cdots {\bf S}_{k,\alpha_k}{\bf S}_{1,\beta_1}^*\cdots {\bf S}_{k,\beta_k}^*(e_{\gamma_1}^1\otimes \cdots \otimes e_{\gamma_k}^k), e_{\omega_1}^1\otimes \cdots \otimes e_{\omega_k}^k\right>\\
 &\qquad=\begin{cases} \left<A_{(c_r^+(\boldsymbol\omega, \boldsymbol\gamma); c_r^-(\boldsymbol\omega, \boldsymbol\gamma))}h,h'\right>,&\quad \boldsymbol\omega\sim_{rc} \boldsymbol\gamma\\
 0,& \text{otherwise}
 \end{cases}\\
  &\qquad =\left< A(h\otimes e_{\gamma_1}^1\otimes \cdots \otimes e_{\gamma_k}^k), h'\otimes e_{\omega_1}^1\otimes \cdots \otimes e_{\omega_k}^k\right>
\end{split}
\end{equation*}
for any $h,h'\in \cE$ and $\boldsymbol\gamma=(\gamma_1,\ldots, \gamma_k)$ and $\boldsymbol \omega=(\omega_1,\ldots, \omega_k)$   in $\FF_{n_1}^+\times\cdots \times \FF_{n_k}^+$, which shows that
$Aq=\varphi_A({\bf S})q$ for any vector-valued polynomial in $\cE\otimes \bigotimes_{i=1}^k F^2(H_{n_i})$. Therefore, if the formal  Fourier series $\varphi_A({\bf S})$ and $\varphi_B({\bf S})$ are equal, then $A=B$. The proof is complete.
\end{proof}

When  $\cG$ is a Hilbert space, $C_{(\alpha; \beta)}\in B(\cG)$,  and the series
$$\Sigma_1:=\sum_{m\in \ZZ, m<0} \sum_{{\alpha,\beta\in \FF_{n}^+}\atop{|\alpha|=m^-, |\beta|=m^+}}C_{(\alpha; \beta)},\qquad \Sigma_2:=\sum_{m\in \ZZ, m\geq 0} \sum_{{\alpha,\beta\in \FF_{n}^+}\atop{|\alpha|=m^-, |\beta|=m^+}}C_{(\alpha; \beta)}
$$
are convergent in the operator topology,
 we say that  the series
$$
\sum_{m\in \ZZ} \sum_{{\alpha,\beta\in \FF_{n}^+}\atop{|\alpha|=m^-, |\beta|=m^+}}C_{(\alpha; \beta)}:= \Sigma_1 +\Sigma_2
$$
is convergent in the operator topology.
In what follows, we show how we can recapture  the $k$-multi-Toeplitz operators from their Fourier series. Moreover,  we characterize the formal series which are Fourier series of $k$-multi-Toeplitz operators.
Let $\cP$ denote  the set of all vector-valued polynomials in $\cE\otimes \bigotimes_{i=1}^k F^2(H_{n_i})$, i.e. each $p\in\cP$ has the form
$$q=\sum_{{\omega_i \in \FF_{n_i}^+, i\in \{1,\ldots, k\}}\atop{|\omega_i|\leq p_i }}h_{(\omega_1,\ldots, \omega_k)}\otimes e_{\omega_1}^1\otimes \cdots \otimes e_{\omega_k}^k,
 $$
 where $h_{(\omega_1,\ldots, \omega_k)}\in \cE$ and $ (p_1,\ldots, p_k)\in \NN^k$.

\begin{theorem} \label{structure-Toeplitz}
Let
$\{ A_{(\alpha_1,\ldots,\alpha_k;\beta_1,\ldots, \beta_k)}\}$  be a family of operators in $B(\cE)$, where $\alpha_i,\beta_i\in \FF_{n_i}^+$, $|\alpha_i|=m_i^-, |\beta_i|=m_i^+$, $m_i\in \ZZ$, and $i\in \{1,\ldots, k\}$. Then
$$\varphi({\bf S}):= \sum_{m_1\in \ZZ}\cdots \sum_{m_k\in \ZZ} \sum_{{\alpha_i,\beta_i\in \FF_{n_i}^+, i\in \{1,\ldots, k\}}\atop{|\alpha_i|=m_i^-, |\beta_i|=m_i^+}}A_{(\alpha_1,\ldots,\alpha_k;\beta_1,\ldots, \beta_k)}
\otimes {\bf S}_{1,\alpha_1}\cdots {\bf S}_{k,\alpha_k}{\bf S}_{1,\beta_1}^*\cdots {\bf S}_{k,\beta_k}^*
$$
is the formal Fourier series of a $k$-multi-Toeplitz operator on $\cE\otimes \bigotimes_{i=1}^k F^2(H_{n_i})$ if and only if the following conditions are satisfied.
\begin{enumerate}
\item[(i)] For each $\boldsymbol\gamma=(\gamma_1,\ldots, \gamma_k)\in \FF_{n_1}^+\times\cdots \times \FF_{n_k}^+$, the series
    $$
    \sum_{{\boldsymbol\omega\in \FF_{n_1}^+\times\cdots \times \FF_{n_k}^+}\atop{\boldsymbol\omega\sim_{rc} \boldsymbol\gamma}}
    A_{(c_r^+(\boldsymbol\omega, \boldsymbol\gamma);c_r^-(\boldsymbol\omega, \boldsymbol\gamma))}^*A_{(c_r^+(\boldsymbol\omega, \boldsymbol\gamma);c_r^-(\boldsymbol\omega, \boldsymbol\gamma))}
    $$
    is WOT-convergent.
\item[(ii)]    If $\cP$ is the set of all vector-valued polynomials in $\cE\otimes \bigotimes_{i=1}^k F^2(H_{n_i})$, then
    $$\sup_{r\in [0,1)} \sup_{p\in  \cP, \|p\|\leq 1} \|\varphi(r{\bf S})p\|<\infty.$$
\end{enumerate}
Moreover, if there is a $k$-multi-Toeplitz operator $A\in B(\cE\otimes \bigotimes_{i=1}^k F^2(H_{n_i}))$  such $\varphi({\bf S})=\varphi_A({\bf S})$, then the following statements hold:
\begin{enumerate}
\item[(a)]  the series
$$
\varphi(r{\bf S}):= \sum_{m_1\in \ZZ}\cdots \sum_{m_k\in \ZZ} \sum_{{\alpha_i,\beta_i\in \FF_{n_i}^+, i\in \{1,\ldots, k\}}\atop{|\alpha_i|=m_i^-, |\beta_i|=m_i^+}}A_{(\alpha_1,\ldots,\alpha_k;\beta_1,\ldots, \beta_k)}
\otimes r^{\sum_{i=1}^k(|\alpha_i|+|\beta_i|)}
{\bf S}_{1,\alpha_1}\cdots {\bf S}_{k,\alpha_k}{\bf S}_{1,\beta_1}^*\cdots {\bf S}_{k,\beta_k}^*
$$
is convergent in the operator norm topology and its sum, which does not depend on the order of the series, is an operator in
$$
\text{\rm span} \{f^*g:\ f, g\in B(\cE)\otimes_{min} \boldsymbol\cA_{\bf n}\}^{-\|\cdot\|},
$$
where   $\boldsymbol\cA_{\bf n}$ is the polyball algebra.
\item[(b)]
$A=\text{\rm SOT-}\lim_{r\to 1} \varphi(r{\bf S})
$
and
$$
\|A\|=\sup_{r\in [0,1)}\|\varphi(r{\bf S})\|=\lim_{r\to 1}\|\varphi(r{\bf S})\|=
\sup_{q\in \cP, \|q\|\leq 1} \|\varphi({\bf S})q\|.
$$
\end{enumerate}
\end{theorem}
\begin{proof} First, we assume that $A\in B(\cE\otimes \bigotimes_{i=1}^k F^2(H_{n_i}))$  is a $k$-multi-Toeplitz operator and $\varphi({\bf S})=\varphi_A({\bf S})$, where the coefficients $A_{(\alpha_1,\ldots,\alpha_k;\beta_1,\ldots, \beta_k)}$ are given by relations \eqref{AA} and \eqref{xy}. Note that item (i) follows from the proof of Theorem \ref{Fourier}. Moreover, from the same proof and Lemma \ref{ortho} we have
$\varphi_A({\bf S})(h\otimes e_{\gamma_1}^1\otimes \cdots \otimes e_{\gamma_k}^k)$ and, consequently,
$\varphi_A(r{\bf S})(h\otimes e_{\gamma_1}^1\otimes \cdots \otimes e_{\gamma_k}^k)$, $r\in [0,1)$, are vectors in  $\cE\otimes \bigotimes_{i=1}^k F^2(H_{n_i})$ and
\begin{equation}
\begin{split}
\label{fifia}
\lim_{r\to 1}\varphi_A(r{\bf S})(h\otimes e_{\gamma_1}^1\otimes \cdots \otimes e_{\gamma_k}^k)
&=\varphi_A({\bf S})(h\otimes e_{\gamma_1}^1\otimes \cdots \otimes e_{\gamma_k}^k)\\
&=A(h\otimes e_{\gamma_1}^1\otimes \cdots \otimes e_{\gamma_k}^k)
\end{split}
\end{equation}
for any $\boldsymbol\gamma=(\gamma_1,\ldots, \gamma_k)$    in $\FF_{n_1}^+\times\cdots \times \FF_{n_k}^+$ and $h\in \cE$. Note also that, due to item (i) and Lemma \ref{ortho}, we also have
$\varphi_A(r_1{\bf S}_1, \ldots, r_k{\bf S}_k)(h\otimes e_{\gamma_1}^1\otimes \cdots \otimes e_{\gamma_k}^k)\in \cE\otimes \bigotimes_{i=1}^k F^2(H_{n_i})$ for any $r_i\in [0,1)$, $i\in \{1,\ldots, k\}$.
Now, we show that the series
$$
 \sum_{m_1\in \ZZ}\cdots \sum_{m_k\in \ZZ} \sum_{{\alpha_i,\beta_i\in \FF_{n_i}^+, i\in \{1,\ldots, k\}}\atop{|\alpha_i|=m_i^-, |\beta_i|=m_i^+}}A_{(\alpha_1,\ldots,\alpha_k;\beta_1,\ldots, \beta_k)}
\otimes \left(\prod_{i=1}^k r_i^{|\alpha_i|+|\beta_i|}\right)
{\bf S}_{1,\alpha_1}\cdots {\bf S}_{k,\alpha_k}{\bf S}_{1,\beta_1}^*\cdots {\bf S}_{k,\beta_k}^*
$$
is convergent in the operator norm topology and its sum is in
$
\text{\rm span} \{f^*g:\ f, g\in B(\cE)\otimes_{min} \boldsymbol\cA_{\bf n}\}^{-\|\cdot\|},
$
where   $\boldsymbol\cA_{\bf n}$ is the polyball algebra.
We denote the series above by $\varphi_A(r_1{\bf S}_1, \ldots, r_k{\bf S}_k)$.
Since  $A$ is a $k$-multi-Toeplitz operator, it is also  a $1$-multi-Toeplitz operator  with respect to $R_k:=(R_{k,1},\ldots, R_{k,n_k})$, the right creation operators on the Fock space $F^2(H_{n_k})$. Applying Theorem \ref{Fourier} to $1$-multi-Toeplitz operators, we deduce  that $A$ has a unique Fourier representation
\begin{equation*}
\psi_A(S_{k} ):=\sum_{m_k\in \ZZ}
 \sum_{{\alpha_k,\beta_k\in \FF_{n_k}^+}\atop{|\alpha_k|=m_k^-, |\beta_k|=m_k^+}}C_{(\alpha_k ; \beta_k)}\otimes S_{k,\alpha_k} S_{k,\beta_k}^*,
\end{equation*}
where
$C_{(\alpha_k ; \beta_k)}\in B(\cE\otimes \bigotimes_{i=1}^{k-1} F^2(H_{n_i}))$.
Moreover, we can prove that, for any $r_k\in [0,1)$,
\begin{equation} \label{F2}
\psi_A(r_kS_{k} ):=\sum_{m_k\in \ZZ}
 \sum_{{\alpha_k,\beta_k\in \FF_{n_k}^+}\atop{|\alpha_k|=m_k^-, |\beta_k|=m_k^+}}r_k^{|\alpha_k|+|\beta_k|}C_{(\alpha_k ; \beta_k)}\otimes S_{k,\alpha_k} S_{k,\beta_k}^*
\end{equation}
is convergent in the operator norm topology. Indeed, since $\psi_A(S_{k} )$ is the Fourier representation of the $1$-multi-Toeplitz operator $A$  with respect to $R_k:=(R_{k,1},\ldots, R_{k,n_k})$, item (i) implies, in the particular case when $\gamma_k=g_0^k$, that
$\sum_{\alpha_k\in \FF_{n_k}^+}C_{(\alpha_k ; g_0^k)}^*C_{(\alpha_k ; g_0^k)}$ is WOT-convergent. Since $A^*$ is also a $1$-multi-Toeplitz operator  we can similarly deduce that the series $\sum_{\beta_k\in \FF_{n_k}^+}C_{(g_0^k;\beta_k) }C_{(g_0^k;\beta_k)}^*$ is WOT-convergent.
Since $S_{k,1},\ldots, S_{k,n_k}$ are isometries with orthogonal ranges  we have
$$
\left\|\sum_{\alpha_k\in \FF_{n_k}^+, |\alpha_k|=m}C_{(\alpha_k ; g_0^k)}\otimes r_k^{|\alpha_k|}S_{k,\alpha_k}\right\|
=r_k^m\left\|\sum_{\alpha_k\in \FF_{n_k}^+}C_{(\alpha_k ; g_0^k)}^*C_{(\alpha_k ; g_0^k)}\right\|^{1/2}
$$
 and
$$
\left\|\sum_{\beta_k\in \FF_{n_k}^+, |\beta_k|=m}C_{( g_0^k;\beta_k)}\otimes r_k^{|\alpha_k|}S_{k,\beta_k}^*\right\|
=r_k^m\left\|\sum_{\alpha_k\in \FF_{n_k}^+}C_{(g_0^k;\beta_k) }C_{(g_0^k;\beta_k)}^*\right\|^{1/2}
$$
for any $m\in \NN$.
Now, it is clear that the series defining
$\psi_A(r_kS_{k} )$ is convergent in the operator norm topology and, consequently, $\psi_A(r_kS_{k} )$  belongs  to
$$
\text{\rm span} \{f^*g:\ f, g\in B(\cE\otimes \bigotimes_{i=1}^{k-1} F^2(H_{n_i}))\otimes_{min} \cA_{ n_k}\}^{-\|\cdot\|},
$$
where $\cA_{n_k}$ is the noncommutative disc algebra generated by $S_{k,1},\ldots, S_{k,n_k}$, and the  identity.
For each $i\in \{1,\ldots, k\}$,  we set $\cE_i:=\cE\otimes F^2(H_{n_1})\otimes \cdots \otimes F^2(H_{n_i})$.

The next step in our proof is to show that
\begin{equation}
\label{B-k}
\psi_A(r_k S_k)=\cB_{r_k S_k}^{ext}[A]:=(I_{\cE_{k-1}}\otimes K_{r_kS_k}^*)(A\otimes I_{F^2(H_{n_k})})
(I_{\cE_{k-1}}\otimes K_{r_kS_k}),
\end{equation}
where
$K_{r_kS_k}:F^2(H_{n_k})\to F^2(H_{n_k})\otimes \cD_{r_kS_k}\subset F^2(H_{n_k})\otimes F^2(H_{n_k})$
is the noncommutative Berezin kernel defined by
$$
K_{r_k S_k}\xi:=\sum_{\beta_k \in \FF_{n_k}^+}e_{\beta_k}^k\otimes \Delta_{r_k S_k}(I)^{1/2} S_{k, \beta_k}^*\xi,\qquad \xi\in F^2(H_{n_k}),
$$
and $\cD_{r_kS_k}:=\overline{\Delta_{r_k S_k}(I)(F^2(H_{n_k}))}$.
Let $\boldsymbol\gamma=(\gamma_1,\ldots, \gamma_k)$ and $\boldsymbol \omega=(\omega_1,\ldots, \omega_k)$ be  $k$-tuples in $\FF_{n_1}^+\times\cdots \times \FF_{n_k}^+$, set $q:=\max\{|\gamma_k|, |\omega_k|\}$, and define the operator
$$
Q_q:=\sum_{m_k\in \ZZ, |m_k|\leq q}\sum_{{\alpha_k,\beta_k\in \FF_{n_k}^+}\atop{|\alpha_k|=m_k^-, |\beta_k|=m_k^+}}C_{(\alpha_k ; \beta_k)}\otimes S_{k,\alpha_k} S_{k,\beta_k}^*.
$$
Since $\psi_A(S_k)p=Ap$ for any polynomial $p\in \cP$, a careful computation reveals that
\begin{equation*}
\begin{split}
&\left<\cB_{r_k S_k}^{ext}[A](h\otimes e_{\gamma_1}^1\otimes \cdots \otimes e_{\gamma_k}^k), h'\otimes e_{\omega_1}^1\otimes \cdots \otimes e_{\omega_k}^k\right>\\
&=
\left<(A\otimes I_{F^2(H_{n_k})})
(h\otimes e_{\gamma_1}^1\otimes \cdots \otimes e_{\gamma_{k-1}}^{k-1}\otimes K_{r_kS_k}(e_{\gamma_k}^k)), h'\otimes e_{\omega_1}^1\otimes \cdots \otimes e_{\omega_{k-1}}^{k-1}\otimes K_{r_k S_k}(e_{\omega_k}^k)\right>\\
&=
\left<(A\otimes I_{F^2(H_{n_k})})
\left( \sum_{\alpha_k\in \FF_{n_k}^+}h\otimes e_{\gamma_1}^1\otimes \cdots \otimes  e_{\gamma_{k-1}}^{k-1}\otimes e_{\alpha_k}^k\otimes \Delta_{r_kS_k}(I)^{1/2}S_{k, \alpha_k}^*(e_{\gamma_k}^k)\right)\right.,\\
&\qquad \qquad \qquad \qquad
\left.
\sum_{\beta_k\in \FF_{n_k}^+} h'\otimes e_{\omega_1}^1\otimes \cdots \otimes e_{\omega_{k-1}}^{k-1}\otimes e_{\beta_k}^k\otimes \Delta_{r_k S_k}(I)^{1/2}S_{k, \beta_k}^*(e_{\omega_k}^k)\right>\\
&=
\left<
 \sum_{\alpha_k\in \FF_{n_k}^+}A(h\otimes e_{\gamma_1}^1\otimes \cdots \otimes  e_{\gamma_{k-1}}^{k-1}\otimes e_{\alpha_k}^k)\otimes \Delta_{r_kS_k}(I)^{1/2}S_{k, \alpha_k}^*(e_{\gamma_k}^k)\right.,\\
&\qquad \qquad \qquad \qquad
\left.
\sum_{\beta_k\in \FF_{n_k}^+} h'\otimes e_{\omega_1}^1\otimes \cdots \otimes e_{\omega_{k-1}}^{k-1}\otimes e_{\beta_k}^k\otimes \Delta_{r_k S_k}(I)^{1/2}S_{k, \beta_k}^*(e_{\omega_k}^k)\right>\\
&=
\sum_{m=0}^\infty \sum_{p=0}^\infty \sum_{\alpha_k\in \FF_{n_k}^+, |\alpha_k|=m} \sum_{\beta_k\in \FF_{n_k}^+, |\beta_k|=p}
\left<
A(h\otimes e_{\gamma_1}^1\otimes \cdots \otimes  e_{\gamma_{k-1}}^{k-1}\otimes e_{\alpha_k}^k),
 h'\otimes e_{\omega_1}^1\otimes \cdots \otimes e_{\omega_{k-1}}^{k-1}\otimes e_{\beta_k}^k\right>\\
 &
\qquad \qquad \qquad \qquad \times \left<\Delta_{r_kS_k}(I)^{1/2}S_{k, \alpha_k}^*(e_{\gamma_k}^k),
  \Delta_{r_k S_k}(I)^{1/2}S_{k, \beta_k}^*(e_{\omega_k}^k)\right>\\
  &=
  \sum_{m=0}^q \sum_{p=0}^q \sum_{\alpha_k\in \FF_{n_k}^+, |\alpha_k|=m} \sum_{\beta_k\in \FF_{n_k}^+, |\beta_k|=p}
\left<
Q_q(h\otimes e_{\gamma_1}^1\otimes \cdots \otimes  e_{\gamma_{k-1}}^{k-1}\otimes e_{\alpha_k}^k),
 h'\otimes e_{\omega_1}^1\otimes \cdots \otimes e_{\omega_{k-1}}^{k-1}\otimes e_{\beta_k}^k\right>\\
 &
\qquad \qquad \qquad \qquad \times \left<\Delta_{r_kS_k}(I)^{1/2}S_{k, \alpha_k}^*(e_{\gamma_k}^k),
  \Delta_{r_k S_k}(I)^{1/2}S_{k, \beta_k}^*(e_{\omega_k}^k)\right>\\
  &=
  \sum_{m=0}^\infty \sum_{p=0}^\infty \sum_{\alpha_k\in \FF_{n_k}^+, |\alpha_k|=m} \sum_{\beta_k\in \FF_{n_k}^+, |\beta_k|=p}
\left<
Q_q(h\otimes e_{\gamma_1}^1\otimes \cdots \otimes  e_{\gamma_{k-1}}^{k-1}\otimes e_{\alpha_k}^k),
 h'\otimes e_{\omega_1}^1\otimes \cdots \otimes e_{\omega_{k-1}}^{k-1}\otimes e_{\beta_k}^k\right>\\
 &
\qquad \qquad \qquad \qquad \times \left<\Delta_{r_kS_k}(I)^{1/2}S_{k, \alpha_k}^*(e_{\gamma_k}^k),
  \Delta_{r_k S_k}(I)^{1/2}S_{k, \beta_k}^*(e_{\omega_k}^k)\right>\\
  &=
  \left<(Q_q\otimes I_{F^2(H_{n_k})})
(h\otimes e_{\gamma_1}^1\otimes \cdots \otimes  e_{\gamma_{k-1}}^{k-1}\otimes K_{r_kS_k}(e_{\gamma_k}^k)), h'\otimes e_{\omega_1}^1\otimes \cdots \otimes e_{\omega_{k-1}}^{k-1}\otimes K_{r_k S_k}(e_{\omega_k}^k)\right>\\
&=
\left<\cB_{r_k S_k}^{ext}[Q_q](h\otimes e_{\gamma_1}^1\otimes \cdots \otimes e_{\gamma_k}^k), h'\otimes e_{\omega_1}^1\otimes \cdots \otimes e_{\omega_k}^k\right>\\
&=
\sum_{m_k\in \ZZ, |m_k|\leq q}\sum_{{\alpha_k,\beta_k\in \FF_{n_k}^+}\atop{|\alpha_k|=m_k^-, |\beta_k|=m_k^+}}
\left<\left(C_{(\alpha_k ; \beta_k)}\otimes r_k^{|\alpha_k|+|\beta_k|}S_{k,\alpha_k} S_{k,\beta_k}^*\right)(h\otimes e_{\gamma_1}^1\otimes \cdots \otimes e_{\gamma_k}^k), h'\otimes e_{\omega_1}^1\otimes \cdots \otimes e_{\omega_k}^k\right>\\
&=
\left<\psi_A(r_k S_k)(h\otimes e_{\gamma_1}^1\otimes \cdots \otimes e_{\gamma_k}^k), h'\otimes e_{\omega_1}^1\otimes \cdots \otimes e_{\omega_k}^k\right>
  \end{split}
\end{equation*}
for any $h,h'\in \cE$.
Consequently, relation \eqref{B-k} holds for any $r_k\in [0,1)$. Hence and using the fact  the noncommutative Berezin kernel $K_{rS_k}$ is an isometry we deduce that
\begin{equation*}
\|\psi_A(r_k S_k)\|\leq \|A\|,\qquad r_k\in [0,1).
\end{equation*}
Moreover, one can show that
$$
A=\text{\rm SOT-}\lim_{r_k\to 1} \psi_A(r_k S_k).
$$
Indeed,  due to part (i) (for 1-multi-Toeplitz operators), we have
$\|\psi_A(r_k S_k)p-\psi_A( S_k)p\|\to 0$ as $r_k\to 1$, for any polynomial $p\in \cE_{k-1}\otimes F^2(H_{n_k})$ with coefficients in $\cE_{k-1}$.
Since $\psi_A(S_k)p=Ap$ and $\|\psi_A(r_k S_k)\|\leq \|A\|$ for any $ r_k\in [0,1)$, an approximation argument proves our assertion.

Now, we prove that the coefficients  $C_{(\alpha_k ; \beta_k)}\in B(\cE\otimes \bigotimes_{i=1}^{k-1} F^2(H_{n_i}))$ of  the Fourier series
$\psi_A(S_{k})$ are $1$-multi-Toeplitz operators with respect to $R_{k-1}:=(R_{k-1,1},\ldots, R_{k-1,n_{k-1}})$.
For each $i\in \{1,\ldots, k-1\}$, $s,t\in \{1,\ldots, n_i\}$, and
any vector-valued polynomial $p\in \cE\otimes \bigotimes_{i=1}^k F^2(H_{n_i})$ with coefficients in $\cE$, Theorem \ref{Fourier} implies
\begin{equation*}
\begin{split}
\sum_{m_k\in \ZZ}\sum_{{\alpha_k,\beta_k\in \FF_{n_k}^+}\atop{|\alpha_k|=m_k^-, |\beta_k|=m_k^+}}&\left[ (I_{\cE_{k-2}}\otimes R_{i,s}^*)C_{(\alpha_k ; \beta_k)}(I_{\cE_{k-2}}\otimes R_{i,t})\otimes S_{k,\alpha_k} S_{k,\beta_k}^*\right](p)\\
&=
(I_\cE\otimes {\bf R}_{i,s}^*)\psi_A(S_k) (I_\cE\otimes {\bf R}_{i,t})(p)
=
(I_\cE\otimes {\bf R}_{i,s}^*) A(I_\cE\otimes {\bf R}_{i,t})(p)\\
&=\delta_{st}A(p)=\delta_{st} \psi_A(S_k) (p)\\
&=
\delta_{st} \sum_{m_k\in \ZZ}\sum_{{\alpha_k,\beta_k\in \FF_{n_k}^+}\atop{|\alpha_k|=m_k^-, |\beta_k|=m_k^+}}(C_{(\alpha_k ; \beta_k)}\otimes S_{k,\alpha_k} S_{k,\beta_k}^*)(p).
\end{split}
\end{equation*}
Hence, we deduce that
$$
(I_{\cE_{k-2}}\otimes R_{i,s}^*)C_{(\alpha_k ; \beta_k)}(I_{\cE_{k-2}}\otimes R_{i,t})=\delta_{st}C_{(\alpha_k ; \beta_k)}
$$
for any $i\in \{1,\ldots, k-1\}$ and  $s,t\in \{1,\ldots, n_i\}$, which proves that $C_{(\alpha_k ; \beta_k)}$ is a $1$-multi-Toeplitz operator with respect to $R_{k-1}:=(R_{k-1,1},\ldots, R_{k-1,n_{k-1}})$.
Consequently, similarly to the first part of the proof, $C_{(\alpha_k ; \beta_k)}$ has a Fourier representation
\begin{equation}
\label{F3}
\psi_{(\alpha_k ; \beta_k)}(S_{k-1} ):=\sum_{m_{k-1}\in \ZZ}
 \sum_{{\alpha_{k-1},\beta_{k-1}\in \FF_{n_{k-1}}^+}\atop{|\alpha_{k-1}|=m_{k-1}^-, |\beta_{k-1}|=m_{k-1}^+}}C_{(\alpha_{k-1},\alpha_k ; \beta_{k-1},\beta_k)}\otimes S_{k-1,\alpha_{k-1}} S_{k-1,\beta_{k-1}}^*,
\end{equation}
where
$C_{(\alpha_{k-1},\alpha_k ; \beta_{k-1},\beta_k)}\in B(\cE_{k-2})$.
Moreover, as above, one can prove that, for any $r_{k-1}\in [0,1)$, the series
$\psi_{(\alpha_k ; \beta_k)}(r_{k-1}S_{k-1} )$ is convergent in the operator norm topology, and its limit  is an element in
$$
\text{\rm span} \{f^*g:\ f, g\in B(\cE_{k-2})\otimes_{min} \cA_{ n_{k-1}}\}^{-\|\cdot\|},
$$
where $ \cA_{ n_{k-1}}$ is the noncommutative disc algebra generated by $S_{k-1,1},\ldots, S_{k-1, n_{k-1}}$ and the identity.
We also have
$$\lim_{r_{k-1}\to 1} \psi_{(\alpha_{k} ; \beta_{k})}(r_{k-1}S_{k-1} )p=C_{( \alpha_k ;  \beta_k)}p$$
 for any vector-valued  polynomial $p\in \cE_{k-2}\otimes F^2(H_{n_{k-1}})$.
 As in the first part of the proof, setting
 $$
 \cB_{r_{k-1} S_{k-1}}^{ext}[u]:=(I_{\cE_{k-2}}\otimes K_{r_{k-1}S_{k-1}}^*)(u\otimes I_{F^2(H_{n_{k-1}})})
(I_{\cE_{k-2}}\otimes K_{r_{k-1}S_{k-1}}), \qquad u\in B(\cE_{n-1}),
$$
one can prove that
\begin{equation}
\label{ext}
 \psi_{(\alpha_{k} ; \beta_{k})}(r_{k-1}S_{k-1} )= \cB_{r_{k-1} S_{k-1}}^{ext}[C_{(\alpha_k ;  \beta_k)}]\quad \text{ and } \quad
\|\psi_{(\alpha_{k},\beta_{k})}(r_{k-1}S_{k-1} )||\leq \|C_{(\alpha_k ;  \beta_k)}\|
\end{equation}
 for any $r_{k-1}\in [0,1).
$
Moreover, we can also show that
$$
C_{(\alpha_k ;  \beta_k)}=\text{\rm SOT-}\lim_{r_{k-1}\to 1} \psi_{(\alpha_{k} ; \beta_{k})}(r_{k-1}S_{k-1} ).
$$
Now, due to relations  \eqref{F2}, \eqref{B-k}, \eqref{F3}, and \eqref{ext},  we obtain
\begin{equation*}
\begin{split}
&\left(\left[\cB_{r_{k-1} S_{k-1}}^{ext}\otimes id_{B(F^2(H_{n_k}))}\right]
\circ \cB_{r_{k} S_{k}}^{ext}\right)[A]\\
&\qquad = \sum_{m_k\in \ZZ}\sum_{{\alpha_k,\beta_k\in \FF_{n_k}^+}\atop{|\alpha_k|=m_k^-, |\beta_k|=m_k^+}}\cB_{r_{k-1} S_{k-1}}^{ext}\left[C_{(\alpha_k ;  \beta_k)}\right]\otimes r_k^{|\alpha_k|+|\beta_k}S_{k,\alpha_k} S_{k,\beta_k}^*\\
&\qquad
=\sum_{m_k\in \ZZ}\sum_{m_{k-1}\in \ZZ}\sum_{{\alpha_k,\beta_k\in \FF_{n_k}^+}\atop{|\alpha_k|=m_k^-, |\beta_k|=m_k^+}}
\sum_{{\alpha_{k-1},\beta_{k-1}\in \FF_{n_{k-1}}^+}\atop{|\alpha_{k-1}|=m_{k-1}^-, |\beta_{k-1}|=m_{k-1}^+}}\\
&\qquad \qquad \qquad
r_k^{|m_k|} r_{k-1}^{|m_{k-1}|}
C_{(\alpha_{k-1},\alpha_k ;  \beta_{k-1},\beta_k)}\otimes S_{k-1,\alpha_{k-1}} S_{k-1,\beta_{k-1}}^*\otimes S_{k,\alpha_{k}} S_{k,\beta_{k}}^*,
\end{split}
\end{equation*}
where the series are convergent in the operator norm topology.
Continuing this process, one can prove that there are some operators $C_{(\alpha_1,\ldots, \alpha_k; \beta_1,\ldots, \beta_k)}\in B(\cE)$ such that
the series $\varphi(r_1{\bf S}_1,\ldots, r_k{\bf S}_k)$ given by
\begin{equation*}
 \sum_{m_k\in \ZZ}\cdots \sum_{m_1\in \ZZ} \sum_{{\alpha_i,\beta_i\in \FF_{n_i}^+, i\in \{1,\ldots, k\}}\atop{|\alpha_i|=m_i^-, |\beta_i|=m_i^+}}r_k^{|m_k|}\cdots r_1^{|m_1|}C_{(\alpha_1,\ldots,\alpha_k;\beta_1,\ldots, \beta_k)}
\otimes
{\bf S}_{1,\alpha_1}\cdots {\bf S}_{k,\alpha_k}{\bf S}_{1,\beta_1}^*\cdots {\bf S}_{k,\beta_k}^*
 \end{equation*}
is convergent in the operator norm topology and
\begin{equation}
\label{BBBBB}
\begin{split}
\varphi(r_1{\bf S}_1,\ldots, r_k {\bf S}_k)=
\left[\cB_{r_{1} S_{1}}^{ext}\otimes id_{B(\otimes_{i=2}^k F^2(H_{n_i}))}\right]\circ \left[\cB_{r_{2} S_{2}}^{ext}\otimes id_{B(\otimes_{i=3}^k F^2(H_{n_i}))}\right]\circ\cdots \circ \cB_{r_{k} S_{k}}^{ext} [A].
\end{split}
\end{equation}
Since the noncommutative Berezin kernels $K_{r_iS_i}$, $i\in \{1,\ldots, k\}$, are isometries, we deduce that
$$\|\varphi(r_1{\bf S}_1,\ldots, r_k{\bf S}_k)\|\leq \|A\|, \qquad r_i\in [0,1).
  $$
  Note that the coefficients of the $k$-multi-Toeplitz operator $ \varphi(r_1{\bf S}_1,\ldots, r_k{\bf S}_k)$ satisfy relation
\begin{equation}
\label{rrr}
\left< r_k^{|m_k|}\cdots r_1^{|m_1|}C_{(\alpha_1,\ldots,\alpha_k;\beta_1,\ldots, \beta_k)}h,\ell\right>=\left<\varphi(r_1{\bf S}_1,\ldots, r_k{\bf S}_k)(h\otimes x), (\ell\otimes y)\right>,
\end{equation}
where $x,y$ are defined  as in relation \eqref{xy}.
Since $A$ is a $k$-multi-Toeplitz operator, so is $Y_{r_k}:=\cB_{r_{k} S_{k}}^{ext} [A]=\psi_A(r_k S_k)$ and, iterating the argument, we deduce that
$$Y_{r_2,\ldots, r_k}:=   \left[\cB_{r_{2} S_{2}}^{ext}\otimes id_{B(\otimes_{i=3}^k F^2(H_{n_i}))}\right]\circ\cdots \circ \cB_{r_{k} S_{k}}^{ext} [A]
$$
is a $k$-multi-Toeplitz operator. In particular, $Y_{r_2,\ldots, r_k}$ is a 1-multi-Toeplitz operator with respect to $R_1:=(R_{1,1},\ldots, R_{1,n_1})$. Applying the first part of the proof to $Y_{r_2,\ldots, r_k}$, we deduce that
$$\text{\rm SOT-}\lim_{r_1\to 1}\left[\cB_{r_{1} S_{1}}^{ext}\otimes id_{B(\otimes_{i=2}^k F^2(H_{n_i}))}\right][Y_{r_2,\ldots, r_k}]=Y_{r_2,\ldots, r_k}.
$$
 Continuing this process, we obtain
$$
\text{\rm SOT-}\lim_{r_k\to 1}\cdots \text{\rm SOT-}\lim_{r_1\to 1} \left[\cB_{r_{1} S_{1}}^{ext}\otimes id_{B(\otimes_{i=2}^k F^2(H_{n_i}))}\right][Y_{r_2,\ldots, r_k}]=A.
$$
Consequently, using   relations \eqref{BBBBB}, \eqref{rrr} and  \eqref{AA}, we deduce that
$$
\left< C_{(\alpha_1,\ldots,\alpha_k;\beta_1,\ldots, \beta_k)}h,\ell\right>
=\left<A(h\otimes x, \ell\otimes y\right>
=\left< A_{(\alpha_1,\ldots, \alpha_k,\beta_1,\ldots, \beta_k)}h,\ell\right>,
$$
which shows that $\varphi_A(r_1{\bf S}_1,\ldots, r_k{\bf S}_k)=\varphi(r_1{\bf S}_1,\ldots, r_k{\bf S}_k)$ for any $r_i\in [0,1)$. Hence,   we obtain
\begin{equation*}
\begin{split}
&\varphi_A(r_1{\bf S}_1,\ldots, r_k{\bf S}_k)\\
&=
\sum_{m_k\in \ZZ}\cdots \sum_{m_1\in \ZZ} \sum_{{\alpha_i,\beta_i\in \FF_{n_i}^+, i\in \{1,\ldots, k\}}\atop{|\alpha_i|=m_i^-, |\beta_i|=m_i^+}}r_k^{|m_k|}\cdots r_1^{|m_1|}A_{(\alpha_1,\ldots,\alpha_k;\beta_1,\ldots, \beta_k)}
\otimes
{\bf S}_{1,\alpha_1}\cdots {\bf S}_{k,\alpha_k}{\bf S}_{1,\beta_1}^*\cdots {\bf S}_{k,\beta_k}^*
\end{split}
 \end{equation*}
where the series are  convergent in the operator norm topology. Moreover, due to  relation \eqref{BBBBB}, we have
$$\|\varphi_A(r_1{\bf S}_1,\ldots, r_k{\bf S}_k)\|\leq \|A\|, \qquad  r_i\in [0,1).
$$
Due to relation \eqref{fifia}, we have
\begin{equation*}
 \lim_{r\to 1}\varphi_A(r{\bf S})(h\otimes e_{\gamma_1}^1\otimes \cdots \otimes e_{\gamma_k}^k)
 =A(h\otimes e_{\gamma_1}^1\otimes \cdots \otimes e_{\gamma_k}^k).
 \end{equation*}
Since $\|\varphi_A(r{\bf S}_1,\ldots, r{\bf S}_k)\|\leq \|A\|$, an approximation argument shows that

\begin{equation}
\label{SOT-phi}\text{\rm SOT-}\lim_{r\to 1} \varphi_A(r{\bf S}_1,\ldots, r{\bf S}_k)=A.
\end{equation}
Let $\epsilon>0$ and choose a vector-valued polynomial $q\in  \cP$ with $\|q\|=1$ and $\|Aq\|>\|A\|-\epsilon$. Due to relation \eqref{SOT-phi}, there is $r_0\in (0,1)$ such that $\|\varphi_A(r_0{\bf S}_1,\ldots, r_0{\bf S}_k)q\|>\|A\|-\epsilon$. Hence, we deduce that $\sup_{r\in [0,1)} \|\varphi_A(r{\bf S}_1,\ldots, r{\bf S}_k)\|=\|A\|$.

Now, let $r_1,r_2\in [0,1)$ with $r_1<r_2$. We already proved that $g({\bf S}):=\varphi_A(r_2{\bf S}_1,\ldots, r_2{\bf S}_k)$ is in
$
\text{\rm span} \{f^*g:\ f, g\in B(\cE)\otimes_{min} \boldsymbol\cA_{\bf n}\}^{-\|\cdot\|}.
$
Due to the von Neumann \cite{vN} type inequality  from \cite{Po-Berezin-poly}, we have $\|g(r{\bf S})\|\leq \|g({\bf S})\|$ for any $r\in [0,1)$. In particular, setting $r=\frac{r_1}{r_2}$, we deduce that
$$\|\varphi_A(r_1{\bf S}_1,\ldots, r_1{\bf S}_k)\|\leq \|\varphi_A(r_2{\bf S}_1,\ldots, r_2{\bf S}_k)\|.
$$
Now, it is clear that $\lim_{r\to 1}\|\varphi_A(r{\bf S}_1,\ldots, r{\bf S}_k)\|=\|A\|$.
On the other hand, since $Aq=\varphi_A({\bf S)})q$ for any vector-valued polynomial $q\in \cE\otimes \bigotimes_{i=1}^kF^2(H_{n_i})$, we deduce that
$\|A\|=\sup_{q\in \cP, \|q\|\leq 1} \|\varphi({\bf S})q\|$.

Now, we prove the converse of the theorem.
Let
$\{ A_{(\alpha_1,\ldots,\alpha_k;\beta_1,\ldots, \beta_k)}\}$ be a family of operators in $B(\cE)$, where $\alpha_i,\beta_i\in \FF_{n_i}^+$, $|\alpha_i|=m_i^-, |\beta_i|=m_i^+$, $m_i\in \ZZ$, and $i\in \{1,\ldots, k\}$, and assume that conditions (i) and (ii) hold.
Note that, due to item (i), $\varphi({\bf S})p$ and $\varphi(r{\bf S})p$, $r\in [0,1)$, are vectors in $\cE\otimes \bigotimes_{i=1}^kF^2(H_{n_i})$ and
$$
\lim_{r\to 1}\varphi(r{\bf S})p=\varphi({\bf S})p
$$
for any $p\in \cP$. Since $\sup_{p\in \cP, \|p\|\leq 1} \|\varphi(r{\bf S})p\|<\infty$, there is a unique bounded linear operator $A_r\in B(\cE\otimes \bigotimes_{i=1}^kF^2(H_{n_i}))$ such that $A_rp=\varphi(r{\bf S})p$ for any $p\in \cP$. If $f\in \cE\otimes \bigotimes_{i=1}^kF^2(H_{n_i})$ and $\{p_m\}$ is a sequence of polynomials $p_m\in \cP$ such that $p_m\to f$ as $m\to \infty$, we set $A_r(f):=\lim_{m\to \infty} \varphi(r{\bf S})p_m$.  Note that the definition is correct.
On the other hand, note that
$$
\sup_{p\in \cP, \|p\|\leq 1} \|\varphi({\bf S})p\|<\infty.
$$
Indeed, this follows from the fact that $\lim_{r\to 1}\varphi(r{\bf S})p=\varphi({\bf S})p$ and $\sup_{p\in \cP, \|p\|\leq 1} \|\varphi(r{\bf S})p\|<\infty$. Consequently, there is a unique operator $A\in B(\cE\otimes \bigotimes_{i=1}^kF^2(H_{n_i}))$ such that $Ap=\varphi ({\bf S})p$ for any $p\in \cP$.
Since $\lim_{r\to 1}A_rp=\lim_{r\to 1}\varphi(r{\bf S})p=\varphi({\bf S})p=Ap$ and
$\sup_{r\in [0,1)}  \|A_r\|<\infty,$
we deduce that
$A=\text{\rm SOT-}\lim_{r\to 1} A_r$.

Now, we show that $A$ is a $k$-multi-Toeplitz operator. First, note that
${\bf S}_{1,\alpha_1}\cdots {\bf S}_{k,\alpha_k}{\bf S}_{1,\beta_1}^*\cdots {\bf S}_{k,\beta_k}^*$   is a $k$-multi-Toeplitz operator for any
$\alpha_i,\beta_i\in \FF_{n_i}^+, i\in \{1,\ldots, k\} $ with $m_i\in \ZZ$, $|\alpha_i|=m_i^-$,  $ |\beta_i|=m_i^+$. It is enough to check this on monomials of the form $h\otimes e_{\gamma_1}^1\otimes\ldots \otimes e_{\gamma_k}^k$.
Consequently,
\begin{equation*}
(I_\cE\otimes {\bf R}_{i,s}^*)\varphi(r{\bf S})(I_\cE\otimes {\bf R}_{i,t})p=\delta_{st}\varphi(r{\bf S}) p,\qquad s,t\in \{1,\ldots, n_i\},
 \end{equation*}
  for any $p\in \cP$ and every $i\in\{1,\ldots, k\}$.  Hence, $A_r$ has the same property. Taking $r\to 1$, we conclude that $A$ is a
  $k$-multi-Toeplitz operator.
  On the other hand, if  $x:=x_1\otimes \cdots \otimes x_k$, $y=y_1\otimes \cdots \otimes y_k$ satisfy relation \eqref{xy}, and $h,\ell\in \cE$, we have
  \begin{equation*}
  \begin{split}
  \left<A(h\otimes x), \ell\otimes y\right>&=\lim_{r\to 1} \left<A_r(h\otimes x), \ell\otimes y\right>=\lim_{r\to 1} \left<\varphi(r{\bf S})(h\otimes x), \ell\otimes y\right>\\
  &=\lim_{r\to 1} \left<r^{\sum_{i=1}^k |\alpha_i|+|\beta_i|} A_{(\alpha_1,\ldots,\alpha_k;\beta_1,\ldots, \beta_k)}h, \ell\right>
  =\left< A_{(\alpha_1,\ldots,\alpha_k;\beta_1,\ldots, \beta_k)}h,\ell\right>.
  \end{split}
  \end{equation*}
Therefore,
$$\varphi({\bf S}):= \sum_{m_1\in \ZZ}\cdots \sum_{m_k\in \ZZ} \sum_{{\alpha_i,\beta_i\in \FF_{n_i}^+, i\in \{1,\ldots, k\}}\atop{|\alpha_i|=m_i^-, |\beta_i|=m_i^+}}A_{(\alpha_1,\ldots,\alpha_k;\beta_1,\ldots, \beta_k)}
\otimes {\bf S}_{1,\alpha_1}\cdots {\bf S}_{k,\alpha_k}{\bf S}_{1,\beta_1}^*\cdots {\bf S}_{k,\beta_k}^*
$$
is the formal Fourier series of the  $k$-multi-Toeplitz operator $A$  on $\cE\otimes \bigotimes_{i=1}^k F^2(H_{n_i})$.
The proof is complete.
\end{proof}

\begin{theorem}
\label{caract2}
Let
$\{ A_{(\alpha_1,\ldots,\alpha_k;\beta_1,\ldots, \beta_k)}\}$ be a family of operators in $B(\cE)$, where $\alpha_i,\beta_i\in \FF_{n_i}^+$, $|\alpha_i|=m_i^-, |\beta_i|=m_i^+$, $m_i\in \ZZ$, and $i\in \{1,\ldots, k\}$, and let
$$\varphi({\bf S}):= \sum_{m_1\in \ZZ}\cdots \sum_{m_k\in \ZZ} \sum_{{\alpha_i,\beta_i\in \FF_{n_i}^+, i\in \{1,\ldots, k\}}\atop{|\alpha_i|=m_i^-, |\beta_i|=m_i^+}}A_{(\alpha_1,\ldots,\alpha_k;\beta_1,\ldots, \beta_k)}
\otimes {\bf S}_{1,\alpha_1}\cdots {\bf S}_{k,\alpha_k}{\bf S}_{1,\beta_1}^*\cdots {\bf S}_{k,\beta_k}^*
$$
   be the associated
 formal Fourier series. Then  $
\varphi({\bf S}) $ is the formal Fourier series of a $k$-multi-Toeplitz operator $A$ on $\cE\otimes \bigotimes_{i=1}^k F^2(H_{n_i})$ if and only if  the series defining $
\varphi(r{\bf S}) $ is convergent in the operator norm topology  for any $r\in[0,1)$, and
$$
\sup_{r\in [0,1)} \|\varphi(r{\bf S})\|<\infty.
$$
 Moreover, if $A$ is a $k$-multi-Toeplitz operator   on $\cE\otimes \bigotimes_{i=1}^k F^2(H_{n_i})$, then $\varphi(r{\bf S})=\boldsymbol\cB^{ext}_{r{\bf S}}[A]$ and
$$\text{\rm SOT-}\lim_{r\to 1}\boldsymbol\cB^{ext}_{r{\bf S}}[A]=A,$$
where
$$
 \boldsymbol\cB_{r {\bf S}}^{ext}[u]:=(I_{\cE}\otimes {\bf K}_{r{\bf S}}^*)(u\otimes I_{\otimes_{i=1}^kF^2(H_{n_{i}})})
(I_{\cE}\otimes {\bf K}_{r{\bf S}}), \qquad u\in B(\cE_{k}),
$$
and ${\bf K}_{r{\bf S}}$ is the noncommutative Berezin kernel associated with $r{\bf S}\in {\bf B_n}(\otimes _{i=1}^k F^2(H_{n_i}))$.
\end{theorem}
\begin{proof}  Assume that  $
\varphi({\bf S}) $ is the formal Fourier series of a $k$-multi-Toeplitz operator $A$ on $\cE\otimes \bigotimes_{i=1}^k F^2(H_{n_i})$. Then Theorem \ref{structure-Toeplitz} implies that
$\varphi(r{\bf S})$  is convergent in the operator norm topology
and
$$\|A\|=\sup_{r\in [0,1)}\|\varphi(r{\bf S})\|.
$$
We recall that  the noncommutative Berezin kernel associated with $r{\bf S}\in {\bf B_n}(\otimes _{i=1}^k F^2(H_{n_i}))$ is defined on
$\otimes_{i=1}^kF^2(H_{n_{i}})$ with values in $ \otimes_{i=1}^kF^2(H_{n_{i}})\otimes \cD_{r{\bf S}}\subset \left(\otimes_{i=1}^kF^2(H_{n_{i}})\right)\otimes \left(\otimes_{i=1}^kF^2(H_{n_{i}})\right)$, where $\cD_{r{\bf S}}:=\overline{\boldsymbol{\Delta}_{r{\bf S}}(I)(\otimes_{i=1}^kF^2(H_{n_i}))}$.
Let $\boldsymbol\gamma=(\gamma_1,\ldots, \gamma_k)$ and $\boldsymbol \omega=(\omega_1,\ldots, \omega_k)$ be  $k$-tuples in $\FF_{n_1}^+\times\cdots \times \FF_{n_k}^+$, set $q:=\max\{|\gamma_1|, \ldots |\gamma_k|, |\omega_1|,\ldots,|\omega_k|\}$, and define the operator
$$
\Gamma_q:=\sum_{m_1\in \ZZ, |m_1|\leq q}\cdots\sum_{m_k\in \ZZ, |m_k|\leq q}\sum_{{\alpha_i,\beta_i\in \FF_{n_k}^+, i\in \{1,\ldots, k\}}\atop{|\alpha_i|=m_i^-, |\beta_i|=m_i^+}} A_{(\alpha_1,\ldots, \alpha_k, \beta_1,\ldots, \beta_k)}\otimes {\bf S}_{\boldsymbol\alpha} {\bf S}_{\boldsymbol\beta}^*,
$$
where we use the notation ${\bf S}_{\boldsymbol\alpha}:={\bf S}_{1,\alpha_1}\cdots {\bf S}_{k,\alpha_k}$ if $\boldsymbol\alpha:=(\alpha_1,\ldots, \alpha_k)\in \FF_{n_1}^+\times\cdots \times \FF_{n_k}^+$. We also set $e_{\boldsymbol\alpha}:=e_{\alpha_1}^1\otimes \cdots \otimes e_{\alpha_k}^k$.
Note that
\begin{equation*}
\begin{split}
&\left<\boldsymbol\cB_{r {\bf S}}^{ext}[A](h\otimes e_{\boldsymbol\gamma}), h'\otimes e_{\boldsymbol\omega}\right>\\
&=
\left<(I_{\cE}\otimes {\bf K}_{r{\bf S}}^*)(A\otimes I_{\otimes_{i=1}^kF^2(H_{n_{i}})})
(I_{\cE}\otimes {\bf K}_{r{\bf S}})(h\otimes e_{\boldsymbol\gamma}), h'\otimes e_{\boldsymbol\omega}\right>\\
&=\left<(A\otimes I_{\otimes_{i=1}^kF^2(H_{n_{i}})})
\sum_{\boldsymbol\alpha\in \FF_{n_1}^+\times\cdots \times \FF_{n_k}^+}h\otimes e_{\boldsymbol\alpha}\otimes \boldsymbol\Delta_{r{\bf S}}(I)^{1/2}{\bf S}_{\boldsymbol\alpha}^*(e_{\boldsymbol\gamma})\right.,\\
&\qquad\qquad \left.
\sum_{\boldsymbol\beta\in \FF_{n_1}^+\times\cdots \times \FF_{n_k}^+}h'\otimes e_{\boldsymbol\beta}\otimes \boldsymbol\Delta_{r{\bf S}}(I)^{1/2}{\bf S}_{\boldsymbol\beta}^*(e_{\boldsymbol\omega})\right>\\
&=
\sum_{\boldsymbol\alpha\in \FF_{n_1}^+\times\cdots \times \FF_{n_k}^+}
\sum_{\boldsymbol\beta\in \FF_{n_1}^+\times\cdots \times \FF_{n_k}^+}
\left<
A(h\otimes e_{\boldsymbol\alpha})\otimes \boldsymbol\Delta_{r{\bf S}}(I)^{1/2}{\bf S}_{\boldsymbol\alpha}^*(e_{\boldsymbol\gamma}),
h'\otimes e_{\boldsymbol\beta}\otimes \boldsymbol\Delta_{r{\bf S}}(I)^{1/2}{\bf S}_{\boldsymbol\beta}^*(e_{\boldsymbol\omega})\right>\\
&=
\sum_{\boldsymbol\alpha\in \FF_{n_1}^+\times\cdots \times \FF_{n_k}^+}
\sum_{\boldsymbol\beta\in \FF_{n_1}^+\times\cdots \times \FF_{n_k}^+}
\left<
A(h\otimes e_{\boldsymbol\alpha}), h'\otimes e_{\boldsymbol\beta}\right>
\left<\boldsymbol\Delta_{r{\bf S}}(I)^{1/2}{\bf S}_{\boldsymbol\alpha}^*(e_{\boldsymbol\gamma}),
 \boldsymbol\Delta_{r{\bf S}}(I)^{1/2}{\bf S}_{\boldsymbol\beta}^*(e_{\boldsymbol\omega})\right> \\
 &=
 \sum_{m_1\in \ZZ, |m_1|\leq q}\cdots\sum_{m_k\in \ZZ, |m_k|\leq q}\sum_{{\alpha_i,\beta_i\in \FF_{n_k}^+, i\in \{1,\ldots, k\}}\atop{|\alpha_i|=m_i^-, |\beta_i|=m_i^+}}
 \left<
\Gamma_q(h\otimes e_{\boldsymbol\alpha}), h'\otimes e_{\boldsymbol\beta}\right>
\left<\boldsymbol\Delta_{r{\bf S}}(I)^{1/2}{\bf S}_{\boldsymbol\alpha}^*(e_{\boldsymbol\gamma}),
 \boldsymbol\Delta_{r{\bf S}}(I)^{1/2}{\bf S}_{\boldsymbol\beta}^*(e_{\boldsymbol\omega})\right> \\
 &=
 \sum_{\boldsymbol\alpha\in \FF_{n_1}^+\times\cdots \times \FF_{n_k}^+}
\sum_{\boldsymbol\beta\in \FF_{n_1}^+\times\cdots \times \FF_{n_k}^+}
\left<
\Gamma_q(h\otimes e_{\boldsymbol\alpha}), h'\otimes e_{\boldsymbol\beta}\right>
\left<\boldsymbol\Delta_{r{\bf S}}(I)^{1/2}{\bf S}_{\boldsymbol\alpha}^*(e_{\boldsymbol\gamma}),
 \boldsymbol\Delta_{r{\bf S}}(I)^{1/2}{\bf S}_{\boldsymbol\beta}^*(e_{\boldsymbol\omega})\right> \\
 &=
 \left<(I_{\cE}\otimes {\bf K}_{r{\bf S}}^*)(\Gamma_q\otimes I_{\otimes_{i=1}^kF^2(H_{n_{i}})})
(I_{\cE}\otimes {\bf K}_{r{\bf S}})(h\otimes e_{\boldsymbol\gamma}), h'\otimes e_{\boldsymbol\omega}\right>\\
&=\left<\boldsymbol\cB_{r {\bf S}}^{ext}[\Gamma_q](h\otimes e_{\boldsymbol\gamma}), h'\otimes e_{\boldsymbol\omega}\right>\\
&=
\sum_{m_1\in \ZZ, |m_1|\leq q}\cdots\sum_{m_k\in \ZZ, |m_k|\leq q}\sum_{{\alpha_i,\beta_i\in \FF_{n_k}^+, i\in \{1,\ldots, k\}}\atop{|\alpha_i|=m_i^-, |\beta_i|=m_i^+}}
\left< \left(A_{(\alpha_1,\ldots, \alpha_k, \beta_1,\ldots, \beta_k)}\otimes r^{\sum_{i=1}^k(|\alpha_i|+|\beta_i|)} {\bf S}_{\boldsymbol\alpha} {\bf S}_{\boldsymbol\beta}^*\right)(h\otimes e_{\gamma}), h'\otimes e_{\boldsymbol\omega}\right>\\
&=\left<\varphi_A(rS_1,\ldots, rS_k)(h\otimes e_{\boldsymbol\gamma}), h'\otimes e_{\boldsymbol\omega}\right>.
\end{split}
\end{equation*}
Consequently, we obtain
$$\boldsymbol\cB_{r {\bf S}}^{ext}[A]=\varphi_A(r{\bf S}_1,\ldots, r{\bf S}_k), \qquad r\in [0,1),
$$
which proves the second part of the theorem.

To prove the converse, assume that $\{ A_{(\alpha_1,\ldots,\alpha_k;\beta_1,\ldots, \beta_k)}\}$ is a family of operators in $B(\cE)$, where $\alpha_i,\beta_i\in \FF_{n_i}^+$, $|\alpha_i|=m_i^-, |\beta_i|=m_i^+$, $m_i\in \ZZ$, and $i\in \{1,\ldots, k\}$, and let $\varphi({\bf S})$ be the associated formal Fourier series. We also assume that $\varphi(r{\bf S})$ is convergent in the operator norm topology for each $r\in [0,1)$, and that $$
M:=\sup_{r\in [0,1)} \|\varphi(r{\bf S})\|<\infty.
$$
Note that $\varphi(r{\bf S})$ is a $k$-multi-Toeplitz operator
and
$$
  \varphi(r{\bf S})(h\otimes e_{\gamma_1}^1\otimes \cdots \otimes e_{\gamma_k}^k)=
   \sum_{{\boldsymbol\omega=(\omega_1,\ldots, \omega_k)\in \FF_{n_1}^+\times\cdots \times \FF_{n_k}^+}\atop{\boldsymbol\omega\sim_{rc} \boldsymbol\gamma}} r^{\sum_{i=1}^k \left(|c_r^+(\boldsymbol\omega, \boldsymbol\gamma)|+ |c_r^-(\boldsymbol\omega, \boldsymbol\gamma)|\right)}
    A_{(c_r^+(\boldsymbol\omega, \boldsymbol\gamma); c_r^-(\boldsymbol\omega, \boldsymbol\gamma))}h\otimes  e_{\omega_1}^1\otimes \cdots \otimes e_{\omega_k}^k
  $$
  is a vector in $\cE\otimes \bigotimes_{i=1}^kF^2(H_{n_i})$. Hence, we deduce that, for each $\boldsymbol\gamma=(\gamma_1,\ldots, \gamma_k)\in \FF_{n_1}^+\times\cdots \times \FF_{n_k}^+$,
  \begin{equation*}
  \begin{split}
    &\left<r^{\sum_{i=1}^k \left(c_r^+(\boldsymbol\omega, \boldsymbol\gamma)+ c_r^-(\boldsymbol\omega, \boldsymbol\gamma)\right)}\sum_{{\boldsymbol\omega\in \FF_{n_1}^+\times\cdots \times \FF_{n_k}^+}\atop{\boldsymbol\omega\sim_{rc} \boldsymbol\gamma}}
    A_{(c_r^+(\boldsymbol\omega, \boldsymbol\gamma);c_r^-(\boldsymbol\omega, \boldsymbol\gamma))}^*A_{(c_r^+(\boldsymbol\omega, \boldsymbol\gamma);c_r^-(\boldsymbol\omega, \boldsymbol\gamma))}
    h,h\right>\\
    &\qquad \qquad \leq\|\varphi(r{\bf S})\|^2 \|h\|^2\leq M\|h\|^2,
    \end{split}
    \end{equation*}
for any $r\in [0,1)$ and $h\in \cE$. Taking $r\to 1$, we get condition (i) of Theorem \ref{structure-Toeplitz}. Applying the latter theorem, we deduce that
$\varphi({\bf S})$ is the Fourier series of a $k$-multi-Toeplitz operators.
The proof is complete.
\end{proof}

We remark that,   due to  Theorem \ref{caract2},   the order of the series in the  definition of  $\varphi_A(r{\bf S}_1,\ldots, r{\bf S}_k)$  (see item (a) of  Theorem \ref{structure-Toeplitz}) is irrelevant.

\begin{theorem} \label{characterization}  Let  ${\bf n}=(n_1,\ldots, n_k)\in \NN^k$  and let  $\boldsymbol{\cT_{\bf n}}$ be the set of all   $k$-multi-Toeplitz operators  on $\cE\otimes\bigotimes_{i=1}^k F^2(H_{n_i})$.
 Then
\begin{equation*}\begin{split}
\boldsymbol{\cT_{\bf n}}&=\text{\rm span} \{f^*g:\ f, g\in B(\cE)\otimes_{min} \boldsymbol\cA_{\bf n}\}^{- \text{\rm SOT}}\\
&=\text{\rm span} \{f^*g:\ f, g\in B(\cE)\otimes_{min} \boldsymbol\cA_{\bf n}\}^{- \text{\rm WOT}},
\end{split}
\end{equation*}
where $\boldsymbol\cA_{\bf n}$ is the polyball algebra.
\end{theorem}
\begin{proof} Denote
$$
\cG:=\text{\rm span} \{f^*g:\ f, g\in B(\cE)\otimes_{min} \boldsymbol\cA_{\bf n}\}^{\|\cdot\|}.
$$
According to Theorem \ref{structure-Toeplitz}, if $A\in \boldsymbol{\cT_{\bf n}}$ and $\varphi_A({\bf S})$ is its Fourier series, then $\varphi_A(r{\bf S})\in \cG$ for any $r\in [0,1)$, and
$A=\text{\rm SOT-}\lim \varphi_A(r{\bf S})$. Consequently, $\boldsymbol{\cT_{\bf n}}\subseteq \overline{\cG}^{\rm SOT}$. Conversely, note that each monomial
${\bf S}_{\boldsymbol\alpha}^* {\bf S}_{\boldsymbol\beta}$, $\boldsymbol\alpha, \boldsymbol\beta\in \FF_{n_1}^+\times\cdots \times \FF_{n_k}^+$,   is a $k$-multi-Toeplitz operator. This shows that, for each $Y\in \cG$,
$$
 (I_\cE\otimes {\bf R}_{i,s}^*)Y(I_\cE\otimes {\bf R}_{i,t})=\delta_{st}Y,\qquad s,t\in \{1,\ldots, n_i\},
 $$
  for every $i\in\{1,\ldots, k\}$. Consequently, taking SOT-limits,  we deduce that $\overline{\cG}^{\rm SOT}\subseteq \boldsymbol{\cT_{\bf n}}$, which proves that  $\overline{\cG}^{\rm SOT}=\boldsymbol{\cT_{\bf n}}$.

  Now, if $T\in \overline{\cG}^{\rm WOT}$, an argument as above shows that $T\in \boldsymbol{\cT_{\bf n}}=\overline{\cG}^{\rm SOT}$. Since $\overline{\cG}^{\rm SOT}\subseteq \overline{\cG}^{\rm WOT}$, we conclude that
  $\boldsymbol{\cT_{\bf n}}
  =\overline{\cG}^{\rm SOT}=\overline{\cG}^{\rm WOT}$.
  The proof is complete.
  \end{proof}

\begin{corollary} \label{characterization2} The set of all   $k$-multi-Toeplitz operators  on $\bigotimes_{i=1}^k F^2(H_{n_i})$ coincides with

\begin{equation*}\begin{split}
\text{\rm span} \{\boldsymbol\cA_{\bf n}^* \boldsymbol\cA_{\bf n}\}^{- \text{\rm SOT}}
=\text{\rm span} \{\boldsymbol\cA_{\bf n}^* \boldsymbol\cA_{\bf n}\}^{- \text{\rm WOT}},
\end{split}
\end{equation*}
where $\boldsymbol\cA_{\bf n}$ is the polyball algebra.
\end{corollary}

\bigskip

\section{Bounded free $k$-pluriharmonic functions
and  Dirichlet extension problem}

In this section,
 we show that the bounded free $k$-pluriharmonic functions on ${\bf B_n}$ are precisely the noncommutative Berezin transforms of $k$-multi-Toeplitz operators and
  solve the Dirichlet extension problem for the
 regular polyball ${\bf B}_{\bf n}$.

\begin{definition} \label{pluri-def} A function $F$ with operator-valued coefficients in $B(\cE)$ is called free $k$-pluriharmonic on the polyball ${\bf B_n}$    if it has the form
$$
F({\bf X})= \sum_{m_1\in \ZZ}\cdots \sum_{m_k\in \ZZ} \sum_{{\alpha_i,\beta_i\in \FF_{n_i}^+, i\in \{1,\ldots, k\}}\atop{|\alpha_i|=m_i^-, |\beta_i|=m_i^+}}A_{(\alpha_1,\ldots,\alpha_k;\beta_1,\ldots, \beta_k)}
\otimes {\bf X}_{1,\alpha_1}\cdots {\bf X}_{k,\alpha_k}{\bf X}_{1,\beta_1}^*\cdots {\bf X}_{k,\beta_k}^*,
$$
where the series converge in the operator norm topology for any  ${\bf X}=(X_1,\ldots, X_k)\in {\bf B_n}(\cH)$, with $X_i:=(X_{i,1},\ldots, X_{i,n_i})$, and any Hilbert space $\cH$.
\end{definition}

Due to the remark following Theorem \ref{caract2}, one can prove that the order of the series in the definition above is irrelevant. Note that any free holomorphic function on ${\bf B_n}$ is $k$-pluriharmonic. Indeed, according to \cite{Po-automorphisms-polyball}, any free holomorphic function on the polyball ${\bf B_n}$  has the form
$$
f({\bf X})= \sum_{m_1\in \NN}\cdots \sum_{m_k\in \NN} \sum_{{\alpha_i\in \FF_{n_i}^+, i\in \{1,\ldots, k\}}\atop{|\alpha_i|=m_i}}A_{(\alpha_1,\ldots,\alpha_k)}
\otimes {\bf X}_{1,\alpha_1}\cdots {\bf X}_{k,\alpha_k},\qquad {\bf X}\in {\bf B_n}(\cH),
$$
where the series converge in the operator norm topology.
 A
function $F:{\bf B}_{\bf n}(\cH)\to B(\cE\otimes \cH)$ is called bounded if
$$\|F\|:=\sup_{{\bf X} \in {\bf B}_{\bf n}(\cH)}||F({\bf X} )\|<\infty.
$$
A free $k$-pluriharmonic function is bounded if its
representation on any Hilbert space is bounded.
Denote by  ${\bf PH}_\cE^\infty({\bf B}_{\bf n})$  the set of all bounded free $k$-
pluriharmonic functions on the polyball ${\bf B}_{\bf n}$ with coefficients in
$B(\cE)$.
 For each $m=1,2,\ldots$,
we define the norms $\|\cdot
\|_m:M_m\left({\bf PH}_\cE^\infty({\bf B}_{\bf n} )\right)\to [0,\infty)$ by
setting
$$
\|[F_{ij}]_m\|_m:= \sup \|[F_{ij}({\bf X})]_m\|,
$$
where the supremum is taken over all $n$-tuples ${\bf X}\in {\bf B}_{\bf n}(\cH)$ and any Hilbert space $\cH$. It is easy to see that the norms
$\|\cdot\|_m$, $m=1,2,\ldots$, determine  an operator space
structure  on ${\bf PH}_\cE^\infty({\bf B}_{\bf n})$,
 in the sense of Ruan (see e.g. \cite{ER}).

Let  $\boldsymbol{\cT_{\bf n}}$ be the set of all  of all $k$-multi-Toeplitz operators  on $\cE\otimes\bigotimes_{i=1}^k F^2(H_{n_i})$.
 According to Theorem \ref{characterization}, we have
\begin{equation*}
\boldsymbol{\cT_{\bf n}}=\text{\rm span} \{f^*g:\ f, g\in B(\cE)\otimes_{min} \boldsymbol\cA_{\bf n}\}^{- \text{\rm SOT}},
\end{equation*}
where $\boldsymbol\cA_{\bf n}$ is the polyball algebra.
The main result of this section is the following characterization of
bounded  free $k$-pluriharmonic  functions.

 \begin{theorem}\label{bounded}
  If $F: {\bf B_n}(\cH)\to B(\cE)\otimes_{min}B(\cH)$, then the following statements are equivalent:
\begin{enumerate}
\item[(i)] $F$ is a bounded free $k$-pluriharmonic function;
\item[(ii)]
there exists $A\in \boldsymbol{\cT_{\bf n}}$ such
that
$$F({\bf X})=\boldsymbol{\cB}_{\bf X}^{ext}[A]:= (I_\cE\otimes {\bf K}^*_{\bf X}) (A\otimes I_\cH)(I_\cE\otimes {\bf K}_{\bf X}), \qquad
{\bf X}\in {\bf B}_{\bf n}(\cH).
$$
\end{enumerate}
In this case,
  $A=\text{\rm SOT-}\lim\limits_{r\to 1}F(r{\bf S}).
  $
   Moreover, the map
$$
\Phi:{\bf PH}_\cE^\infty({\bf B}_{\bf n})\to \boldsymbol{\cT_{\bf n}}\quad
\text{ defined by } \quad \Phi(F):=A
$$ is a completely   isometric isomorphism of operator spaces.
\end{theorem}
\begin{proof} Assume that $F$ is a bounded free $k$-pluriharmonic function on
${\bf B}_{\bf n}$  and has the representation from  Definition \ref{pluri-def}. Then, for any $r\in [0,1)$,
$$
F(r{\bf S})\in \text{\rm span} \{f^*g:\ f, g\in B(\cE)\otimes_{min} \boldsymbol\cA_{\bf n}\}^{-\|\cdot\|}
$$
and, due to the noncommutative von Neumann inequality \cite{Po-poisson}, we have
$\sup_{r\in [0,1)}\|F(r{\bf S})\|=\|F\|_\infty<\infty$. According to Theorem \ref{caract2}, $F({\bf S})$ is the formal Fourier  series of a $k$-multi-Toeplitz operator $A\in B(\cE\otimes \otimes_{i=1} F^2(H_{n_i}))$ and
$A=\text{\rm SOT-}\lim_{r\to 1}F(r{\bf S})\in \boldsymbol{\cT_{\bf n}}$.
Using the properties of the noncommutative Berezin kernel on polyballs, we have
\begin{equation*}
F(r{\bf X})=(I_\cE\otimes {\bf K}^*_{\bf X})[F(r{\bf S})\otimes I_\cH](I_\cE\otimes  {\bf K}_{\bf X}),\qquad {\bf X}\in {\bf B_n}(\cH).
\end{equation*}
Since the map $Y\mapsto Y\otimes I_\cH$ is  SOT-continuous on bounded subsets of $B(\cE\otimes \otimes_{i=1} F^2(H_{n_i}))$,  we deduce that
$$
\text{\rm SOT-}\lim_{r\to 1}F(r{\bf X})=(I_\cE\otimes {\bf K}^*_{\bf X})[A\otimes I_\cH](I_\cE\otimes  {\bf K}_{\bf X})=\boldsymbol{\cB}_{\bf X}^{ext}[A].
$$
Since $F$ is continuous in the norm topology on ${\bf B_n}(\cH)$, we have  $F(r{\bf X})\to F({\bf X})$ as $r\to 1$. Consequently, the relation above implies  $F({\bf X})=\boldsymbol{\cB}_{\bf X}^{ext}[A]$, which completes the proof of the implication (i)$\implies$ (ii).

To prove that (ii)$\implies$ (i), let $A\in \cT_{\bf n}$ and
$F({\bf X}):=\boldsymbol{\cB}_{\bf X}^{ext}[A]$  for
${\bf X}\in {\bf B}_{\bf n}(\cH)$.
Since $A$ is a $k$-multi-Toeplitz  operator, Theorem \ref{structure-Toeplitz} shows that it has a formal Fourier series
$$\varphi({\bf S}):= \sum_{m_1\in \ZZ}\cdots \sum_{m_k\in \ZZ} \sum_{{\alpha_i,\beta_i\in \FF_{n_i}^+, i\in \{1,\ldots, k\}}\atop{|\alpha_i|=m_i^-, |\beta_i|=m_i^+}}A_{(\alpha_1,\ldots,\alpha_k;\beta_1,\ldots, \beta_k)}
\otimes {\bf S}_{1,\alpha_1}\cdots {\bf S}_{k,\alpha_k}{\bf S}_{1,\beta_1}^*\cdots {\bf S}_{k,\beta_k}^*
$$
 with the property that the series $\varphi(r{\bf S})$ is convergent in the operator
 norm topology to an operator in $\text{\rm span} \{f^*g:\ f, g\in B(\cE)\otimes_{min} \boldsymbol\cA_{\bf n}\}^{-\|\cdot\|}$. Moreover, we have
 $A=\text{\rm SOT-}\lim_{r\to 1} \varphi(r{\bf S})
$
and
$$
\|A\|=\sup_{r\in [0,1)}\|\varphi(r{\bf S})\|.
$$
Hence, the map ${\bf X}\mapsto \varphi({\bf X})$ is a $k$-pluriharmonic function on ${\bf B_n}(\cH)$. On the other hand,
due to Theorem \ref{caract2}, we have
$\varphi(r{\bf S})=\boldsymbol\cB^{ext}_{r{\bf S}}[A]$,
where
$$
 \boldsymbol\cB_{r {\bf S}}^{ext}[u]:=(I_{\cE}\otimes {\bf K}_{r{\bf S}}^*)(u\otimes I_{\otimes_{i=1}^kF^2(H_{n_{i}})})
(I_{\cE}\otimes {\bf K}_{r{\bf S}}), \qquad u\in B(\cE_{k}),
$$
and ${\bf K}_{r{\bf S}}$ is the noncommutative Berezin kernel associated with $r{\bf S}\in {\bf B_n}(\otimes _{i=1}^k F^2(H_{n_i}))$.
Note that
\begin{equation*}
\begin{split}
\varphi(r{\bf X})&=\boldsymbol{\cB}_{\bf X}^{ext}[\varphi(r{\bf S})]
=(I_\cE\otimes {\bf K}^*_{\bf X})[\varphi(r{\bf S})\otimes I_\cH](I_\cE\otimes  {\bf K}_{\bf X}).
\end{split}
\end{equation*}
Now, using continuity of $\varphi$ on ${\bf B_n}(\cH)$ and the fact that $A=\text{\rm SOT-}\lim_{r\to 1} \varphi(r{\bf S})$, we deduce that
$$
\varphi({\bf X})=\text{\rm SOT-}\lim_{r\to 1} \varphi(r{\bf X})=\boldsymbol{\cB}_{\bf X}^{ext}[A]=F({\bf X}),\qquad {\bf X}\in {\bf B_n}(\cH).
$$
To prove the last part of the theorem, let $[F_{ij}]_m\in M_m({\bf PH}_\cE^\infty({\bf B}_{\bf n}))$ and use the noncommutative von Neumann inequality to obtain
$$
\|[F_{ij}]_m\|=\sup_{{\bf X}\in {\bf B_n}(\cH)}\|[F_{ij}({\bf X})]_m\|
=\sup_{r\in [0,1)}\|[F_{ij}(r{\bf S})]_m\|.
$$
On the other hand, $A_{ij}:=\text{\rm SOT-}\lim_{r\to 1} F_{ij}(rS)$ is a $k$-multiToeplitz operator and
$$
 F_{ij}(rS)=(I_{\cE}\otimes {\bf K}_{r{\bf S}}^*)(A_{ij}\otimes I_{\otimes_{i=1}^kF^2(H_{n_{i}})})
(I_{\cE}\otimes {\bf K}_{r{\bf S}}).
$$
Hence, we obtain
$$
\sup_{r\in [0,1)}\|[F_{ij}(r{\bf S})]_m\|\leq \|[A_{ij}]_m\|.
$$
Since  $[A_{ij}]_m:=\text{\rm SOT-}\lim_{r\to 1} [F_{ij}(rS)]_m$, we deduce that the inequality above is in fact an equality. This shows that $\Phi$ is a completely isometric isomorphisms of operator spaces. The proof is complete.
\end{proof}

As a consequence, we can obtain the following Fatou type result concerning the boundary behaviour of bounded $k$-pluriharmonic functions.

\begin{corollary}
If $F:{\bf B}_{\bf n}(\cH)\to B(\cE)\otimes_{min} B( \cH)$ is a bounded free $k$-pluriharmonic function and ${\bf X}$ is a pure element in
${\bf B}_{\bf n}(\cH)^-$, then the limit
$$
\text{\rm SOT-}\lim_{r\to 1} F(r{\bf X})
$$
exists.
\end{corollary}
\begin{proof} If ${\bf X}$ is a pure element in
${\bf B}_{\bf n}(\cH)^-$, then the noncommutative Berezin kernel ${\bf K}_{\bf X}$ is an isometry (see \cite{Po-Berezin-poly}).  Since $F$ is free $k$-pluriharmonic function on ${\bf B_n}$, we have
$$
F(r{\bf S})\in \text{\rm span} \{f^*g:\ f, g\in B(\cE)\otimes_{min} \boldsymbol\cA_{\bf n}\}^{-\|\cdot\|}
$$
and $F(r{\bf S})$ converges in the operator norm topology. Consequently,
$$
F(r{\bf X})=(I_\cE\otimes {\bf K}_{\bf X}^*)[F(r{\bf S})\otimes I_\cH](I_\cE\otimes  {\bf K}_{\bf X}).
$$
Since $F$ is bounded, Theorem  \ref{bounded} implies
 $ \text{\rm SOT-}\lim\limits_{r\to 1}F(r{\bf S})=A\in \boldsymbol{\cT_{\bf n}}$
 and $\sup_{0\leq r<1}\|F(r{\bf S})\|<\infty$. Using these facts in relation above, we conclude that $ \text{\rm SOT-}\lim\limits_{r\to 1}F(r{\bf X})$
 exists. The proof is complete.
  \end{proof}

We denote by ${\bf HP}_\cE^c({\bf B}_{\bf n}) $ the set of all
  free $k$-pluriharmonic functions on ${\bf B}_{\bf n}$ with operator-valued coefficients in $B(\cE)$, which
 have continuous extensions   (in the operator norm topology) to
the closed polyball ${\bf B}_{\bf n}(\cH)^-$, for any Hilbert space $\cH$. Throughout this section, we assume
that $\cH$ is  an  infinite dimensional Hilbert space.
In what follows we solve the Dirichlet extension problem for the regular polyballs.

\begin{theorem}\label{Dirichlet}  If $F:{\bf B}_{\bf n}(\cH)\to B(\cE)\otimes_{min} B( \cH)$, then
 the following statements are equivalent:
\begin{enumerate}
\item[(i)] $F$ is a free $k$-pluriharmonic function on ${\bf B}_{\bf n}(\cH)$
such that \ $F(r{\bf S})$ converges in the operator norm
topology, as $r\to 1$;

\item[(ii)]
there exists $A\in \boldsymbol\cP:=\text{\rm span} \{f^*g:\ f, g\in B(\cE)\otimes_{min} \boldsymbol\cA_{\bf n}\}^{-\|\cdot\|}$ such
that $$F({\bf X})=\boldsymbol{\cB}^{ext}_{\bf X}[A], \qquad {\bf X}\in {\bf B}_{\bf n}(\cH);
  $$
   \item[(iii)] $F$ is a free $k$-pluriharmonic function on ${\bf B}_{\bf n}(\cH)$ which
 has a continuous extension  (in the operator norm topology) to
the closed ball ${\bf B}_{\bf n}(\cH)^-$.

\end{enumerate}
In this case, $A=\lim\limits_{r\to 1}F(r{\bf S})$, where
the convergence is in the operator norm. Moreover, the map
$$
\Phi:{\bf PH}_\cE^c({\bf B}_{\bf n}) \to \boldsymbol\cP
 \quad \text{ defined
by } \quad \Phi(F):=A
$$ is a  completely   isometric isomorphism of
operator spaces.
\end{theorem}
\begin{proof}

Assume that item (i) holds. Then    $F$ has a representation
$$
F({\bf X})= \sum_{m_1\in \ZZ}\cdots \sum_{m_k\in \ZZ} \sum_{{\alpha_i,\beta_i\in \FF_{n_i}^+, i\in \{1,\ldots, k\}}\atop{|\alpha_i|=m_i^-, |\beta_i|=m_i^+}}A_{(\alpha_1,\ldots,\alpha_k;\beta_1,\ldots, \beta_k)}
\otimes {\bf X}_{1,\alpha_1}\cdots {\bf X}_{k,\alpha_k}{\bf X}_{1,\beta_1}^*\cdots {\bf X}_{k,\beta_k}^*,
$$
where the series converge in the operator norm topology for any  ${\bf X}=(X_1,\ldots, X_k)\in {\bf B_n}(\cH)$. Since the series defining $F(r{\bf S})$ converges in the operator topology, we deduce that
\begin{equation}
\label{AA:}
A:=\lim\limits_{r\to 1}F(r{\bf S})\in \boldsymbol\cP.
\end{equation}
On the other hand, we have
$$
\boldsymbol{\cB}^{ext}_{\bf X}[A]=(I_\cE\otimes {\bf K}_{\bf X}^*)[F(r{\bf S})\otimes I_\cH](I_\cE\otimes  {\bf K}_{\bf X})=F(r{\bf X})
$$
for any $r\in [0,1)$ and ${\bf X}\in {\bf B_n}(\cH)$. Hence, and using relation \eqref{AA:}, we deduce that
$$
\boldsymbol{\cB}^{ext}_{\bf X}[A]=\lim_{r\to 1} F(r{\bf X})=F({\bf X}),
$$
which proves item (ii). Now, we show that (ii) $\implies $ (i).
Assuming item (ii) and taking into account Theorem \ref{characterization}, one can see that $A$ is a $k$-multi-Toeplitz operator.
As in the proof of Theorem \ref{bounded}, the map defined by
$F({\bf X}):=\boldsymbol{\cB}_{\bf X}^{ext}[A]$, $ {\bf X}\in {\bf B_n}(\cH)$, is a bounded free $k$-pluriharmonic function.
Moreover, we proved that
\begin{equation}\label{Frs}
F(r{\bf S})=\boldsymbol\cB^{ext}_{r{\bf S}}[A],\qquad r\in [0,1),
\end{equation}
$F(r{\bf S})\in \boldsymbol\cP$  and also that
$A=\text{\rm SOT-}\lim_{r\to 1} F(r{\bf S})
$
and
$
\|A\|=\sup_{r\in [0,1)}\|F(r{\bf S})\|.
$
Since $A\in \boldsymbol\cP $, there is a sequence of polynomials $q_m$ in  ${\bf S}_{\boldsymbol\alpha}^*{\bf S}_{\boldsymbol\beta}$ such that $q_m\to A$ in norm as $m\to \infty$.
For any $\epsilon>0$, let $N\in \NN$ be such that $\|A-q_m\|<\frac{\epsilon}{3}$ for any $m\geq N$. Choose $\delta\in (0,1)$ such that
$\|\boldsymbol\cB^{ext}_{r{\bf S}}[q_N]-q_N\|<\frac{\epsilon}{3}$ for any $r\in (\delta,1)$. Note that
\begin{equation*}
\begin{split}
\|\boldsymbol\cB^{ext}_{r{\bf S}}[A]-A\|
&\leq\|\boldsymbol\cB^{ext}_{r{\bf S}}[A-q_N]\|+\|\boldsymbol\cB^{ext}_{r{\bf S}}[q_N]-q_N\|+\|q_N-A\|\\
&\leq \|A-q_N\|+2\frac{\epsilon}{3}<\epsilon
\end{split}
\end{equation*}
for any $r\in (\delta, 1)$. Therefore, $\lim_{r\to 1}\boldsymbol\cB^{ext}_{r{\bf S}}[A]=A$ in the norm topology. Hence and due to  \eqref{Frs}, we deduce that
$\lim_{r\to 1} F(r{\bf S})=A$ in the norm topology, which shows that item (i) holds.
Since $\cH$ is infinite dimensional, the implication (iii) $\implies $ (i) is clear. It remains to prove that
(ii) $\implies$ (iii).  We assume that (ii) holds. Then
there exists $A\in \boldsymbol\cP$ such
that $F({\bf X})=\boldsymbol{\cB}^{ext}_{\bf X}[A]$ for all ${\bf X}\in {\bf B}_{\bf n}(\cH)$.
Due to Theorem \ref{bounded}, $F$ is a bounded free $k$-pluriharmonic function on ${\bf B_n}(\cH)$.
For any ${\bf Y}\in {\bf B_n}(\cH)^-$, one can show, as in the proof of the implication (ii) $\implies$ (i), that
$\widetilde F({\bf Y}):=\lim_{r\to 1}\boldsymbol\cB^{ext}_{r{\bf Y}}[A]$
exists in the operator norm topology.
Since $\|\boldsymbol\cB^{ext}_{r{\bf Y}}[A]\|\leq \|A\|$ for any $r\in [0,1)$, we deduce that
$\|\widetilde F({\bf Y})\|\leq \|A\|$ for any ${\bf Y}\in {\bf B_n}(\cH)^-$.
Note also that  $\widetilde F$ is an extension of $F$.
It remains to show that $\widetilde F$ is continuous on ${\bf B_n}(\cH)^-$. To this end, let $\epsilon>0$ and, due to the equivalence  of (ii) and (i), we can choose $r_0\in [0,1)$ such that
$\|A-F(r_0 {\bf S})\|<\frac{\epsilon}{3}$. Since $A-F(r_0{\bf S})\in \boldsymbol\cP$, we deduce that
\begin{equation*}
\begin{split}
\|\widetilde F({\bf Y})-F(r_0{\bf Y})\|&=\|\lim_{r\to1}\boldsymbol\cB^{ext}_{r{\bf Y}}[A] -F(r_0{\bf Y})\|\\
&\leq \limsup_{r\to 1} \|\boldsymbol\cB^{ext}_{r{\bf Y}}[A] -F(r_0{\bf Y})\|\\
&\leq \|A-F(r_0{\bf Y})\|<\frac{\epsilon}{3}
\end{split}
\end{equation*}
for any ${\bf Y}\in {\bf B_n}(\cH)^-$.
Since $F$ is continuous on ${\bf B_n}(\cH)$, there exists $\delta>0$ such that
$\|F(r_0{\bf Y})-F(r_0{\bf W})\|<\frac{\epsilon}{3}$ for any ${\bf W}\in {\bf B_n}(\cH)^-$ with $\|{\bf W}-{\bf Y}\|<\delta$.
Now, note that
\begin{equation*}
\begin{split}
\|\widetilde F({\bf Y})-\widetilde F({\bf W})\|&\leq
\|\widetilde F({\bf Y})- F(r_0{\bf Y})\|+\|F(r_0{\bf Y})-F(r_0{\bf W})\| + \|F(r_0{\bf W})-\widetilde F({\bf W})\|< \epsilon
\end{split}
\end{equation*}
for any ${\bf W}\in {\bf B_n}(\cH)^-$ with $\|{\bf W}-{\bf Y}\|<\delta$.
The proof is complete.
\end{proof}

\bigskip

\section{Naimark type dilation theorem for  direct products of free semigroups}

In this section, we  provide a Naimark type dilation theorem for  direct products
  $ \FF_{n_1}^+\times \cdots \times \FF_{n_k}^+$  of unital free semigroups, and use it  to obtain a structure theorem which  characterizes the positive free $k$-pluriharmonic functions on the regular polyball, with operator valued coefficients.

Consider  the unital semigroup ${\bf F}^+_{\bf n}:=\FF_{n_1}^+\times \cdots \times \FF_{n_k}^+$  with neutral element ${\bf g}:=(g_0^1,\ldots, g_0^k)$.
Let $\boldsymbol\omega=(\omega_1,\ldots, \omega_k)$ and $\boldsymbol\gamma=(\gamma_1,\ldots, \gamma_k)$ be in $\FF_{n_1}^+\times\cdots \times \FF_{n_k}^+$. We say that $\boldsymbol\omega$ and $\boldsymbol\gamma$ are left comparable, and write  $\boldsymbol\omega\sim_{lc} \boldsymbol\gamma$, if for each $i\in \{1,\ldots, k\}$,  either one of the conditions  $\omega_i<_l \gamma_i$, $\gamma_i<_l \omega_i$, or $\omega_i=\gamma_i$ holds (see the definitions preceding Lemma \ref{inner}).  In this case, we define
\begin{equation*}
c_l^+(\boldsymbol\omega, \boldsymbol\gamma):=(c_l^+(\omega_1, \gamma_1),\ldots, c_l^+(\omega_k, \gamma_k))\quad \text{ and } \quad
c_l^-(\boldsymbol\omega, \boldsymbol\gamma):=(c_l^-(\omega_1, \gamma_1),\ldots, c_l^-(\omega_k, \gamma_k)),
\end{equation*}
where
$$
  c_l^+(\omega, \gamma):=
  \begin{cases}\omega\backslash_l\gamma&; \quad \text{ if } \gamma<_l \omega\\
   g_0&; \quad \text{ if } \omega<_l \gamma \ \text{ or } \ \omega=\gamma
   \end{cases}\quad \text{ and } \quad
   c_l^-(\omega, \gamma):=
  \begin{cases}\gamma\backslash_r\omega&; \quad \text{ if } \omega<_l \gamma\\
   g_0&; \quad \text{ if } \gamma<_l \omega \ \text{ or } \ \omega=\gamma.
   \end{cases}
   $$
 We say that $K:{\bf F}^+_{\bf n}\times {\bf F}^+_{\bf n}\to B(\cE)$  is a
{\it left $k$-multi-Toeplitz kernel} if
 $K({\bf g}, {\bf g})=I_\cE$  and
 $$
 K(\boldsymbol\sigma, \boldsymbol \omega)=\begin{cases}
 K(c_l^+(\boldsymbol\sigma, \boldsymbol \omega);c_l^-(\boldsymbol\sigma, \boldsymbol \omega))&\quad \text{ if } \boldsymbol\sigma\sim_{lc}\boldsymbol \omega\\
 0& \quad \text{otherwise}.
 \end{cases}
 $$
  The kernel $K$ is   positive semi-definite   if
for each $m\in \NN$, any choice of $h_1,\ldots h_m\in \cE$, and any $\boldsymbol\sigma^{(i)}:=(\sigma_1^{(i)},\ldots, \sigma_k^{(i)})\in {\bf F}^+_{\bf n}$, it satisfies the inequality
 $$
 \sum_{i,j=1}^m\left<K(\boldsymbol\sigma^{(i)},\boldsymbol\sigma^{(j)})h_j,h_i\right> \geq 0.
 $$
 \begin{definition}
A map $K:{\bf F}^+_{\bf n}\times {\bf F}^+_{\bf n}\to B(\cE)$ has a Naimark dilation if there exists a
$k$-tuple   of   commuting row isometries ${\bf V}=(V_1,\ldots, V_k)$, $V_i=(V_{i,1},\ldots, V_{i,n_i})$, on  a Hilbert space $\cK\supset \cE$, i.e. the non-selfadjoint algebra $Alg(V_i)$ commutes with $Alg(V_s)$ for any $i,s\in \{1,\ldots, k\}$ with $i\neq s$,  such that
$$K(\boldsymbol\sigma, \boldsymbol\omega)=P_\cE{\bf V}_{\boldsymbol \sigma}^*{\bf V}_{\boldsymbol \omega}|_\cE,\qquad \boldsymbol \sigma, \boldsymbol \omega\in {\bf F}^+_{\bf n}.
 $$
 The dilation is called minimal if \
  $\cK=\bigvee_{\boldsymbol \omega\in {\bf F}^+_{\bf n}} {\bf V}_{\boldsymbol \omega}\cE$.
\end{definition}

\begin{theorem}\label{Naimark} A map $K:{\bf F}^+_{\bf n}\times {\bf F}^+_{\bf n}\to B(\cH)$  is a  positive semi-definite
 left $k$-multi-Toeplitz kernel on the direct product ${\bf F}^+_{\bf n}$ of free semigroups
if and only if it admits a Naimark dilation.
\end{theorem}
\begin{proof}
Let $\cK_0$ be the vector space of all sums of tensor monomials  $\sum_{\boldsymbol\sigma\in {\bf F}_{\bf n}^+} e_{\boldsymbol\sigma}\otimes h_{\boldsymbol\sigma}$, where $\{h_{\boldsymbol\sigma}\}_{\boldsymbol\sigma\in {\bf F}_{\bf n}^+}$ is a finitely supported sequence of vectors in $\cH$.
Define the sesquilinear form $\left<\cdot, \cdot\right>_{\cK_0}$ on $\cK_0$ by setting
$$
\left<\sum_{\boldsymbol\omega\in {\bf F}_{\bf n}^+} e_{\boldsymbol\omega}\otimes h_{\boldsymbol\omega},
\sum_{\boldsymbol\sigma\in {\bf F}_{\bf n}^+} e_{\boldsymbol\sigma}\otimes h_{\boldsymbol\sigma}'\right>_{\cK_0}
:=\sum_{\boldsymbol\omega,\boldsymbol \sigma\in {\bf F}_{\bf n}^+}
\left< K(\boldsymbol \sigma, \boldsymbol\omega)h_{\boldsymbol \omega}, h'_{\boldsymbol\sigma}\right>_\cH,\qquad h_{\boldsymbol \omega}, h'_{\boldsymbol\sigma}\in \cH.
$$
Since $K$ is positive semi-definite, so is $\left<\cdot, \cdot\right>_{\cK_0}$.
Set $\cN:=\{f\in \cK_0:\ \left< f,f\right>=0\}$ and define the Hilbert space obtained by completing $\cK_0/\cN$ with the induced inner product.
For each $i\in\{1,\ldots, k\}$ and $j\in\{1,\ldots, n_i\}$, define the operator
$V_{i,j}$ on $\cK_0$ by setting
$$
V_{i,j}\left(\sum_{\boldsymbol\sigma\in {\bf F}_{\bf n}^+} e_{\boldsymbol\sigma}\otimes h_{\boldsymbol\sigma}\right)
:=\sum_{\boldsymbol\sigma=(\sigma_1,\ldots, \sigma_k)\in  {\bf F}_{\bf n}^+}
e_{\sigma_1}\otimes\cdots \otimes e_{\sigma_{i-1}}\otimes e_{g_j\sigma_i}\otimes e_{\sigma_{i+1}}\otimes \cdots \otimes e_{\sigma_k}\otimes h_{\boldsymbol \sigma}.
$$
Note that if $p\in \{1,\ldots,n_i\}$, then
\begin{equation*}
\begin{split}
&\left<V_{i,j}\left(\sum_{\boldsymbol\omega\in {\bf F}_{\bf n}^+} e_{\boldsymbol\omega}\otimes h_{\boldsymbol\omega}\right),
V_{i,p}\left(\sum_{\boldsymbol\sigma\in {\bf F}_{\bf n}^+} e_{\boldsymbol\sigma}\otimes h'_{\boldsymbol\sigma}\right)\right>_{\cK_0}\\
&\quad=
\sum_{\boldsymbol\omega, \boldsymbol\sigma\in  {\bf F}_{\bf n}^+}
\left<K(\sigma_1,\ldots, \sigma_{i-1},g_j\sigma_i,\sigma_{i+1},\ldots, \sigma_k;
\omega_1,\ldots, \omega_{i-1},g_p\omega_i,\omega_{i+1},\ldots, \omega_k)
h_{\boldsymbol\omega},h'_{\boldsymbol\sigma}\right>_\cH\\
&\quad=\begin{cases}
\sum_{\boldsymbol\omega,\boldsymbol \sigma\in {\bf F}_{\bf n}^+}
\left< K(\boldsymbol \sigma, \boldsymbol\omega)h_{\boldsymbol \omega}, h'_{\boldsymbol\sigma}\right>_\cH & \text{ if } j=p\\
0& \text{otherwise}.
\end{cases}
\end{split}
\end{equation*}
Hence and using the definition of $\left<\cdot, \cdot\right>_{\cK_0}$, we deduce that, for each $i\in \{1,\ldots, k\}$, the operators $V_{i,1},\ldots, V_{i, n_i}$ can be extended by continuity to isometries on $\cK$ with orthogonal ranges.
Note also that if $i,s\in \{1,\ldots, k\}$, $i\neq s$, $j\in \{1,\ldots, n_i\}$ and $t\in \{1,\ldots, n_s\}$, then
$$
V_{i,j}V_{st}(e_{\sigma_1}\otimes\cdots\otimes e_{\sigma_k}\otimes h)=
e_{\sigma_1}\otimes\cdots\otimes e_{\sigma_{i-1}}\otimes e_{g_j\sigma_i}\otimes e_{\sigma_{i+1}}\cdots \otimes e_{\sigma_{s-1}}\otimes e_{\sigma_{g_t \sigma_s}}\otimes e_{\sigma_{s+1}}\otimes \cdots \otimes e_{\sigma_k}\otimes h,
$$
when $i<s$. This shows that $V_{ij}V_{st}=V_{st}V_{ij}$.
Since
$$
\left<e_{\bf g}\otimes h, e_{\bf g}\otimes h'\right>_\cK=\left<K({\bf g}, {\bf g})h,h'\right>_\cH=\left<h,h'\right>_\cH,\qquad h,h'\in \cH,
$$
we can embed $\cH$ into $\cK$ setting $h=e_{\bf g}\otimes h$.
Note that, for any $\boldsymbol\omega,\boldsymbol \sigma\in {\bf F}_{\bf n}^+$ and $h,h'\in \cH$, we have
\begin{equation*}
\begin{split}
\left< V_{\boldsymbol \sigma}^* V_{\boldsymbol \omega} h,h'\right>_{\cK}&=\left< V_{\boldsymbol \omega} h,  V_{\boldsymbol \sigma}h'\right>_{\cK}\\
&=\left<e_{\boldsymbol \omega}\otimes h, e_{\boldsymbol \sigma}\otimes h'\right>_\cK\\
&=\left<K(\boldsymbol \sigma, \boldsymbol \omega)h,h'\right>_\cH.
\end{split}
\end{equation*}
Therefore,
$K(\boldsymbol\sigma, \boldsymbol\omega)=P_\cH{\bf V}_{\boldsymbol \sigma}^*{\bf V}_{\boldsymbol \omega}|_\cH$ for any $\boldsymbol \sigma, \boldsymbol \omega\in {\bf F}^+_{\bf n}$.
Since any element in $\cK_0$ is a linear combination of vectors
${\bf V}_{\boldsymbol \sigma} h$, where  $\boldsymbol \sigma\in {\bf F}_{\bf n}^+$ and $h\in \cH$, we deduce that $\cK=\bigvee_{\boldsymbol \omega\in {\bf F}^+_{\bf n}} {\bf V}_{\boldsymbol \omega}\cH$, which proves the minimality of the Naimark dilation.

Now, we prove  the converse.  Let
 ${\bf V}=(V_1,\ldots, V_n)$ and  $V_i=(V_{i,1},\ldots, V_{i,n_i})$ be  $k$-tuples   of   commuting row isometries on  a Hilbert space $\cK\supset \cH$.
 Define
 $K:{\bf F}^+_{\bf n}\times {\bf F}^+_{\bf n}\to B(\cH)$
  by setting
$K(\boldsymbol\sigma, \boldsymbol\omega)=P_\cH{\bf V}_{\boldsymbol \sigma}^*{\bf V}_{\boldsymbol \omega}|_\cH$ for any $\boldsymbol \sigma, \boldsymbol \omega\in {\bf F}^+_{\bf n}$.
Assume that $\boldsymbol\sigma, \boldsymbol\omega\in {\bf F}^+_{\bf n}$  and $\boldsymbol\sigma\sim_{lc}\boldsymbol\omega$.
 Using the commutativity of the row isometries $V_1,\ldots, V_k$, we can assume without loss of generality that there is $p\in \{1,\ldots, k\}$ such that
 $\omega_1\leq_l\sigma_1,\ldots, \omega_p\leq_l\sigma_p,\sigma_{p+1}\leq_l\omega_{p+1},\ldots, \sigma_k\leq_l\omega.$
 Since each $V_i=(V_{i,1},\ldots, V_{i,n_i})$ is an isometry, we have $V_{i,t}^*V_{i,s}=\delta_{ts}I$.
 Consequently, and using the commutativity of the row isometries, we deduce that
 \begin{equation*}
 \begin{split}
 &\left<V_{1,\omega_1}\cdots V_{k,\omega_k}h, V_{1,\sigma_1}\cdots V_{k,\sigma_k}h'\right>\\
 &=\left<V_{2,\omega_2}\cdots V_{k,\omega_k}h, V_{1,\sigma_1\backslash_l\omega_1}V_{2,\sigma_2}\cdots V_{k,\sigma_k}h'\right>\\
 &=
 \left<V_{2,\omega_2}\cdots V_{k,\omega_k}h, V_{2,\sigma_2}\cdots V_{k,\sigma_k}V_{1,\sigma_1\backslash_l\omega_1}h'\right>\\
 &= \quad \cdots \cdots \cdots \cdots \\
 &=\left< V_{p+1,\omega_{p+1}}\cdots V_{k,\omega_k}h, V_{p+1,\sigma_{p+1}}\cdots V_{k,\sigma_k}V_{1,\sigma_1\backslash_l\omega_1}\cdots V_{p,\sigma_p\backslash_l\omega_p}h'\right>\\
 &=
 \left< V_{p+1,\omega_{p+1}\backslash_l\sigma_{p+1}} V_{p+2,\omega_{p+2}}\cdots V_{k,\omega_k}h, V_{p+2,\sigma_{p+2}}\cdots V_{k,\sigma_k}V_{1,\sigma_1\backslash_l\omega_1}\cdots V_{p,\sigma_p\backslash_l\omega_p}h'\right>\\
 &=
 \left< V_{p+2,\omega_{p+2}}\cdots V_{k,\omega_k} V_{p+1,\omega_{p+1}\backslash_l\sigma_{p+1}} h, V_{p+2,\sigma_{p+2}}\cdots V_{k,\sigma_k}V_{1,\sigma_1\backslash_l\omega_1}\cdots V_{p,\sigma_p\backslash_l\omega_p}h'\right>\\
 &= \quad \cdots \cdots \cdots \cdots \\
 &=
 \left< V_{p+1,\omega_{p+1}\backslash_l\sigma_{p+1}} \cdots V_{k,\omega_k\backslash_l\sigma_{k}} h,  V_{1,\sigma_1\backslash_l\omega_1}\cdots V_{p,\sigma_p\backslash_l\omega_p}h'\right>\\
 &=
 \left< V^*_{1,\sigma_1\backslash_l\omega_1}\cdots V^*_{p,\sigma_p\backslash_l\omega_p}V_{p+1,\omega_{p+1}\backslash_l\sigma_{p+1}} \cdots V_{k,\omega_k\backslash_l\sigma_{k}} h,  h'\right>
 \end{split}
 \end{equation*}
for any $h,h'\in \cH$.
Therefore, for any $\boldsymbol\sigma,\boldsymbol\omega\in {\bf F}_{\bf n}^+$,  we have
 \begin{equation*}
 \begin{split}
 K(\boldsymbol\sigma, \boldsymbol \omega)&=P_\cH{\bf V}_{\boldsymbol \sigma}^*{\bf V}_{\boldsymbol \omega}|_\cH\\
 &=
 \begin{cases}
 P_\cH{\bf V}_{c_l^+(\boldsymbol\sigma, \boldsymbol \omega)}^* {\bf V}_{c_l^-(\boldsymbol\sigma, \boldsymbol \omega)}|_\cH&\quad \text{ if } \boldsymbol\sigma\sim_{lc}\boldsymbol \omega\\
 0& \quad \text{otherwise},
 \end{cases}\\
 &=
 \begin{cases}
 K(c_l^+(\boldsymbol\sigma, \boldsymbol \omega);c_l^-(\boldsymbol\sigma, \boldsymbol \omega))&\quad \text{ if } \boldsymbol\sigma\sim_{lc}\boldsymbol \omega\\
 0& \quad \text{otherwise},
 \end{cases}
 \end{split}
 \end{equation*}
and $K({\bf g}, {\bf g})=I_\cH$.
This shows that $K$ is a left $k$-multi-Toeplitz kernel on  $ {\bf F}_{\bf n}^+$.
On the other hand, for any finitely supported sequence $\{h_{\boldsymbol\omega}\}_{\boldsymbol\omega\in {\bf F}_{\bf n}^+}$ of elements in $\cH$, we have
\begin{equation*}
\begin{split}
\sum_{\boldsymbol\omega,\boldsymbol \sigma\in {\bf F}_{\bf n}^+}
\left< K(\boldsymbol \sigma, \boldsymbol\omega)h_{\boldsymbol \omega}, h_{\boldsymbol\sigma}\right>
&=
\sum_{\boldsymbol\omega,\boldsymbol \sigma\in {\bf F}_{\bf n}^+}
\left< P_\cH{\bf V}_{\boldsymbol \sigma}^*{\bf V}_{\boldsymbol \omega}|_{\cH}h_{\boldsymbol \omega}, h_{\boldsymbol\sigma}\right>\\
&=\left\|\sum_{\boldsymbol \omega\in {\bf F}_{\bf n}^+} {\bf V}_{\boldsymbol \omega}h_{\boldsymbol \omega}\right\|^2\geq 0.
\end{split}
\end{equation*}
Therefore, $K$ is a  positive semi-definite
 left $k$-multi-Toeplitz kernel on ${\bf F}^+_{\bf n}$. The proof is complete.
\end{proof}

We remark that the Naimark dilation provided in Theorem \ref{Naimark} is minimal.
To prove the uniqueness of the minimal Naimark dilation, let
${\bf V'}=(V'_1,\ldots, V'_n)$, $V'_i=(V'_{i,1},\ldots, V'_{i,n_i})$, be a $k$-tuple of   commuting row isometries  on  a Hilbert space $\cK'\supset \cH$ such that
$K(\sigma, \omega)=P_\cH^{\cK'}({\bf V}'_{\boldsymbol \sigma})^*{\bf V}'_{\boldsymbol \omega}|_\cH$ for any $\boldsymbol \sigma, \boldsymbol \omega\in {\bf F}^+_{\bf n}$  and with the property that $\cK=\bigvee_{\boldsymbol \omega\in {\bf F}^+_{\bf n}} {\bf V}'_{\boldsymbol \omega}\cH$.
For any $x,y\in \cH$, we have
\begin{equation*}
\begin{split}
\left<{\bf V}_{\boldsymbol\omega}x,{\bf V}_{\boldsymbol\sigma}y\right>_{\cK}
&=\left<K(\boldsymbol\sigma, \boldsymbol\omega)x,y\right>_\cH\\
&=\left<P_\cH^{\cK'}({\bf V}'_{\boldsymbol \sigma})^*{\bf V}'_{\boldsymbol \omega}x,y\right>_{\cK'}=\left<{\bf V}'_{\boldsymbol\omega}x,{\bf V}'_{\boldsymbol\sigma}y\right>_{\cK'}.
\end{split}
\end{equation*}
Consequently, the map
$$
W\left(\sum_{\boldsymbol\sigma\in {\bf F}_{\bf n}^+}{\bf V}_{\boldsymbol \sigma}h_{\boldsymbol\sigma}\right):=\sum_{\boldsymbol\sigma\in {\bf F}_{\bf n}^+}{\bf V}'_{\boldsymbol \sigma}h_{\boldsymbol\sigma},
$$
 where $\{h_{\boldsymbol\sigma}\}_{\boldsymbol\sigma\in {\bf F}_{\bf n}^+}$ is any finitely supported sequence of vectors in $\cH$,
is well-defined. Due to the minimality of the spaces $\cK$ and $\cK'$, the map  extends to a unitary operator $W$ from $\cK$ onto $\cK'$.
Note also that $WV_{i,j}=V_{i,j}' W$ for any $i\in \{1,\ldots, k\}$ and $j\in \{1,\ldots, n_i\}$, which completes the proof for the uniqueness of the minimal Naimark dilation.

We should mention that there is a dual   Naimark type dilation for positive semi-definite right $k$-multi-Toeplitz kernels.
A  kernel $\Gamma:{\bf F}^+_{\bf n}\times {\bf F}^+_{\bf n} \to B(\cE)$ is called {\it right $k$-multi-Toeplitz} if
 $\Gamma({\bf g}, {\bf g})=I_\cE$  and
 $$
 \Gamma(\boldsymbol\sigma, \boldsymbol \omega)=\begin{cases}
 \Gamma(c_r^+(\boldsymbol\sigma, \boldsymbol \omega);c_r^-(\boldsymbol\sigma, \boldsymbol \omega))&\quad \text{ if } \boldsymbol\sigma\sim_{rc}\boldsymbol \omega\\
 0& \quad \text{otherwise},
 \end{cases}
 $$
where $c_r^+(\boldsymbol\sigma, \boldsymbol \omega),c_r^-(\boldsymbol\sigma, \boldsymbol \omega)$ are defined by relation \eqref{c+}.
We say that $\Gamma$ has a Naimark dilation if there exists a
$k$-tuple   ${\bf W}=(W_1,\ldots, W_n)$, $W_i=(W_{i,1},\ldots, W_{i,n_i})$, of   commuting row isometries on  a Hilbert space $\cK\supset \cE$ such that
$\Gamma(\widetilde{\boldsymbol\sigma}, \widetilde{\boldsymbol\omega})=P_\cE{\bf W}_{\boldsymbol \sigma}^*{\bf W}_{\boldsymbol \omega}|_\cE$ for any $\boldsymbol \sigma, \boldsymbol \omega\in {\bf F}^+_{\bf n}$.

\begin{theorem} \label{Naimark2}
A map $\Gamma:{\bf F}^+_{\bf n}\times {\bf F}^+_{\bf n}\to B(\cH)$  is a  positive semi-definite
 right $k$-multi-Toeplitz kernel on ${\bf F}^+_{\bf n}$
if and only if it admits a   Naimark dilation. In this case, there is a  minimal dilation which is uniquely  determined  up to an isomorphism.
\end{theorem}
\begin{proof} We only sketch the proof which is very similar to that of Theorem \ref{Naimark}, pointing out the differences.
First, $\cK_0$ is  the vector space of all sums of tensor monomials  $\sum_{\boldsymbol\sigma\in {\bf F}_{\bf n}^+} e_{\widetilde{\boldsymbol\sigma}}\otimes h_{\boldsymbol\sigma}$, where $\{h_{\boldsymbol\sigma}\}_{\boldsymbol\sigma\in {\bf F}_{\bf n}^+}$ is a finitely supported sequence of vectors in $\cH$, while
the sesquilinear form $\left<\cdot, \cdot\right>_{\cK_0}$ on $\cK_0$  is define by setting
$$
\left<\sum_{\boldsymbol\omega\in {\bf F}_{\bf n}^+} e_{\widetilde{\boldsymbol\omega}}\otimes h_{\boldsymbol\omega},
\sum_{\boldsymbol\sigma\in {\bf F}_{\bf n}^+} e_{\widetilde{\boldsymbol\sigma}}\otimes h_{\boldsymbol\sigma}'\right>_{\cK_0}
:=\sum_{\boldsymbol\omega,\boldsymbol \sigma\in {\bf F}_{\bf n}^+}
\left< \Gamma(\boldsymbol \sigma, \boldsymbol\omega)h_{\boldsymbol \omega}, h'_{\boldsymbol\sigma}\right>_\cH.
$$
For each $i\in\{1,\ldots, k\}$ and $j\in\{1,\ldots, n_i\}$,  we define the operator
$W_{i,j}$ on $\cK_0$ by setting
$$
W_{i,j}\left(\sum_{\boldsymbol\sigma\in {\bf F}_{\bf n}^+} e_{\widetilde{\boldsymbol\sigma}}\otimes h_{\boldsymbol\sigma}\right)
:=\sum_{\boldsymbol\sigma=(\sigma_1,\ldots, \sigma_k)\in  {\bf F}_{\bf n}^+}
e_{\widetilde{\sigma}_1}\otimes\cdots \otimes e_{\widetilde{\sigma}_{i-1}}\otimes e_{g_j\widetilde{\sigma}_i}\otimes e_{\widetilde{\sigma}_{i+1}}\otimes \cdots \otimes e_{\widetilde{\sigma}_k}\otimes h_{\boldsymbol \sigma}.
$$
Taking into account the   relations
$$
\widetilde{c_r^+(\boldsymbol\sigma, \boldsymbol \omega)}=
c_l^+(\widetilde{\boldsymbol\sigma}, \widetilde{\boldsymbol \omega})\quad \text{ and } \quad
\widetilde{c_r^-(\boldsymbol\sigma, \boldsymbol \omega)}=
c_l^-(\widetilde{\boldsymbol\sigma}, \widetilde{\boldsymbol \omega}),
$$
we deduce that
\begin{equation*}
 \begin{split}
 P_\cH{\bf W}_{\boldsymbol \sigma}^*{\bf W}_{\boldsymbol \omega}|_\cH
 &=
 \begin{cases}
 P_\cH{\bf W}_{c_l^+(\boldsymbol\sigma, \boldsymbol \omega)}^* {\bf W}_{c_l^-(\boldsymbol\sigma, \boldsymbol \omega)}|_\cH &\quad \text{ if } \boldsymbol\sigma\sim_{lc}\boldsymbol \omega\\
 0& \quad \text{ otherwise}
 \end{cases}\\
 &=
 \begin{cases}
 \Gamma(\widetilde{c_l^+(\boldsymbol\sigma, \boldsymbol \omega)};\widetilde{c_l^-(\boldsymbol\sigma, \boldsymbol \omega))}&\quad \text{ if } \boldsymbol\sigma\sim_{lc}\boldsymbol \omega\\
 0& \quad \text{ otherwise}
 \end{cases}\\
 &=
 \begin{cases}
 \Gamma({c_r^+(\widetilde{\boldsymbol\sigma}, \widetilde{\boldsymbol \omega})};{c_r^-(\widetilde{\boldsymbol\sigma}, \widetilde{\boldsymbol \omega}))}&\quad \text{ if } \widetilde{\boldsymbol\sigma}\sim_{rc}\widetilde{\boldsymbol \omega}\\
 0& \quad \text{ otherwise}
 \end{cases}\\
 &=\Gamma(\widetilde{\boldsymbol\sigma},\widetilde{\boldsymbol\omega})
 \end{split}
 \end{equation*}
 for any $\boldsymbol \sigma, \boldsymbol \omega\in {\bf F}^+_{\bf n}$.
The rest of the proof is similar to that  of  Theorem \ref{Naimark}.  We leave it to the reader.
 \end{proof}

Let   $F$ be a free $k$-pluriharmonic on the polyball ${\bf B_n}$ with  operator-valued coefficients in $B(\cE)$      with representation
\begin{equation} \label{repre}
F({\bf X})= \sum_{m_1\in \ZZ}\cdots \sum_{m_k\in \ZZ} \sum_{{\alpha_i,\beta_i\in \FF_{n_i}^+, i\in \{1,\ldots, k\}}\atop{|\alpha_i|=m_i^-, |\beta_i|=m_i^+}}A_{(\alpha_1,\ldots,\alpha_k;\beta_1,\ldots, \beta_k)}
\otimes {\bf X}_{1,\alpha_1}\cdots {\bf X}_{k,\alpha_k}{\bf X}_{1,\beta_1}^*\cdots {\bf X}_{k,\beta_k}^*,
\end{equation}
where the series converge in the operator norm topology for any  ${\bf X}=(X_1,\ldots, X_k)\in {\bf B_n}(\cH)$, with $X_i:=(X_{i,1},\ldots, X_{i,n_i})$, and any Hilbert space $\cH$.
We associate with $F$ the kernel
$\Gamma_F:{\bf F}^+_{\bf n}\times {\bf F}^+_{\bf n} \to B(\cE)$  given by
\begin{equation}
\label{right-To}
 \Gamma_F(\boldsymbol\sigma, \boldsymbol \omega):=\begin{cases}
 A_{(c_r^+(\boldsymbol\sigma, \boldsymbol \omega);c_r^-(\boldsymbol\sigma, \boldsymbol \omega)})&\quad \text{ if } \boldsymbol\sigma\sim_{rc}\boldsymbol \omega\\
 0& \quad \text{otherwise}.
 \end{cases}
 \end{equation}
 One can easily see that $\Gamma_F$ is a right $k$-multi-Toeplitz kernel
 on  ${\bf F}_{\bf n}^+$.
 In what follows, we prove a Schur type results for positive $k$-pluriharmonic function in polyballs.

 \begin{theorem}\label{pluri-positive} Let $F$ be a $k$-pluriharmonic function on the regular polyball ${\bf B_n}$, with coefficients in $B(\cE)$. Then $F$ is positive on ${\bf B_n}$ if and only if the kernel $\Gamma_{F_r}$ is positive semi-definite for any $r\in [0,1)$, where $F_r$ stands for  the mapping ${\bf X}\mapsto F(r{\bf X})$.
 \end{theorem}
 \begin{proof} For every  $\boldsymbol \gamma:=(\gamma_1,\ldots, \gamma_k)\in \FF_{n_1}^+\times \cdots \times \FF_{n_k}^+$, we set  $e_{\boldsymbol\gamma}:=e_{\gamma_1}^1\otimes \cdots \otimes e_{\gamma_k}^k$  and
 ${\bf S}_{\boldsymbol\gamma}:={\bf S}_{1,\gamma_1}\cdots {\bf S}_{k,\gamma_k}$. Let $F$ be  a $k$-pluriharmonic function with representation \eqref{repre}.
  Taking into account Lemma \ref{inner}, for each $\boldsymbol \gamma, \boldsymbol \omega \in \FF_{n_1}^+\times \cdots \times \FF_{n_k}^+$,   $r\in [0,1)$, and $h,h'\in \cE$, we have
 \begin{equation*}
 \begin{split}
 &\left<F(r{\bf S})(h\otimes e_{\boldsymbol\gamma}), h'\otimes e_{\boldsymbol\omega}\right>\\
 &\qquad =  \sum_{m_1\in \ZZ}\cdots \sum_{m_k\in \ZZ} \sum_{{\alpha_i,\beta_i\in \FF_{n_i}^+, i\in \{1,\ldots, k\}}\atop{|\alpha_i|=m_i^-, |\beta_i|=m_i^+}}
 \left<A_{(\alpha_1,\ldots,\alpha_k;\beta_1,\ldots, \beta_k)}h,h'\right>
 r^{\sum_{i=1}^k (|\alpha_i|+|\beta_i|} \left<{\bf S}_{\boldsymbol \alpha}
 {\bf S}_{\boldsymbol \beta}^* e_{\boldsymbol\gamma},e_{\boldsymbol\omega}\right>\\
 &\qquad =
 \begin{cases}
 r^{\sum_{i=1}^k (|\alpha_i|+|\beta_i|} \left<A_{(c_r^+(\boldsymbol\omega, \boldsymbol \gamma); c_r^-(\boldsymbol\omega, \boldsymbol \gamma)})h, h'\right>&\quad \text{ if } \boldsymbol\omega\sim_{rc}\boldsymbol \gamma\\
 0& \quad \text{otherwise}
 \end{cases}
 \\
 &\qquad=
 \left<\Gamma_{F_r}(\boldsymbol\omega, \boldsymbol \gamma)h,h'\right>.
\end{split}
 \end{equation*}
 Hence, we  deduce that the kernel $\Gamma_{F_r}$ is positive semi-definite for any $r\in [0,1)$ if and only if
 $F(r{\bf S}) \geq 0$ for any $r\in[0,1)$.
 Now,  let ${\bf X}\in {\bf B_n}(\cH)$ and let $r\in (0,1)$ be such that $\frac{1}{r}{\bf X}\in {\bf B_n}(\cH)$.
Since the noncommutative Berezin transform
$\boldsymbol\cB_{\frac{1}{r}{\bf X}}$ is continuous in the operator norm and completely positive, so is $id\otimes \boldsymbol\cB_{\frac{1}{r}{\bf X}}$. Consequently,  we obtain
$$
F({\bf X})=(id\otimes \boldsymbol\cB_{\frac{1}{r}{\bf X}})[F(r{\bf S})]\geq 0, \qquad {\bf X}\in {\bf B_n}(\cH).
$$
Note that if $F$ is positive on ${\bf B_n}$, then $F(r{\bf S})\geq 0$ for any $r\in [0,1)$.
 This  completes the proof.
 \end{proof}

\begin{corollary} Let $f:{\bf B_n}(\cH)\to B(\cE)\otimes_{min}B(\cH)$ be a free holomorphic function.  Then  the following statements are equivalent.
\begin{enumerate}
\item[(i)]$\Re f\geq 0$ on the polyball ${\bf B_n}$;
 \item[(ii)] $\Re f(r{\bf S})\geq 0$ for any $r\in [0,1)$;
 \item[(iii)] the right $k$-multi Toeplitz kernel
  $\Gamma_{\Re f_r}$ is positive semidefinite   for any $r\in [0,1)$.
  \end{enumerate}
\end{corollary}

 Let us  define  the {\it free $k$-pluriharmonic Poisson kernel}     by setting
\begin{equation*}
 {\boldsymbol\cP}({\bf Y}, {\bf X}):=\sum_{m_1\in \ZZ}\cdots \sum_{m_k\in \ZZ} \sum_{{\alpha_i,\beta_i\in \FF_{n_i}^+, i\in \{1,\ldots, k\}}\atop{|\alpha_i|=m_i^-, |\beta_i|=m_i^+}} { Y}_{1,\widetilde\alpha_1}^*\cdots { Y}_{k,\widetilde\alpha_k}^*{Y}_{1,\widetilde\beta_1}\cdots { Y}_{k,\widetilde\beta_k}
\otimes { X}_{1,\alpha_1}\cdots {X}_{k,\alpha_k}{X}_{1,\beta_1}^*\cdots {X}_{k,\beta_k}^*
\end{equation*}
for any ${\bf X}\in {\bf B_n}(\cH)$ and any ${\bf Y}=(Y_1,\ldots, Y_k)$ with $Y_i=(Y_{i,1},\ldots,Y_{i,n_i})\in B(\cK)^{n_i}$ such that the series above is convergent
in the operator norm topology.

 \begin{theorem}\label{pluri-structure} A map $F:{\bf B_n}(\cH)\to B(\cE)\otimes_{min} B(\cH)$, with $F(0)=I$,   is a positive free $k$-pluriharmonic function on the regular  polyball if and only if   it has the form
 $$
 F({\bf X})=
 \sum_{(\boldsymbol \alpha, \boldsymbol \beta)\in {\bf F}_{\bf n}^+\times {\bf F}_{\bf n}^+}
 P_\cE{\bf V}_{\widetilde{\boldsymbol\alpha}}^*{\bf V}_{\widetilde{\boldsymbol\beta}}|_\cE\otimes
 {\bf X}_{\boldsymbol \alpha} {\bf X}_{\boldsymbol\beta}^*,
 $$
 where ${\bf V}=(V_1,\ldots, V_k)$ is a
 $k$-tuple   of commuting row isometries on a space $\cK\supset \cE$ such that
  $$
  \sum_{(\boldsymbol \alpha, \boldsymbol \beta)\in {\bf F}_{\bf n}^+\times {\bf F}_{\bf n}^+}
 P_\cE{\bf V}_{\widetilde{\boldsymbol\alpha}}^*{\bf V}_{\widetilde{\boldsymbol\beta}}|_\cE\otimes r^{|\boldsymbol\alpha|+|\boldsymbol\beta|}
 {\bf S}_{\boldsymbol \alpha} {\bf S}_{\boldsymbol\beta}^*\geq 0,\qquad r\in [0,1),
  $$
and the series is convergent in the operator topology.
  \end{theorem}
\begin{proof} Assume that $F$ is a positive free $k$-pluriharmonic function which has the representation \eqref{repre} and $F(0)=I$. Due to Theorem
\ref{pluri-positive}, $F(r{\bf S})\geq 0$ and the right $k$-multi-Toeplitz kernel $\Gamma_{F_r}$ is positive semi-definite for any $r\in [0,1)$.
Taking limits as $r\to \infty$, we deduce that $\Gamma_F$ is positive semi-definite as well. According to Theorem \ref{Naimark2}, $\Gamma_F$ has a Naimark type dilation. Therefore,  there is a $k$-tuple ${\bf V}=(V_1,\ldots, V_k)$ of commuting row isometries on a Hilbert space $\cK\supset\cE$ such that $\Gamma(\widetilde{\boldsymbol\sigma}, \widetilde{\boldsymbol\omega})=P_\cE{\bf V}_{\boldsymbol \sigma}^*{\bf V}_{\boldsymbol \omega}|_\cE$ for any $\boldsymbol \sigma, \boldsymbol \omega\in {\bf F}^+_{\bf n}$. Using relations
\eqref{repre} and \eqref{right-To}, we deduce that
$$
 F({\bf X})=
 \sum_{(\boldsymbol \alpha, \boldsymbol \beta)\in {\bf F}_{\bf n}^+\times {\bf F}_{\bf n}^+}
 P_\cE{\bf V}_{\widetilde{\boldsymbol\alpha}}^*{\bf V}_{\widetilde{\boldsymbol\beta}}|_\cE\otimes
 {\bf X}_{\boldsymbol \alpha} {\bf X}_{\boldsymbol\beta}^*,
 $$
where the convergence is in the norm topology. This shows, in particular, that $F(r{\bf S})$ is convergent.

To prove the converse, assume that ${\bf V}=(V_1,\ldots, V_k)$ is a
 $k$-tuple   of commuting row isometries on a space $\cK\supset \cE$ such that
  \begin{equation}
  \label{posi}
  \sum_{(\boldsymbol \alpha, \boldsymbol \beta)\in {\bf F}_{\bf n}^+\times {\bf F}_{\bf n}^+}
 P_\cE{\bf V}_{\widetilde{\boldsymbol\alpha}}^*{\bf V}_{\widetilde{\boldsymbol\beta}}|_\cE\otimes r^{|\boldsymbol\alpha|+|\boldsymbol\beta|}
 {\bf S}_{\boldsymbol \alpha} {\bf S}_{\boldsymbol\beta}^*\geq 0,\qquad r\in [0,1),
  \end{equation}
and the convergence is in the operator norm topology. Let ${\bf X}\in {\bf B_n}(\cH)$ and let $r\in (0,1)$ be such that $\frac{1}{r}{\bf X}\in {\bf B_n}(\cH)$.
Since the noncommutative Berezin transform
$\boldsymbol\cB_{\frac{1}{r}{\bf X}}$ is continuous in the operator norm and completely positive, so is $id\otimes \boldsymbol\cB_{\frac{1}{r}{\bf X}}$. Consequently,  we obtain
$$
F({\bf X}):=(id\otimes \boldsymbol\cB_{\frac{1}{r}{\bf X}})\left[\sum_{(\boldsymbol \alpha, \boldsymbol \beta)\in {\bf F}_{\bf n}^+\times {\bf F}_{\bf n}^+}
 P_\cE{\bf V}_{\widetilde{\boldsymbol\alpha}}^*{\bf V}_{\widetilde{\boldsymbol\beta}}|_\cE\otimes r^{|\boldsymbol\alpha|+|\boldsymbol\beta|}
 {\bf S}_{\boldsymbol \alpha} {\bf S}_{\boldsymbol\beta}^*\right]\geq 0, \qquad {\bf X}\in {\bf B_n}(\cH).
$$
This completes the proof.
\end{proof}

We remark that condition \eqref{posi} is equivalent to the condition that
the kernel  defined by relation $\Gamma_{r{\bf V}}(\boldsymbol\sigma, \boldsymbol \omega):=
r^{|\boldsymbol\sigma|+|\boldsymbol\omega|}P_\cE{\bf V}_{\boldsymbol \sigma}^*{\bf V}_{\boldsymbol \omega}|_\cE$, for any $\boldsymbol \sigma, \boldsymbol \omega\in {\bf F}^+_{\bf n}$, is positive semi-definite.
We should also  mention that one can find a version of the theorem above when the condition $F(0)=I$ is dropped. In this  case,  $F(0)=A\otimes I$ with $A\geq 0$ and  we set
$$
G_\epsilon:= [(A+\epsilon I_\cE)^{-1/2}\otimes I] (F+\epsilon I_\cE\otimes I)[(A+\epsilon I_\cE)^{-1/2}\otimes I],\qquad \epsilon >0.
$$
Since $G_\epsilon$ is a positive $k$-pluriharmonic function with $G_\epsilon(0)=I$,  we can apply Theorem \ref{pluri-structure} to
get a family   ${\bf V}(\epsilon)=(V_1(\epsilon),\ldots, V_k(\epsilon))$   of
 $k$-tuples   of commuting row isometries on a space $\cK_\epsilon\supset \cE$
 such that
 $$
 F({\bf X})=\lim_{\epsilon\to 0}
 \sum_{(\boldsymbol \alpha, \boldsymbol \beta)\in {\bf F}_{\bf n}^+\times {\bf F}_{\bf n}^+}(A+\epsilon I_\cE)^{1/2}[
 P_\cE{\bf V}_{\widetilde{\boldsymbol\alpha}}^*(\epsilon){\bf V}_{\widetilde{\boldsymbol\beta}}(\epsilon)|_\cE](A+\epsilon I_\cE)^{1/2}\otimes
 {\bf X}_{\boldsymbol \alpha} {\bf X}_{\boldsymbol\beta}^*,
 $$
  where the convergence is in the operator norm topology.

\begin{definition}
A $k$-tuple ${\bf V}=(V_1,\ldots,V_k)$ of commuting row isometries $V_i=(V_{i,1},\ldots, V_{i,n_i})$ is called {\it pluriharmonic} if the free $k$-pluriharmonic Poisson kernel  \  $\boldsymbol \cP({\bf V}, r{\bf S})$ is a   positive operator, for any $r\in [0,1)$.
\end{definition}

\begin{proposition} \label{part-case} Let  ${\bf V}=(V_1,\ldots,V_k)$, $V_i=(V_{i,1},\ldots, V_{i,n_i})$,  be a $k$-tuple of commuting row isometries. Then ${\bf V}$ is pluriharmonic
in each of  the following particular cases:
\begin{enumerate}
\item[(i)] if $k=1$ and  $n_1\in \NN$;
\item[(ii)] if ${\bf V}$ is doubly commuting, i.e. the $C^*$-algebra $C^*(V_i)$ commutes with $C^*(V_s)$ if \  $i,s\in\{1,\ldots, k\}$ with $i\neq s$;

\item[(iii)] if $n_1=\cdots=n_k=1$.

\end{enumerate}

\end{proposition}
\begin{proof} It is easy to see  that ${\bf V}$ is pluriharmonic if the condition in item (i) is satisfied. Under the condition (ii), the proof that  ${\bf V}$ is pluriharmonic is similar to the proof of item (i)  from Theorem \ref{Poisson-factor}, when we replace the universal operator ${\bf R}$ with ${\bf V}$. Now, we assume that $n_1=\cdots=n_k=1$. Then ${\bf V}=(V_1,\ldots, V_k)$, where $V_1,\ldots, V_k$ are commuting isometries on a Hilbert space $\cK$.
It is well-known \cite{SzFBK-book} that  there is a $k$-tuple ${\bf U}=(U_1,\ldots, U_k)$ of commuting unitaries on a Hilbert space $\cG\supset \cK$ such that $U_i|_\cK=V_i$ for $i\in \{1,\ldots, k\}$. Due to Fuglede's theorem (see \cite{Dou}), the unitaries are doubly commuting. Due to item (ii),
$\boldsymbol \cP({\bf U}, r{\bf S})$ is a well-defined  positive operator, for any $r\in [0,1)$, where the convergence defining the  free $k$-pluriharmonic  Poisson kernel $\boldsymbol \cP({\bf U}, r{\bf S})$ is in the operator norm topology. On the other hand, we have
$$
\boldsymbol \cP({\bf V}, r{\bf S})=(P_\cK\otimes I)\boldsymbol \cP({\bf U}, r{\bf S})|_{\cK\otimes_{i=1}^kF^2(H_{n_i})}\geq 0,
$$
which completes the proof.
\end{proof}

\begin{proposition}\label{pos-pluri} Let ${\bf V}=(V_1,\ldots,V_k)$ be  a pluriharmonic tuple of  commuting row isometries on a Hilbert space $\cK$ and let $\cE\subset \cK$ be a subspace. Then the map
$$
 F({\bf X}):=(P_\cE\otimes I)\boldsymbol\cP\left({\bf V}, {\bf X}\right)|_{\cE\otimes \cH}, \qquad {\bf X}\in {\bf B_n}(\cH)
 $$
is a  positive free $k$-pluriharmonic function on the polyball ${\bf B_n}$ with operator-valued coefficients in $B(\cE)$, and $F(0)=I$.

Moreover, in the particular cases  (i) and (iii) of Proposition \ref{part-case}, each positive  free $k$-pluriharmonic function $F$ with $F(0)=I$  has the form above.
\end{proposition}
\begin{proof} Since ${\bf V}$ is a  tuple of  commuting row isometries, the  free $k$-pluriharmonic  Poisson kernel $\boldsymbol \cP({\bf V}, r{\bf S})$  is a   positive operator for any  $r\in [0,1)$ and  so is
the compression $(P_\cE\otimes I)\boldsymbol \cP({\bf V}, r{\bf S})|_{\cE\otimes_{i=1}^k F^2(H_{n_i})}$.
Let ${\bf X}\in {\bf B_n}(\cH)$ and let $r\in (0,1)$ be such that $\frac{1}{r}{\bf X}\in {\bf B_n}(\cH)$.
Since the noncommutative Berezin transform $id\otimes \boldsymbol\cB_{\frac{1}{r}{\bf X}}$
 is continuous in the operator norm and completely positive, we deduce that
 $$
 F({\bf X}):=(P_\cE\otimes I)\boldsymbol \cP({\bf V}, {\bf X})|_{\cE\otimes \cH}\geq 0,\qquad {\bf X}\in {\bf B_n}(\cH),
 $$
 where the convergence of $\cP({\bf V}, {\bf X})$ is in the operator norm topology. Therefore, $F$  is a  positive free $k$-pluriharmonic function on the polyball ${\bf B_n}$ with operator-valued coefficients in $B(\cE)$, and $F(0)=I$.
 To prove that second part of this proposition, assume that $F$ is a  positive  free $k$-pluriharmonic function with $F(0)=I$.
 According to Theorem \ref{pluri-structure}, $F$ has the form
 $$
 F({\bf X})=
 \sum_{(\boldsymbol \alpha, \boldsymbol \beta)\in {\bf F}_{\bf n}^+\times {\bf F}_{\bf n}^+}
 P_\cE{\bf V}_{\widetilde{\boldsymbol\alpha}}^*{\bf V}_{\widetilde{\boldsymbol\beta}}|_\cE\otimes
 {\bf X}_{\boldsymbol \alpha} {\bf X}_{\boldsymbol\beta}^*,
 $$
 where ${\bf V}=(V_1,\ldots, V_k)$ is a
 $k$-tuple   of commuting row isometries on a space $\cK\supset \cE$  and the convergence of the series is in the operator norm topology. Since, in the particular cases (i) and (ii) of Proposition \ref{part-case}, ${\bf V}$ is pluriharmonic, one can easily complete the proof.
\end{proof}
We remark that the theorem above contains, in particular,    a structure theorem for  positive  $k$-harmonic functions on the regular  polydisk  included in $[B(\cH)]_1\times_c\cdots \times_c[B(\cH)]_1$,
    which extends the corresponding  classical result on scalar polydisks \cite{Ru1}. In the general case of the polyball it is unknown if all positive free $k$-pluriharmonic functions $F$ with $F(0)=I$ have the form of Proposition \ref{pos-pluri}.

\bigskip

\section{Berezin transforms of completely bounded maps in regular polyballs}

 We define  a class of   noncommutative Berezin transforms of   completely bounded linear maps   and give necessary an sufficient conditions for a function to be the Poisson transform of a completely bounded (resp. positive) map.

  Let $\cH$ be a Hilbert space and  identify $M_m(B(\cH))$, the set of
$m\times m$ matrices with entries from $B(\cH)$, with
$B( \cH^{(m)})$, where $\cH^{(m)}$ is the direct sum of $m$ copies
of $\cH$.
Thus we have a natural $C^*$-norm on
$M_m(B(\cH))$. If $\cX$ is an operator space, i.e. a closed
subspace of $B(\cH)$, we consider $M_m(\cX)$ as a subspace of
$M_m(B(\cH))$ with the induced norm.
Let $\cX, \cY$ be operator spaces and $u:\cX\to \cY$ be a linear
map. Define the map
$u_m:M_m(\cX)\to M_m(\cY)$ by
$ u_m ([x_{ij}]):=[u(x_{ij})]. $
We say that $u$ is completely bounded  if
$ \|u\|_{cb}:=\sup_{m\ge1}\|u_m\|<\infty. $
If $\|u\|_{cb}\leq1$
(resp. $u_m$ is an isometry for any $m\geq1$) then $u$ is completely
contractive (resp. isometric),
and if $u_m$ is positive for all $m$, then $u$ is called
 completely positive. For basic results concerning  completely bounded maps
 and operator spaces we refer to \cite{Pa-book}, \cite{Pi-book}, and \cite{ER}.

Let $\cK$ be a Hilbert space and let  $\mu: B(\otimes_{i=1}^kF^2(H_{n_i}))\to B(\cK)$
be a completely bounded map. It is well-known (see e.g. \cite{Pa-book})
that  there exists a completely bounded linear map
$$\widehat
\mu:=\mu\otimes \text{\rm id} : B(\otimes_{i=1}^kF^2(H_{n_i})) \otimes_{min} B(\cH)\to
B(\cK)\otimes_{min} B(\cH)
$$
such that $ \widehat \mu(f\otimes Y):= \mu(f)\otimes Y$ for $f\in
B(\otimes_{i=1}^kF^2(H_{n_i})) $ and  $Y\in B(\cH)$. Moreover, $\|\widehat
\mu\|_{cb}=\|\mu\|_{cb}$ and, if $\mu$ is completely positive, then
so is $\widehat \mu$.
 We introduce  the   {\it
noncommutative Berezin transform} associated with $\mu$ as the map
$$\boldsymbol\cB_\mu:
B(\otimes_{i=1}^kF^2(H_{n_i}))\times  {\bf B}_{\bf n}(\cH)\to B(\cK)\otimes_{min} B(\cH)$$
defined by
\begin{equation*}
 \boldsymbol\cB_\mu(A,{\bf X}):=\widehat{\mu}\left[ C_{\bf X}^*(A\otimes
I_\cH)C_{\bf X}\right], \qquad A\in B(\otimes_{i=1}^kF^2(H_{n_i})), \ {\bf X}\in {\bf B}_{\bf n}(\cH),
\end{equation*}
  where the operator  $C_{\bf X} \in B(\otimes_{i=1}^kF^2(H_{n_i})\otimes
\cH)$  is defined
  by
\begin{equation*}
 C_{\bf X} := (I_{\otimes_{i=1}^kF^2(H_{n_i})} \otimes \boldsymbol\Delta_{\bf X}(I)^{1/2})\prod_{i=1}^k
 \left(I-{\bf R}_{i,1}\otimes X_{i,1}^*-\cdots -{\bf R}_{i,n_i}\otimes X_{i,n_i}^*\right)^{-1}
\end{equation*}
and
 the {\it defect mapping} ${\bf \Delta_{X}}:B(\cH)\to  B(\cH)$ is given by
$$
{\bf \Delta_{X}}:=\left(id -\Phi_{X_1}\right)\circ \cdots \circ\left(id -\Phi_{ X_k}\right),
$$
 where
$\Phi_{X_i}:B(\cH)\to B(\cH)$  is the completely positive linear map defined by
$$\Phi_{X_i}(Y):=\sum_{j=1}^{n_i}   X_{i,j} Y X_{i,j} ^*, \qquad Y\in B(\cH).
$$
We need to show that the operator $I-{\bf R}_{i,1}\otimes X_{i,1}^*-\cdots -{\bf R}_{i,n_i}\otimes X_{i,n_i}^*$ is invertible.
Let  ${\bf Y}=({ Y}_1,\ldots, { Y}_k) $ with $Y_i:=(Y_{i,1},\ldots, Y_{i,n_i})\in B(\cH)^{n_i}$. We introduce the spectral radius of ${\bf Y}$ by setting
$$
 r({\bf Y}):=\limsup_{(p_1,\ldots, p_k)\in \ZZ_+^k} \|\sum\limits_{{ \alpha_i\in \FF_{n_i}^+, |\alpha_i|=p_i}\atop{i\in \{1,\ldots, k\} }}Y_{\boldsymbol\alpha}Y_{\boldsymbol\alpha}^*\|^{\frac{1}{2(p_1+\cdots +p_k)}},
 $$
where  ${Y}_{\boldsymbol\alpha}:= {Y}_{1,\alpha_1}\cdots {Y}_{k,\alpha_k}$ if  $\boldsymbol\alpha:=(\alpha_1,\ldots, \alpha_k)\in \FF_{n_1}^+\times \cdots \times\FF_{n_k}^+$ and   $Y_{i,\alpha_i}:=Y_{i,j_1}\cdots Y_{i,j_p}$
  if   $\alpha_i=g_{j_1}^i\cdots g_{j_p}^i\in \FF_{n_i}^+$. We remark that, when $k=1$, we recover the spectral radius of an $n_i$-tuple of operators, i.e. $r(Y_i)=\lim_{p\to \infty} \|\sum_{\beta_i\in \FF_{n_i}, |\beta_i|=p}Y_{i,\beta_i} Y_{i,\beta_i}^*\|^{1/2p}$.
  Note also that
  $$
  r(Y_i)=r({\bf R}_{i,1}\otimes Y_{i,1}^*+\cdots +{\bf R}_{i,n_i}\otimes Y_{i,n_i}^*)
  $$
and $ r(Y_i)\leq r({\bf Y})$ for any $i\in \{1,\ldots, k\}$.
Consequently, if $r({\bf Y})<1$, then $ r(Y_i)<1$ and     the spectrum of
${\bf R}_{i,1}\otimes Y_{i,1}^*+\cdots +{\bf R}_{i,n_i}\otimes Y_{i,n_i}^*$ is included in $ \DD:=\{z\in \CC:\ |z|<1\}$.  In particular, when ${\bf X}\in {\bf B}_{\bf n}(\cH)$, the noncommutative von Neumann inequality  \cite{Po-poisson}  implies   $r({\bf X})\leq r(t{\bf S})=t$ for some $t\in (0,1)$, which proves our assertion.

\begin{proposition}\label{berezin}
Let \ $\boldsymbol\cB_\mu$ be the  noncommutative Berezin transform  associated
with a completely bounded linear map $\mu: B(\otimes_{i=1}^kF^2(H_{n_i}))\to B(\cK)$.
\begin{enumerate}
\item[(i)] If ${\bf X}\in {\bf B}_{\bf n}(\cH)$ is fixed, then
$$\boldsymbol\cB_\mu(\cdot, {\bf X}):B(\otimes_{i=1}^kF^2(H_{n_i}))\to B(\cK)\otimes_{min} B(\cH)$$ is a
completely bounded linear map with $ \|\boldsymbol\cB_\mu(\cdot, {\bf X})\|_{cb}\leq
\|\mu\|_{cb}  \|C_{\bf X}\|^2. $
\item[(ii)] If $\mu$ is selfadjoint, then
$\boldsymbol\cB_\mu(A^*,{\bf X})=\boldsymbol\cB_\mu(A,{\bf X})^*. $
 Moreover, if $\mu$ is completely positive, then so is the map \ $\boldsymbol\cB_\mu(\cdot,
 {\bf X})$.
\item[(iii)] If $A\in B(\otimes_{i=1}^kF^2(H_{n_i}))$ is fixed, then the map
$$\boldsymbol\cB_\mu(A, \cdot):   {\bf B}_{\bf n}(\cH) \to B(\cK)\otimes_{min} B(\cH)$$ is
continuous  and $\|\boldsymbol\cB_\mu(A,{\bf X})\|\leq \|\mu\|_{cb}\|A\| \|C_{\bf X}\|^2 $
for any    $ {\bf X}\in  {\bf B}_{\bf n}(\cH)$.
\end{enumerate}
\end{proposition}
\begin{proof}
The items  (i) and (ii) follow easily from the definition of the
noncommutative Berezin transform associated with $\mu$. To prove part (iii), let ${\bf X}, {\bf Y}\in {\bf B}_{\bf n}(\cH)$ and note that
\begin{equation*}
\begin{split}
\|\boldsymbol\cB_\mu(A,{\bf X})-\boldsymbol\cB_\mu(A,{\bf Y})\|&\leq \|\mu\|\|C_{\bf X}^*(A\otimes
I_\cH)(C_{\bf X}-C_{\bf Y})\|+ \|\mu\|\|(C_{\bf X}^*-C_{\bf Y}^*)(A\otimes I_\cH)C_{\bf Y}\|\\
&\leq \|\mu\|\|A\|\|C_{\bf X}-C_{\bf Y}\|\left(\|C_{\bf X}\|+\|C_{\bf Y}\|\right).
\end{split}
\end{equation*}
The continuity of the map ${\bf X}\mapsto \boldsymbol\cB_\mu(A, {\bf X})$ will follow once
we prove that ${\bf X}\mapsto C_{\bf X}$ is  a continuous  map on
${\bf B}_{\bf n}(\cH)$. Note that
\begin{equation*}
\begin{split}
\|C_{\bf X}-C_{\bf Y}\|&\leq \boldsymbol\Delta_{\bf X}(I)^{1/2}\|
\|\prod_{i=1}^k
 \left(I-{\bf R}_{X_i^*} \right)^{-1}-\prod_{i=1}^k
 \left(I-{\bf R}_{Y_i^*}\right)^{-1}\|\\
 &+\|\boldsymbol\Delta_{\bf X}(I)^{1/2}-\boldsymbol\Delta_{\bf Y}(I)^{1/2}\|\|\prod_{i=1}^k
 \left(I-{\bf R}_{X_i^*}\right)^{-1}\|,
 \end{split}
\end{equation*}
where ${\bf R}_{X_i^*}:=I-{\bf R}_{i,1}\otimes X_{i,1}^*-\cdots -{\bf R}_{i,n_i}\otimes X_{i,n_i}^*$.
Since the maps ${\bf X}\mapsto \prod_{i=1}^k
 \left(I-{\bf R}_{X_i^*} \right)^{-1}$  and ${\bf X}\mapsto \boldsymbol\Delta_{\bf X}(I)^{1/2}$  are  continuous on ${\bf B}_{\bf n}(\cH)$ in the operator norm topology, our assertion follows.
     The inequality in  (iii) is obvious.
The proof is complete.
\end{proof}

We remark  that the noncommutative Poisson transform introduced
in \cite{Po-poisson} is in fact a particular case of the
noncommutative Berezin transform associated with a linear functional.
 Indeed, let $\tau$ be the
linear functional on $B(\otimes_{i=1}^kF^2(H_{n_i}))$ defined by $\tau(A):=\left<
A(1),1\right>$. If ${\bf X}\in {\bf B}_{\bf n}(\cH)$ is fixed, then
$\boldsymbol\cB_\tau(\cdot, {\bf X}):B(\otimes_{i=1}^kF^2(H_{n_i}))\to   B(\cH)$ is a completely
contractive linear map and
\begin{equation*}
\begin{split}
\left<\boldsymbol\cB_\tau(A, {\bf X})x,y\right> =\left< C_{\bf X}^*( A \otimes I_\cH)
C_{\bf X}(1\otimes x),1\otimes y\right>,\quad x,y\in \cH.
\end{split}
\end{equation*}
Hence, we have
 $$\boldsymbol\cB_\tau(A, {\bf X})={\bf K}_{\bf X}^*(A\otimes I){\bf K_X},
 $$
 where ${\bf K}_{\bf X}$ is the noncommutative Berezin kernel at ${\bf X}$.
Note also that if $A\in B(\otimes_{i=1}^kF^2(H_{n_i}))$  is fixed,
then $\boldsymbol\cB_\tau(A, \cdot\, ):
{\bf B}_{\bf n}(\cH) \to   B(\cH)$ is a bounded continuous map and
$\|\boldsymbol\cB_\tau (A,{\bf X})\|\leq  \|A\| $ for any    $ {\bf X}\in {\bf B}_{\bf n}(\cH)$.

We mention that if $n_1=\cdots =n_k=1$, $\cH=\CC$, ${\bf X}=\boldsymbol \lambda=(\lambda_1,\ldots, \lambda_k)\in \DD^k$, we recover the Berezin
transform of a bounded linear operator on the Hardy space
$H^2(\DD^k)$, i.e.
$$
\boldsymbol \cB_\tau(A,\lambda)=\prod_{i=1}^k(1-|\lambda_i|^2)\left<A k_{\boldsymbol\lambda},
k_{\boldsymbol\lambda}\right>,\quad A\in B(H^2(\DD^k)),
$$
where $k_\lambda({\bf z}):=\prod_{i=1}^k(1-\overline{\lambda}_i z_i)^{-1}$ and  ${\bf z}=(z_1,\ldots, z_k)\in \DD^k$.

Define the set
 \begin{equation}
 \label{La}
 \Lambda:=\{(\boldsymbol \sigma, \boldsymbol \omega)\in {\bf F}^+_{\bf n}\times
 {\bf F}^+_{\bf n}: \ \boldsymbol\sigma\sim_{lc}\boldsymbol \omega \text{ and }
 (\boldsymbol \sigma, \boldsymbol \omega)=
 ({c_l^+(\boldsymbol\sigma, \boldsymbol \omega)}, {c_l^-(\boldsymbol\sigma, \boldsymbol \omega))}\}.
 \end{equation}
 Set $\widetilde \Lambda:=\{(\widetilde{\boldsymbol \sigma}, \widetilde{\boldsymbol \omega}): \ (\boldsymbol \sigma, \boldsymbol \omega)\in \Lambda\}$
 and note that
 $$
 \widetilde\Lambda:=\{(\widetilde{\boldsymbol \sigma}, \widetilde{\boldsymbol \omega})\in {\bf F}^+_{\bf n}\times
 {\bf F}^+_{\bf n}: \ \widetilde{\boldsymbol\sigma}\sim_{rc}\widetilde{\boldsymbol \omega} \text{ and }
 (\widetilde{\boldsymbol \sigma}, \widetilde{\boldsymbol \omega})=
 ({c_r^+(\widetilde{\boldsymbol\sigma}, \widetilde{\boldsymbol \omega})}, {c_r^-(\widetilde{\boldsymbol\sigma}, \widetilde{\boldsymbol \omega}))}\}.
 $$
 Moreover, we have $\Lambda=\widetilde\Lambda$. In case
 $( {\boldsymbol \sigma},  {\boldsymbol \omega})\in \Lambda$, then one can easily see that
 ${c_l^+(\boldsymbol\sigma, \boldsymbol \omega)}={c_r^+(\boldsymbol\sigma, \boldsymbol \omega)}$ and
 $c_l^-(\boldsymbol\sigma, \boldsymbol \omega)=c_r^-(\boldsymbol\sigma, \boldsymbol \omega)$.

In what follows, we introduce the {\it noncommutative Poisson transform} of a completely positive  linear map on the operator system
$$
\boldsymbol\cR_{\bf n}^* \boldsymbol\cR_{\bf n}:=\text{\rm span}\{{\bf R}_{\boldsymbol\alpha}^* {\bf R}_{\boldsymbol\beta}: \ \boldsymbol\alpha, \boldsymbol\beta\in \FF_{n_1}^+\times\cdots \times \FF_{n_k}^+\},
$$
where  ${\bf R}:=({\bf R}_1,\ldots, {\bf R}_k)$ and   ${\bf R}_i:=({\bf R}_{i,1},\ldots,{\bf R}_{i,n_i})$ is the $n_i$-tuple
of right creation operators (see Section 1).
Regard
$M_m(\boldsymbol\cR_{\bf n}^* \boldsymbol\cR_{\bf n})$ as a subspace of $M_m(B(\otimes_{i=1}^k F^2(H_{n_i})))$. Let
$M_m(\boldsymbol\cR_{\bf n}^* \boldsymbol\cR_{\bf n})$ have the norm structure that it inherits from
the (unique) norm structure on the $C^*$-algebra $M_m(B(\otimes_{i=1}^kF^2(H_{n_i})))$.
We remark that
\begin{equation*}
\begin{split}
\boldsymbol\cR_{\bf n}^* \boldsymbol\cR_{\bf n}
&=
\text{\rm span}\{{\bf R}_{\boldsymbol\alpha}^* {\bf R}_{\boldsymbol\beta}: \ (\boldsymbol\alpha, \boldsymbol\beta)\in \Lambda\}\\
&=\text{\rm span}\{{\bf R}_{\widetilde{\boldsymbol\alpha}}^* {\bf R}_{\widetilde{\boldsymbol\beta}}: \ (\boldsymbol\alpha, \boldsymbol\beta)\in \Lambda\},
\end{split}
\end{equation*}
where $\Lambda=\widetilde\Lambda$ is given by relation \eqref{La}.
If $\mu:\boldsymbol\cR_{\bf n}^*\boldsymbol\cR_{\bf n}\to B(\cE)$ is a  completely  bounded linear
map, then there exists a unique completely bounded linear map
$$\widehat
\mu:=\mu\otimes id:\overline{\boldsymbol\cR_{\bf n}^* \boldsymbol\cR_{\bf n}}^{\|\cdot \|}\otimes_{min} B(\cH)\to
B(\cE)\otimes_{min} B(\cH)
$$
such that
$$ \widehat \mu(A\otimes Y)= \mu(A)\otimes Y, \qquad A\in
\boldsymbol\cR_{\bf n}^* \boldsymbol\cR_{\bf n}, \,Y\in B(\cH).
 $$
 Moreover, $\|\widehat
\mu\|_{cb}=\|\mu\|_{cb}$ and, if $\mu$ is completely positive, then
 so is $\widehat \mu$.

We define  the {\it free pluriharmonic Poisson kernel}     by setting
\begin{equation*}
 \boldsymbol{\cP}({\bf R}, {\bf X}):=\sum_{m_1\in \ZZ}\cdots \sum_{m_k\in \ZZ} \sum_{{\alpha_i,\beta_i\in \FF_{n_i}^+, i\in \{1,\ldots, k\}}\atop{|\alpha_i|=m_i^-, |\beta_i|=m_i^+}} {\bf R}_{1,\widetilde\alpha_1}^*\cdots {\bf R}_{k,\widetilde\alpha_k}^*{\bf R}_{1,\widetilde\beta_1}\cdots {\bf R}_{k,\widetilde\beta_k}
\otimes { X}_{1,\alpha_1}\cdots {X}_{k,\alpha_k}{X}_{1,\beta_1}^*\cdots {X}_{k,\beta_k}^*
\end{equation*}
for any ${\bf X}\in {\bf B_n}(\cH)$,
where the convergence is in the operator norm topology. We need to show that the latter convergence holds. Indeed, note that,
for each $i\in \{1,\ldots, k\}$ and $r\in [0,1)$, we have
\begin{equation*}
\begin{split}
W_i&:=\sum_{m_i\in \ZZ} \sum_{{\alpha_i,\beta_i\in \FF_{n_i}^+ }\atop{|\alpha_i|=m_i^-, |\beta_i|=m_i^+}}{\bf R}_{i,\widetilde\alpha_i}^* {\bf R}_{i,\widetilde\beta_i}
\otimes  r^{|\alpha_i|+|\beta_i|}{ \bf S}_{i,\alpha_i} {\bf S}_{i,\beta_i}^* \\
&= \lim_{p_i\to\infty} \left(\sum_{{\alpha_i\in \FF_{n_i}}\atop{0<|\alpha_i|\leq p_i}}
{\bf R}_{i,\widetilde\alpha_i}^*\otimes r^{|\alpha_i|}{ \bf S}_{i,\alpha_i}+ \sum_{{\beta_i\in \FF_{n_i}}\atop{0\leq|\beta_i|\leq p_i}}
{\bf R}_{i,\widetilde\beta_i}\otimes r^{|\beta_i|}{ \bf S}_{i,\beta_i}^*\right),
\end{split}
\end{equation*}
where the limit is in the operator norm topology.
One can easily see that
\begin{equation*}
\begin{split}
W_1\cdots W_k&=
\boldsymbol{\cP}({\bf R}, r{\bf S})\\
&:=
\lim_{p_1\to \infty}\cdots \lim_{p_k\to \infty}
\sum_{{m_1\in \ZZ}\atop{|m_1|\leq p_1}}\cdots \sum_{{m_k\in \ZZ}\atop{|m_k|\leq p_k}} \sum_{{\alpha_i,\beta_i\in \FF_{n_i}^+, i\in \{1,\ldots, k\}}\atop{|\alpha_i|=m_i^-, |\beta_i|=m_i^+}}\\
 &
\qquad\qquad  {\bf R}_{1,\widetilde\alpha_1}^*\cdots {\bf R}_{k,\widetilde\alpha_k}^*{\bf R}_{1,\widetilde\beta_1}\cdots {\bf R}_{k,\widetilde\beta_k}
\otimes r^{\sum_{i=1}^k (|\alpha_i|+|\beta_i|)}{\bf S}_{1,\alpha_1}\cdots {\bf S}_{k,\alpha_k}{\bf S}_{1,\beta_1}^*\cdots {\bf S}_{k,\beta_k}^*.
\end{split}
\end{equation*}
Therefore, the series defining $\boldsymbol{\cP}({\bf R}, r{\bf S})$, i.e.
 \begin{equation*}
 \sum_{m_1\in \ZZ}\cdots \sum_{m_k\in \ZZ} \sum_{{\alpha_i,\beta_i\in \FF_{n_i}^+, i\in \{1,\ldots, k\}}\atop{|\alpha_i|=m_i^-, |\beta_i|=m_i^+}} {\bf R}_{1,\widetilde\alpha_1}^*\cdots {\bf R}_{k,\widetilde\alpha_k}^*{\bf R}_{1,\widetilde\beta_1}\cdots {\bf R}_{k,\widetilde\beta_k}
\otimes r^{\sum_{i=1}^k (|\alpha_i|+|\beta_i|)}{\bf S}_{1,\alpha_1}\cdots {\bf S}_{k,\alpha_k}{\bf S}_{1,\beta_1}^*\cdots {\bf S}_{k,\beta_k}^*
\end{equation*}
 are convergent in the operator norm topology. We remark that
  due to the fact that the operators $W_1,\ldots, W_k$ are commuting, the order of the limits above is irrelevant.
  Fix ${\bf X}\in {\bf B_n}(\cH)$ and let $r\in (0,1)$ be such that $\frac{1}{r}{\bf X}$ is in ${\bf B_n}(\cH)$.
Since the noncommutative Berezin transform
$\boldsymbol\cB_{\frac{1}{r}{\bf X}}$ is continuous in the operator norm, so is $id\otimes \boldsymbol\cB_{\frac{1}{r}{\bf X}}$. Consequently, applying $id\otimes \boldsymbol\cB_{\frac{1}{r}{\bf X}}$
to the relation above, we deduce that
\begin{equation*}
\begin{split}
(id\otimes \boldsymbol\cB_{\frac{1}{r}{\bf X}})[\boldsymbol{\cP}({\bf R}, r{\bf S})]
&=
\lim_{p_1\to \infty}\cdots \lim_{p_k\to \infty}
\sum_{{m_1\in \ZZ}\atop{|m_1|\leq p_1}}\cdots \sum_{{m_k\in \ZZ}\atop{|m_k|\leq p_k}} \sum_{{\alpha_i,\beta_i\in \FF_{n_i}^+, i\in \{1,\ldots, k\}}\atop{|\alpha_i|=m_i^-, |\beta_i|=m_i^+}}\\
 &
\qquad\qquad  {\bf R}_{1,\widetilde\alpha_1}^*\cdots {\bf R}_{k,\widetilde\alpha_k}^*{\bf R}_{1,\widetilde\beta_1}\cdots {\bf R}_{k,\widetilde\beta_k}
\otimes  {X}_{1,\alpha_1}\cdots {X}_{k,\alpha_k}{X}_{1,\beta_1}^*\cdots {X}_{k,\beta_k}^*,
\end{split}
\end{equation*}
where the limits are in the operator norm topology. This proves our assertion.
Now, we  introduce   the {\it noncommutative Poisson transform} of a completely  bounded linear  map $\mu:\boldsymbol\cR_{\bf n}^*\boldsymbol\cR_{\bf n}\to B(\cE)$ to
be the map \ $\boldsymbol\cP\mu : {\bf B}_{\bf n}(\cH)\to B(\cE)\otimes_{min}B(\cH)$
defined by
$$
(\boldsymbol\cP \mu)({\bf X}):=\widehat \mu[\boldsymbol{\cP}({\bf R}, {\bf X})],\qquad
{\bf X} \in {\bf B}_{\bf n}(\cH).
$$

The next result contains some of the basic properties  of the noncommutative Poisson  kernel (resp. transform).

\begin{theorem}
\label{Poisson-factor} Let \ $\mu:\boldsymbol\cR_{\bf n}^*\boldsymbol\cR_{\bf n}\to B(\cE)$ be a
completely bounded linear map. The following statements hold.

\begin{enumerate}
\item[(i)]
The map ${\bf X}\mapsto \boldsymbol{\cP}({\bf R}, {\bf X})$ is a positive $k$-pluriharmonic function  on the polyball
${\bf B}_{\bf n}$, with coefficients in $B(\otimes_{i=1}^k F^2(H_{n_i}))$, and has the factorization
$ {\cP}({\bf R}, {\bf X})=C_{\bf X} ^* C_{\bf X}$,
  where
  $$ C_{\bf X} := (I_{\otimes_{i=1}^kF^2(H_{n_i})} \otimes \boldsymbol\Delta_{\bf X}(I)^{1/2})\prod_{i=1}^k
 \left(I-{\bf R}_{i,1}\otimes X_{i,1}^*-\cdots -{\bf R}_{i,n_i}\otimes X_{i,n_i}^*\right)^{-1}.
$$
 \item[(ii)] The noncommutative
Poisson transform $\boldsymbol\cP\mu$ is a free $k$-pluriharmonic function on the regular polyball ${\bf B_n}$, which coincides with the Berezin transform
$\boldsymbol\cB_\mu(I,\cdot\,)$.
\item[(iii)]
 If $\mu$ is a completely positive
linear map, then $\boldsymbol\cP\mu$ is a  positive free $k$-pluriharmonic function on   ${\bf B}_{\bf n}$.
\end{enumerate}
\end{theorem}
\begin{proof} The fact  that the map ${\bf X}\mapsto \boldsymbol{\cP}({\bf R}, {\bf X})$ is a free $k$-pluriharmonic function on the polyball ${\bf B_n}$ with coefficients in $\boldsymbol\cR_{\bf n}^*\boldsymbol\cR_{\bf n}$ was proved in the remarks preceding the theorem.
Setting $\Lambda_i:={\bf R}_{i,1}\otimes r{\bf S}_{i,1}^*-\cdots -
{\bf R}_{i,n_i}\otimes r{\bf S}_{i,n_i}^*$ for each $i\in\{1,\ldots, k\}$, we have
\begin{equation*}
\begin{split}
W_i&:=\sum_{m_i\in \ZZ} \sum_{{\alpha_i,\beta_i\in \FF_{n_i}^+ }\atop{|\alpha_i|=m_i^-, |\beta_i|=m_i^+}}{\bf R}_{i,\widetilde\alpha_i}^* {\bf R}_{i,\widetilde\beta_i}
\otimes  r^{|\alpha_i|+|\beta_i|}{ \bf S}_{i,\alpha_i} {\bf S}_{i,\beta_i}^*\\
&=(I-\Lambda_i)^{-1} -I +(I-\Lambda_i^*)^{-1}\\
&=(I-\Lambda_i^*)^{-1}[(I-\Lambda_i) -(I-\Lambda_i^*)(I-\Lambda_i)+(I-\Lambda_i^*)] (I-\Lambda_i)^{-1}\\
&=
(I-\Lambda_i^*)^{-1}\left[I_{\otimes_{i=1}^k F^2(H_{n_i})}\otimes
 \left(I_{\otimes_{i=1}^k F^2(H_{n_i})}-\sum_{j=1}^{n_i} r^2 {\bf S}_{i,j} {\bf S}_{i,j}^*\right)\right] (I-\Lambda_i)^{-1}.
\end{split}
\end{equation*}
Recall that ${\bf R}_{i,s}{\bf R}_{j,t}={\bf R}_{j,t}{\bf R}_{i,s}$ and
${\bf R}_{i,s}{\bf R}_{j,t}^*={\bf R}_{j,t}^*{\bf R}_{i,s}$ for any $i,j\in \{1,\ldots, k\}$ with $i\neq j$ and  for any $s\in \{1,\ldots, n_i\}$ and
$t\in \{1,\ldots, n_j\}$. Similar commutation relations hold for the universal model ${\bf S}$. Since
$\boldsymbol{\cP}({\bf R}, r{\bf S})= W_1\cdots W_k$ and $W_1,\ldots, W_k$ are commuting positive operators,  we deduce that
\begin{equation*}
\boldsymbol{\cP}({\bf R}, r{\bf S})=\left[\prod_{i=1}^k(I-{\bf R}_{i,1}^*\otimes r{\bf S}_{i,1}-\cdots
-{\bf R}_{i,n_i}^*\otimes r{\bf S}_{i,n_i})^{-1}\right] \left(I\otimes \boldsymbol\Delta_{r{\bf S}}\right)
\prod_{i=1}^k(I-{\bf R}_{i,1}\otimes r{\bf S}^*_{i,1}-\cdots
-{\bf R}_{i,n_i}\otimes r{\bf S}^*_{i,n_i})^{-1}
\end{equation*}
for any $r\in [0,1)$, and  $\boldsymbol{\cP}({\bf R}, r{\bf S})=C_{r{\bf S}}^* C_{r{\bf S}}\geq 0$.
Now,  let ${\bf X}\in {\bf B_n}(\cH)$ and let $r\in (0,1)$ be such that $\frac{1}{r}{\bf X}\in {\bf B_n}(\cH)$.
Since the noncommutative Berezin transform
$\boldsymbol\cB_{\frac{1}{r}{\bf X}}$ is continuous in the operator norm and completely positive, so is $id\otimes \boldsymbol\cB_{\frac{1}{r}{\bf X}}$. Consequently, applying $id\otimes \boldsymbol\cB_{\frac{1}{r}{\bf X}}$
to the relations above, we deduce that
\begin{equation*}
\begin{split}
\boldsymbol{\cP}({\bf R}, {\bf X})&=\left(id\otimes \boldsymbol\cB_{\frac{1}{r}{\bf X}}\right)\left[\boldsymbol{\cP}({\bf R}, r{\bf S})\right]\\
&=\prod_{i=1}^k
 \left(I-{\bf R}_{i,1}^*\otimes X_{i,1}-\cdots -{\bf R}_{i,n_i}^*\otimes X_{i,n_i}\right)^{-1}\left(I\otimes \boldsymbol\Delta_{\bf X}\right)
 \prod_{i=1}^k
 \left(I-{\bf R}_{i,1}\otimes X_{i,1}^*-\cdots -{\bf R}_{i,n_i}\otimes X_{i,n_i}^*\right)^{-1}\\
 &=C_{\bf X} ^* C_{\bf X},
\end{split}
\end{equation*}
which completes the proof of item (i).

 Using the results above and  the continuity of $\widehat \mu$ in the operator
norm, we deduce that the noncommutative Berezin transform $\boldsymbol\cB_\mu(I, \cdot \,)$ associated  with $\mu$  coincides with the Poisson transform $\boldsymbol\cP\mu$. Indeed,  we have
\begin{equation*}
\begin{split}
\boldsymbol\cB_\mu(I,{\bf X})&= \widehat \mu(C_{\bf X} ^* C_{\bf X})\\
&=
\sum_{m_1\in \ZZ}\cdots \sum_{m_k\in \ZZ} \sum_{{\alpha_i,\beta_i\in \FF_{n_i}^+, i\in \{1,\ldots, k\}}\atop{|\alpha_i|=m_i^-, |\beta_i|=m_i^+}} \mu\left({\bf R}_{1,\widetilde\alpha_1}^*\cdots {\bf R}_{k,\widetilde\alpha_k}^*{\bf R}_{1,\widetilde\beta_1}\cdots {\bf R}_{k,\widetilde\beta_k}\right)
\otimes { X}_{1,\alpha_1}\cdots {X}_{k,\alpha_k}{X}_{1,\beta_1}^*\cdots {X}_{k,\beta_k}^*\\
&=
\widehat \mu (\boldsymbol{\cP}({\bf R}, {\bf X}))\\
&=(\boldsymbol\cP\mu)({\bf X})
\end{split}
\end{equation*}
for any ${\bf X}\in {\bf B}_{\bf n}(\cH)$,
where the convergence is in the operator norm topology of
$B(\cK\otimes \cH)$. This proves item (ii).
Note also that   the Poisson transform $\boldsymbol\cP\mu$
is a free $k$-pluriharmonic function on ${\bf B}_{\bf n}$ with coefficients
in $B(\cE)$.
If $\mu$ is completely positive, then
 so is $\widehat \mu$. Using the fact that $\widehat \mu(C_{\bf X} ^* C_{\bf X})=(\boldsymbol\cP\mu)({\bf X})$, we deduced
item (iii). The proof is complete.
\end{proof}

 Consider  the particular case when $n_1=\cdots=n_k=1$,
$\cH=\cK=\CC$, ${\bf X}=(X_1,\ldots, X_k)$ with $X_j=r_je^{i\theta_j}\in \DD$,  and $\mu$ is a complex
Borel measure on $\TT^k$. Note that $\mu$ can be seen as a bounded linear
functional on $C(\TT^k)$. Consequently,  there is a unique bounded linear functional
$\hat\mu$ on the operator system  generated by the monomials
$S_1^{m_1^-}\cdots S_k^{m_k^-} S_1^{*\, m_1^+}\cdots S_k^{*\, m_k^+}$, where $m_1,\ldots, m_k\in \ZZ$, and  $S_1,\ldots, S_k$ are the unilateral shifts acting on the Hardy space $H^2(\TT^k)$, such that
$$
\hat \mu(S_1^{m_1^-}\cdots S_k^{m_k^-} S_1^{*\, m_1^+}\cdots S_k^{*\, m_k^+})=\mu(e^{im_1^- \varphi_1}\cdots e^{im_k^- \varphi_k}e^{-im_1^+ \varphi_1}\cdots e^{-im_k^+ \varphi_k}),\qquad m_1,\ldots, m_k\in \ZZ.
$$
Indeed, if $p$ is any polynomial function of the form
$$
p(z_1,\ldots, z_k,\bar z_1,\ldots, \bar z_k)=\sum a_{(m_1,\ldots, m_k)} z_1^{m_1^-}\cdots z_k^{m_k^-} \bar z_1^{m_1^+}\cdots \bar z_k^{m_k^+},\qquad (z_1,\ldots, z_k)\in \DD^k,
$$
where $a_{(m_1,\ldots, m_k)} \in \CC$, then, due to the noncommutative von Neumann inequality \cite{Po-poisson}, we have
\begin{equation*}
\begin{split}
|\hat \mu(p(S_1,\ldots, S_k, S_1^*,\ldots, S_k^*))|&=|\mu(p(e^{i\varphi_1},\ldots, e^{i\varphi_k},e^{-i\varphi_1},\ldots, e^{-i\varphi_k}))|\\
&\leq \|\mu\| \|p(S_1,\ldots, S_k, S_1^*,\ldots, S_k^*)\|.
\end{split}
\end{equation*}
Therefore, $\hat \mu$ is a bounded linear functional on the operator system $\text{\rm span} \{\boldsymbol\cA_{\bf n}^*\boldsymbol\cA_{\bf n}\}^{-\|\cdot \|}$.
Note that the noncommutative Poisson transform of $\hat\mu$, i.e. $B_{\hat\mu}(I, \cdot \,)$, coincides with the classical Poisson transform of $\mu$.
Indeed, for any $ z=(r_1e^{i\theta_1},\ldots, r_ke^{i\theta_k})\in \DD^k$, we have
\begin{equation*}
\begin{split}
\boldsymbol\cB_{\hat\mu}(I,z)&=
\sum_{(p_1,\ldots, p_k)\in \ZZ^k} \hat \mu(S_1^{p_1^-}\cdots p_k^{p_k^-} S_1^{*\, p_1^+}\cdots S_k^{*\, p_k^+})z_1^{p_1}\cdots z_k^{p_k}\\
&=
\sum_{(p_1,\ldots, p_k)\in \ZZ^k}\mu(\bar\zeta_1^{p_1}\cdots \bar \zeta^{p_k})z_1^{p_1}\cdots z_k^{p_k}\\
&=
\sum_{(p_1,\ldots, p_k)\in \ZZ^k}\left(\int_{\TT^k} \bar\zeta_1^{p_1}\cdots \bar \zeta^{p_k}d\mu(\zeta)\right) z_1^{p_1}\cdots z_k^{p_k}\\
&=
\int_{\TT^k} \left(\sum_{(p_1,\ldots, p_k)\in \ZZ^k} r_1^{|p_1|}\cdots r_k^{|p_k|} e^{ip_1(\theta_1-\varphi_1)}\cdots e^{ip_k(\theta_k-\varphi_k)}\right)d\mu(\zeta)\\
&=\int_{\TT^k} P(z,\zeta)d\mu(\zeta),
\end{split}
\end{equation*}
where

$$
  P(z,\zeta)=
P_{r_1}(\theta_1-\varphi_1)\cdots P_{r_k}(\theta_k-\varphi_k),  \qquad
\zeta=(e^{i\varphi_1},\ldots, e^{i\varphi_k})\in \TT^k,
$$
and
$P_{r}(\theta-\varphi)=\frac{1-r^2}{1-2r\cos(\theta- \varphi)+ r^2}$ is the Poisson kernel of the unit disc (see \cite{Ru1}).

 We recall that $\Lambda$ denotes  the set of all pairs $(\boldsymbol\alpha, \boldsymbol\beta)\in {\bf F}^+_{\bf n}\times {\bf F}^+_{\bf n}$, where ${\bf F}^+_{\bf n}:=\FF_{n_1}^+\times \cdots \times \FF_{n_k}^+$,  with the property that $\boldsymbol\alpha\sim_{lc}\boldsymbol\beta$ and $(\boldsymbol\alpha, \boldsymbol\beta)=(c_l^+(\boldsymbol\alpha, \boldsymbol\beta),c_l^-(\boldsymbol\alpha, \boldsymbol\beta))$.
We remark that $(\boldsymbol\alpha, \boldsymbol\beta)\in \Lambda$ if and only if $(\widetilde{\boldsymbol\alpha}, \widetilde{\boldsymbol\beta})\in \Lambda$. As before, we use the notation $\widetilde{\boldsymbol\alpha}=(\widetilde{\boldsymbol\alpha}_1,\ldots, \widetilde{\boldsymbol\alpha}_k)$, if $ {\boldsymbol\alpha}=( {\boldsymbol\alpha}_1,\ldots,  {\boldsymbol\alpha}_k)\in \FF_{\bf n}^+$.

Throughout  the rest of this section, we assume that $\cE$ is a separable Hilbert
space.
\begin{lemma}\label{mu-r}
Let $\mu: \boldsymbol\cR_{\bf n}^* \boldsymbol\cR_{\bf n}\to B(\cE)$ be a completely bounded linear map. For each
$r\in [0,1)$, define  the linear map $\mu_r: \boldsymbol\cR_{\bf n}^* \boldsymbol\cR_{\bf n}\to B(\cE)$
by
$$
\mu_r({\bf R}_{\boldsymbol\alpha}^* {\bf R}_{\boldsymbol\beta}):= r^{|\boldsymbol\alpha|+|\boldsymbol\beta|}\mu({\bf R}_{\boldsymbol\alpha}^* {\bf R}_{\boldsymbol\beta}),\qquad (\boldsymbol\alpha, \boldsymbol\beta)\in \Lambda,
$$
where $|\boldsymbol\alpha|:=|\alpha_1|+\cdots + |\alpha_k|$ if $\boldsymbol\alpha=(\alpha_1,\ldots, \alpha_k)\in \FF^+_{\bf n}$.
Then
\begin{enumerate}
\item[(i)]
$\mu_r$ is a completely bounded linear map;
\item[(ii)]
$ \|\mu\|_{cb}=\sup\limits_{0\leq r<1}
\|\mu_r\|_{cb}=\lim\limits_{r\to 1} \|\mu_r\|_{cb}$;
\item[(iii)] for any $A\in \boldsymbol\cR_{\bf n}^* \boldsymbol\cR_{\bf n}$, $\mu_r(A)\to \mu(A)$, as $r\to 1$, in
the operator norm topology;
\item[(iv)] If $\mu$ is completely positive, then so is $\mu_r$ for any $r\in [0,1)$.
\end{enumerate}
\end{lemma}

\begin{proof} Let $$p({\bf R}^*, {\bf R}):=\sum_{(\boldsymbol\alpha, \boldsymbol\beta)\in \Lambda'\subset \Lambda:\, \text{card}(\Lambda')<\aleph_0}
a_{(\boldsymbol\alpha, \boldsymbol\beta)} {\bf R}_{\boldsymbol\alpha}^* {\bf R}_{\boldsymbol\beta},\qquad a_{(\boldsymbol\alpha, \boldsymbol\beta)}\in \CC,
$$
and  $0\leq r_1<r_2\leq 1$.
Using the noncommutative von Neumann inequality \cite{Po-poisson},  we deduce that

\begin{equation*}
\begin{split}
\|\mu_{r_1}(p({\bf R}^*, {\bf R}))\|&=
\|\mu( p(r_1{\bf R}^*, r_1{\bf R}))\|\\
&=\left\|\mu_{r_2}\left(p(\frac{r_1}{r_2}{\bf R}^*, \frac{r_1}{r_2}{\bf R})\right) \right\|\\
&\leq \|\mu_{r_2}\| \|p({\bf R}^*, {\bf R})\|.
\end{split}
\end{equation*}
 In particular, we have $\|\mu_r\|\leq
\|\mu\|$ for any $r\in [0,1)$. Similarly, passing to matrices over
$\boldsymbol\cR_{\bf n}^* \boldsymbol\cR_{\bf n}$, one can show that $\|\mu_{r_1}\|_{cb}\leq
\|\mu_{r_2}\|_{cb}$ if $0\leq r_1<r_2\leq 1$, and $\|\mu_r\|_{cb}\leq
\|\mu\|_{cb}$ for any $r\in [0,1)$. Now, one can easily see that $\mu_r(A)\to\mu(A)$ in the operator norm topology, as $r\to 1$,
for any $A\in \boldsymbol\cR_{\bf n}^* \boldsymbol\cR_{\bf n}$, and
$\|\mu\|_{cb}=\sup_{0\leq r<1} \|\mu_r\|_{cb}$. Hence, and using the fact that
 the function
$r\mapsto \|\mu_r\|_{cb}$ is increasing for $r\in [0,1)$, we deduce that
    $\lim_{r\to 1} \|\mu_r\|_{cb}$
exists and it is equal to $\|\mu\|_{cb}$.

To prove item (iv), note that
$\mu_r(p({\bf R}^*, {\bf R}))=\mu\left(\boldsymbol\cB_{r{\bf R}}[p({\bf S}^*, {\bf S})]\right)$.   Since the noncommutative Berezin transform $\boldsymbol\cB_{r{\bf R}}$ and $\mu$ are completely positive linear maps and $p({\bf R}^*, {\bf R})$ is unitarily equivalent to $p({\bf S}^*, {\bf S})$, we deduce that  $\mu_r$ is a  completely positive linear map for each $r\in [0,1)$.
This completes the proof.
\end{proof}

Let   $F$ be a free $k$-pluriharmonic  function on the polyball ${\bf B_n}$, with  operator-valued coefficients in $B(\cE)$,      with representation
\begin{equation*}
F({\bf X})= \sum_{m_1\in \ZZ}\cdots \sum_{m_k\in \ZZ} \sum_{{\alpha_i,\beta_i\in \FF_{n_i}^+, i\in \{1,\ldots, k\}}\atop{|\alpha_i|=m_i^-, |\beta_i|=m_i^+}}A_{(\alpha_1,\ldots,\alpha_k;\beta_1,\ldots, \beta_k)}
\otimes {\bf X}_{1,\alpha_1}\cdots {\bf X}_{k,\alpha_k}{\bf X}_{1,\beta_1}^*\cdots {\bf X}_{k,\beta_k}^*.
\end{equation*}
We associate with $F$ and each $r\in[0,1)$ the linear map
$\nu_{F_r}:\boldsymbol\cR_{\bf n}^* \boldsymbol\cR_{\bf n}\to B(\cE)$ by setting
\begin{equation}
\label{nufr}
\nu_{F_r}(  {\bf R}_{\widetilde{\boldsymbol\alpha}}^* {\bf R}_{\widetilde{\boldsymbol\beta}}):=r^{|\boldsymbol\alpha|+|\boldsymbol\beta|}
A_{(\alpha_1,\ldots,\alpha_k;\beta_1,\ldots, \beta_k)},\qquad (\boldsymbol\alpha, \boldsymbol\beta)\in \Lambda.
\end{equation}
 We remark that $\nu_{F_r}$ is uniquely determined by the radial function $r\mapsto F(r{\bf S})$. Indeed, note that
if  $x:=x_1\otimes \cdots \otimes x_k$, $y=y_1\otimes \cdots \otimes y_k$ satisfy relation \eqref{xy}, and $h,\ell\in \cE$, we have
  \begin{equation*}
  \begin{split}
  \left<F(r{\bf S})(h\otimes x), \ell\otimes y\right>
  = \left<r^{|\boldsymbol\alpha|+|\boldsymbol\beta|} A_{(\alpha_1,\ldots,\alpha_k;\beta_1,\ldots, \beta_k)}h, \ell\right>
  =\left<\nu_{F_r}(  {\bf R}_{\widetilde{\boldsymbol\alpha}}^* {\bf R}_{\widetilde{\boldsymbol\beta}})h,\ell\right>,\qquad (\boldsymbol\alpha, \boldsymbol\beta)\in \Lambda.
  \end{split}
  \end{equation*}

In what follows, we denote by $C^*({\bf R})$ the $C^*$-algebra generated by the right creation operators ${\bf R}_{i,j}$, where  $i\in \{1,\ldots, k\}$ and  $j\in \{1,\ldots, n_i\}$.

\begin{theorem}\label{pluri-measure}
Let $F:{\bf B_n}(\cH)\to B(\cE)\otimes_{min} B(\cH)$ be   a free $k$-pluriharmonic function. Then the following statements are equivalent:
\begin{enumerate}
\item[(i)]
there exists a completely  bounded linear map $\mu:C^*({\bf R})\to B(\cE)$ such that $ F=\boldsymbol\cP\mu; $
\item[(ii)] the   linear maps \ $\{\nu_{F_r}\}_{r\in[0,1)}$
associate with   $F$ are completely bounded and
   $\sup\limits_{0\leq r<1}
\|\nu_{F_r}\|_{cb}<\infty$;
\item[(iii)] there exists a
$k$-tuple   ${\bf V}=(V_1,\ldots, V_k)$, $V_i=(V_{i,1},\ldots, V_{i,n_i})$, of doubly commuting row isometries acting on  a Hilbert space $\cK$  and  bounded linear operators $W_1,W_2:\cE\to \cK$ such that
$$
F({\bf X})=  (W_1^*\otimes I)\left[ C_{\bf X}({\bf V})^*
C_{\bf X}({\bf V}) \right] (W_2\otimes I),
$$
where $$C_{\bf X}({\bf V}):=(I\otimes \boldsymbol\Delta_{\bf X}(I)^{1/2})\prod_{i=1}^k(I-V_{i,1}\otimes
X_{i,1}^*-\cdots -{V}_{i,n_i}\otimes X_{i,n_i}^*)^{-1}.$$
\end{enumerate}
Moreover, in this case we can choose $\mu$ such that $\|\mu\|_{cb}=\sup\limits_{0\leq r<1}
\|\nu_{F_r}\|_{cb}$.
\end{theorem}

\begin{proof} Assume that item (i) holds.    Then
$$F({\bf X})=\sum_{m_1\in \ZZ}\cdots \sum_{m_k\in \ZZ} \sum_{{\alpha_i,\beta_i\in \FF_{n_i}^+, i\in \{1,\ldots, k\}}\atop{|\alpha_i|=m_i^-, |\beta_i|=m_i^+}} \mu\left({\bf R}_{1,\widetilde\alpha_1}^*\cdots {\bf R}_{k,\widetilde\alpha_k}^*{\bf R}_{1,\widetilde\beta_1}\cdots {\bf R}_{k,\widetilde\beta_k}\right)
\otimes { X}_{1,\alpha_1}\cdots {X}_{k,\alpha_k}{X}_{1,\beta_1}^*\cdots {X}_{k,\beta_k}^*
$$
 for any ${\bf X}\in {\bf B_n}(\cH)$,  where the convergence is in the operator norm topology. Set
$A_{(\boldsymbol\alpha;\boldsymbol\beta)}:=\mu({\bf R}_{\widetilde{\boldsymbol\alpha}}^* {\bf R}_{\widetilde{\boldsymbol\beta}})$ for
any $(\boldsymbol\alpha, \boldsymbol\beta)\in \Lambda$. Consequently, for each $r\in [0,1)$, we have
$$
\nu_{F_r}(  {\bf R}_{\widetilde{\boldsymbol\alpha}}^* {\bf R}_{\widetilde{\boldsymbol\beta}}):=r^{|\boldsymbol\alpha|+|\boldsymbol\beta|}
\mu({\bf R}_{\widetilde{\boldsymbol\alpha}}^* {\bf R}_{\widetilde{\boldsymbol\beta}}),\qquad (\boldsymbol\alpha, \boldsymbol\beta)\in \Lambda.
$$
We recall  that $(\boldsymbol\alpha, \boldsymbol\beta)\in \Lambda$ if and only if $(\widetilde{\boldsymbol\alpha}, \widetilde{\boldsymbol\beta})\in \Lambda$.
Applying
Lemma \ref{mu-r}, we deduce that  to $\{\nu_{F_r}\}$ is a completely bounded map and
$$\|\mu|_{\boldsymbol\cR_{\bf n}^* \boldsymbol\cR_{\bf n}}\|_{cb}=\sup_{0\leq r<1}\|\nu_{F_r}\|_{cb}<\infty,
$$
which proves that item (i) implies (ii).

Now, we prove the implication (ii)$\implies $ (i).
Assume that $F$ is a free pluriharmonic function on ${\bf B_n}$ with
 coefficients
in $B(\cE)$ and such that condition  (ii) holds. Let $\{q_j\}$ be a countable
dense subset of $\boldsymbol\cR_{\bf n}^* \boldsymbol\cR_{\bf n}$. For instance, we can consider consider all
 the operators  of the form
 $$p({\bf R}^*, {\bf R}):=\sum_{(\boldsymbol\alpha, \boldsymbol\beta)\in \Lambda:\, |\boldsymbol\alpha|\leq m, |\boldsymbol\beta|\leq m}
a_{(\boldsymbol\alpha, \boldsymbol\beta)} {\bf R}_{\boldsymbol\alpha}^* {\bf R}_{\boldsymbol\beta},
$$
where $m\in \NN$ and the
    coefficients $a_{(\boldsymbol\alpha, \boldsymbol\beta)}$ lie in some countable dense
subset of the complex plane.
 For each $j$, we have $\|\nu_{F_r}(q_j)\|\leq M\|q_j\|$ for any $r\in[0,1)$, where $M:=\sup_{0\leq
r<1} \|\nu_{F_r}\|_{cb}$.

Due to Banach-Alaoglu theorem \cite{Dou},  the ball $[B(\cE)]_M^-$  is compact
 in the $w^*$-topology. Since $\cE$ is a separable Hilbert space,
  $[B(\cE)]_M^-$ is  a metric space in the $w^*$-topology which coincides with
   the weak operator topology on $[B(\cE)]_M^-$. Consequently,
 the diagonal process guarantees the
existence  of a sequence $\{r_m\}_{m=1}^\infty$ such that $r_m\to 1$
and WOT-$\lim_{m\to 1} \nu_{F_{r_m}}(q_j)$ exists for each $q_j$. Fix
$A\in \boldsymbol\cR_{\bf n}^* \boldsymbol\cR_{\bf n}$ and $x,y\in \cE$ and let us
 prove that $\{\left<\nu_{F_{r_m}}(A)x,y\right>\}_{m=1}^\infty$ is a Cauchy sequence. Let
 $\epsilon>0$ and choose  $q_j$ so that
 $\|q_j-A\|<\frac{\epsilon}{3M\|x\|\|y\|}$. Now, we choose $N$ so that
 $
 |\left<(\nu_{F_{r_m}}(q_j)-\nu_{F_{r_k}}(q_j))x, y\right>|
 <\frac{\epsilon}{3}\quad \text{ for
 any } \ m,k>N.
 $
Due to the fact that
\begin{equation*}
\begin{split}
|\left<(\nu_{F_{r_m}}(A)-\nu_{F_{r_k}}(A))x,y\right>| &\leq
|\left<(\nu_{F_{r_m}}(A-q_j)x,y\right>|+
|\left<(\nu_{F_{r_m}}(q_j)-\nu_{F_{r_k}}(q_j))x,y\right>| \\
&\qquad  + |\left<\nu_{F_{r_k}}(q_j-A)x,y\right>|\\
&\leq 2 M\|x\|\|y\| \|A-q_j\|+ |\left<(\nu_{F_{r_m}}(q_j)-\nu_{F_{r_k}}(q_j))x,y\right>|,
\end{split}
\end{equation*}
we deduce that $|\left<(\nu_{F_{r_m}}(A)-\nu_{F_{r_k}}(A))x,
y\right>|<\epsilon$ \ for $m,k>N$. Therefore,
$$
\Phi(x,y):=\lim_{m\to\infty} \left<\nu_{F_{r_m}}(A)x,y\right>
$$
 exists
for any $x,y\in \cE$ and defines a  functional  $\Phi:\cE\times \cE\to
\CC$ which is linear in the first variable and conjugate linear in
the second.
 Moreover, we have $|\Phi(x,y)|\leq M \|A\|\|x\|\|y\|$ for any $x,y\in \cE$.
 Due to Riesz representation theorem, there exists a unique bounded linear
 operator $B(\cE)$, which we denote by $\nu(A)$, such that
 $\Phi(x,y)=\left<\nu(A)x,y\right>$ for $x,y\in \cE$. Therefore,
 $$
 \nu(A)=\text{\rm WOT-}\lim_{r_m\to 1} \nu_{F_{r_m}}(A),
\qquad A\in  \boldsymbol\cR_{\bf n}^* \boldsymbol\cR_{\bf n},
 $$
 and $\|\nu(A)\|\leq M\|A\|$.
Note that $\nu: \boldsymbol\cR_{\bf n}^* \boldsymbol\cR_{\bf n}\to B(\cE)$ is a   completely bounded map. Indeed, if $[A_{ij}]_{m}$ is
 an $m\times m$ matrix over $ \boldsymbol\cR_{\bf n}^* \boldsymbol\cR_{\bf n}$, then
 $[\nu(A_{ij})]_{m}=\text{\rm WOT-}\lim_{r_k\to 1}
 [\nu_{F_{r_k}}(A_{ij})]_{m}.
 $
Hence, $\left\|[\nu(A_{ij})]_{m}\right\|\leq M
\left\|[A_{ij}]_{m}\right\|$ for all $m$, and so $\|\nu\|_{cb}\leq
M$.
Note  also that $\nu({\bf R}_{\widetilde{\boldsymbol\alpha}}^* {\bf R}_{\widetilde{\boldsymbol\beta}})=A_{(\boldsymbol\alpha;\boldsymbol\beta)}$ for
any $(\boldsymbol\alpha, \boldsymbol\beta)\in \Lambda$, where $A_{(\boldsymbol\alpha;\boldsymbol\beta)}$ are the coefficients of $F$.
According to
Wittstock's extension
 theorem (see \cite{W1}, \cite{W2}), there exists $\mu:C^*({\bf R})\to B(\cE)$ a completely bounded linear map
 which extends $\nu$ and
 such that $\|\mu\|_{cb}=\|\nu\|_{cb}$.
Since  $ F=\boldsymbol\cP\mu$, the proof  of item (i) is complete.

 Now, we prove the equivalence of (i) with  (iii). If item (i) holds, then
 according to Theorem 8.4 from
\cite{Pa-book},  there exists a Hilbert space $\cK$, a
$*$-representation $\pi:C^*({\bf R} )\to B(\cK)$, and bounded
operators $W_1,W_2:\cE\to \cK$,   with $\|\mu\|=\|W_1\|\|W_2\|$
such that
\begin{equation}
\label{pi*}
\mu(A)=W_1^*\pi(A) W_2,\qquad A\in C^*({\bf R} ).
\end{equation}
Set $V_{i,j}:=\pi({\bf R}_{i,j})$ for $i\in \{1,\ldots, k\}$ and $j\in \{1,\ldots, n_i\}$ and note that    ${\bf V}=(V_1,\ldots, V_k)$, with  $V_i=(V_{i,1},\ldots, V_{i,n_i})$, is a $k$-tuple  of doubly commuting row isometries. Using Theorem \ref{Poisson-factor}, one can easily see that
the equality  $ F=\boldsymbol\cP\mu$ implies  the one from item (iii).
Now, we prove the implication (iii) $\implies (i)$.
Since  the $k$-tuple   ${\bf V}=(V_1,\ldots, V_n)$, $V_i=(V_{i,1},\ldots, V_{i,n_i})$, consists of  doubly commuting row isometries on  a Hilbert space $\cK$, the noncommutative von Neumann inequality \cite{Po-poisson} implies that the map $\pi:C^*({\bf R})\to B(\cE)$ defined by
$$\pi({\bf R}_{\boldsymbol\alpha}{\bf R}_{\boldsymbol\beta}^*):=
{\bf V}_{\boldsymbol\alpha}{\bf V}_{\boldsymbol\beta}^*,\qquad \boldsymbol\alpha,\boldsymbol\beta\in {\bf F}_{\bf n}^+,
$$
is a $*$-representation of $C^*({\bf R})$. Define $\mu:C^*({\bf R})\to B(\cE)$  by setting $\mu(A):=W_1^* \pi(A) W_2$, $A\in C^*({\bf R})$,  and note  that $\mu$ is a completely bounded linear map. Using relation
$$
F({\bf X})=  (W_1^*\otimes I)\left[ C_{\bf X}({\bf V})^*
C_{\bf X}({\bf V}) \right] (W_2\otimes I)
$$
 and  the factorization
$ {P}({\bf V}, {\bf X})=C_{\bf X}({\bf V}) ^* C_{\bf X}({\bf V})$ (see also Theorem \ref{Poisson-factor}), we deduce that
$ F({\bf X})=\boldsymbol\cP\mu({\bf X})$ for ${\bf X}\in {\bf B_n}(\cH)$.
The proof is complete.
\end{proof}

We  introduce the space ${\text{\bf PH}}^1({\bf B_n})$ of all free $k$-pluriharmonic functions $F$ on ${\bf B_n}$  such that the   linear maps \ $\{\nu_{F_r}\}_{r\in[0,1)}$
associate with   $F$ are completely bounded and set
   $\|F\|_1:=\sup\limits_{0\leq r<1}
\|\nu_{F_r}\|_{cb}<\infty$. As a consequence of Theorem \ref{pluri-measure}, one can  see that $\|\cdot\|_1$ is a norm
on ${\bf PH}^1({\bf B_n})$  and  $({\text{\bf PH}}^1({\bf B_n}), \|\cdot\|_1)$ is a Banach space which can be identified with the Banach space
$\text{\rm CB}(
\boldsymbol\cR_{\bf n}^* \boldsymbol\cR_{\bf n}, B(\cE))$
   of all completely bounded linear maps from $\boldsymbol\cR_{\bf n}^* \boldsymbol\cR_{\bf n}$ to $ B(\cE)$.

\begin{corollary}\label{cp} Let $F:{\bf B_n}(\cH)\to B(\cE)\otimes_{min} B(\cH)$ be  a free $k$-pluriharmonic function. Then the following statements are equivalent:
\begin{enumerate}
\item[(i)]
there exists a completely positive linear map $\mu:C^*({\bf R})\to B(\cE)$ such that $ F=\boldsymbol\cP\mu; $
\item[(ii)]
the   linear maps \ $\{\nu_{F_r}\}_{r\in[0,1)}$
associate with   $F$ are completely positive;
\item[(iii)]
there exists a
$k$-tuple   ${\bf V}=(V_1,\ldots, V_k)$, $V_i=(V_{i,1},\ldots, V_{i,n_i})$, of doubly commuting row isometries acting on  a Hilbert space $\cK\supset \cE$  and a bounded operator $W:\cE\to \cK$ such  that
$$
F({\bf X})=  (W^*\otimes I)\left[ C_{\bf X}({\bf V})^*
C_{\bf X}({\bf V}) \right](W\otimes I).
$$
\end{enumerate}
 \end{corollary}

\begin{proof}
The proof is similar to that of Theorem \ref{pluri-measure}. Note that for the  implication (i)$\implies $ (ii) we have to use  Lemma  \ref{mu-r}, part (iv).
For the converse,  note that if  $\nu_{F_r}$, $r\in [0,1)$,  are completely positive linear maps, then
$$
\|\nu_{F_r}\|_{cb}=\|\nu_{F_r}(1)\|=\|\nu_{F_r}\|=\|A_{({\bf g}; {\bf g})}\|,
$$
where ${\bf g}=(g_0^1,\ldots, g_0^k)$ is the identity element in ${\bf F}_{\bf n}^+$.
As in the proof of Theorem \ref{pluri-measure}, we find  a   completely bounded map
$\nu: \boldsymbol\cR_{\bf n}^* \boldsymbol\cR_{\bf n}\to B(\cE)$ such that
$$
 \nu(A)=\text{\rm WOT-}\lim_{r_m\to 1} \nu_{F_{r_m}}(A),
\qquad A\in  \boldsymbol\cR_{\bf n}^* \boldsymbol\cR_{\bf n}.
 $$
Since $\nu_{F_r}$, $r\in [0,1)$,  are completely positive linear maps, one can easily see that $\nu$ is completely positive. Using Arveson's extension theorem \cite{Ar}, we find a completely positive map
$\mu:C^*({\bf R})\to \CC$
 which extends $\nu$ and
 such that $\|\mu\|_{cb}=\|\nu\|_{cb}$. We also have that
  $F=\boldsymbol\cP\mu$.
  Now, the  proof that (iii) is equivalent to (i) uses Stinespring's representation theorem \cite{St} and is similar to the same equivalence from Theorem \ref{pluri-measure}. We leave it to the reader.
\end{proof}

An open question remains. Is any positive free $k$-pluriharmonic function on the regular polyball ${\bf B_n}$ the Poisson transform of a completely positive linear map ? The answer is positive if $k=1$ (see \cite{Po-pluriharmonic}) and also when $n_1=\cdots=n_k$ (see Section 3).

\bigskip

\section{Herglotz-Riesz representations for free holomorphic functions with positive real parts}

In this section, we introduce the { noncommutative Herglotz-Riesz
transform} of  a completely positive linear map $\mu: \boldsymbol\cR_{\bf n}^* \boldsymbol\cR_{\bf n}\to B(\cE)$ and
  obtain Herglotz-Riesz representation theorems for free holomorphic functions  with positive real parts in regular polyballs.

Define the space
$$
{\text{\bf RH}}({\bf B_n}):=\text{\rm span} \left\{  \Re f: \  f\in Hol_\cE({\bf B_n}) \right\},
$$
where $Hol_\cE({\bf B_n})$ is the set of all free holomorphic functions in the polyball ${\bf B_n}$, with coefficients in $B(\cE)$.
 Let $\tau:B(\otimes_{i=1}^k F^2(H_{n_i}))\to \CC$ be the
bounded  linear functional defined by $\tau(A)=\left<A1,1\right>$.
 We remark that the radial function associated with $\varphi \in {\bf RH}({\bf B_n})$, i.e.
$[0,1)\ni r\mapsto \varphi(r{\bf R})$,
uniquely determines the family  $\{\nu_{\varphi_r}\}_{r\in [0, 1)}$ of
linear maps
$\nu_{\varphi_r}: \boldsymbol\cR_{\bf n}^*\boldsymbol\cR_{\bf n}\to  B(\cE)$ defined by  relation \eqref{nufr}. Indeed, note that
\begin{equation*}
\begin{split}
\nu_{\varphi_r}({\bf R}_{\widetilde{\boldsymbol\alpha}}^*)&:= (\text{\rm id}\otimes
\tau)\left[(I\otimes {\bf R}_{\boldsymbol\alpha}^*)
 \varphi(r{\bf R} )\right],\\
 \nu_{\varphi_r}({\bf R}_{\widetilde{\boldsymbol\alpha}})&:= (\text{\rm id}\otimes \tau)\left[ \varphi(r{\bf R)}(I\otimes {\bf R}_{\boldsymbol\alpha}) \right]  \end{split}
\end{equation*}
for any $ \boldsymbol\alpha=(\alpha_1,\ldots, \alpha_k)\in {\bf F}_{\bf n}^+:= \FF_{n_1}^+\times \cdots \times \FF_{n_k}^+$,
where
$ \widetilde{\boldsymbol\alpha}=(\widetilde\alpha_1,\ldots, \widetilde\alpha_k)$ and ${\bf R}_{\boldsymbol\alpha}:={\bf R}_{1,\alpha_1}\cdots {\bf R}_{k,\alpha_k}$,
and
$\nu_{\varphi_r}({\bf R}_{\boldsymbol\alpha}^*{\bf R}_{\boldsymbol\beta})=0$ if ${\bf R}_{\boldsymbol\alpha}^*{\bf R}_{\boldsymbol\beta}$ is different from ${\bf R}_{\boldsymbol\gamma}$ or ${\bf R}_{\boldsymbol\gamma}^*$   for some ${\boldsymbol\gamma}\in {\bf F}_{\bf n}^+$.
Consider the space
$$
{\text{\bf RH}}^1({\bf B_n}):=  \left\{\varphi \in {\text{\bf RH}}({\bf B_n})  : \ \nu_{\varphi_r} \text{ is bounded  and } \  \sup\limits_{0\leq r<1}
\|\nu_{\varphi_r}\|<\infty\right\}.
$$
If $\varphi\in {\text{\bf RH}}^1({\bf B_n})$, we define
$\|\varphi\|_1:=\sup_{0\leq r<1}\|\nu_{\varphi_r}\|_{cb}$.
Denote by $\text{\rm CB}_0(
\boldsymbol\cR_{\bf n}^* \boldsymbol\cR_{\bf n}, B(\cE))$
 the space of all completely bounded linear maps
 $ \lambda:\boldsymbol\cR_{\bf n}^* \boldsymbol\cR_{\bf n}\to B(\cE)$ such that
 $\lambda({\bf R}_{\boldsymbol\alpha}^* {\bf R}_{\boldsymbol\beta})=0$ if
 ${\bf R}_{\boldsymbol\alpha}^* {\bf R}_{\boldsymbol\beta}$ is not equal to ${\bf R}_{\boldsymbol\gamma}$ or ${\bf R}_{\boldsymbol\gamma}^*$ for some $\boldsymbol\gamma \in {\bf F}_{\bf n}^+$.

\begin{theorem}\label{Har-1}
$\left({\text{\bf RH}}^1({\bf B_n}), \|\cdot \|_1\right)$ is a Banach space
which can be identified with the Banach space
 $\text{\rm CB}_0(
\boldsymbol\cR_{\bf n}^* \boldsymbol\cR_{\bf n}, B(\cE))$.
Moreover, if $\varphi:{\bf B_n}(\cH)\to B(\cE)\otimes_{min} B(\cH)$ is a function,   then the following statements are equivalent:
\begin{enumerate}
\item[(i)]
$\varphi$ is in ${\text{\bf RH}}^1({\bf B_n})$;
\item[(ii)]
there is a unique completely bounded linear  map $\mu_\varphi \in \text{\rm CB}_0(
\boldsymbol\cR_{\bf n}^* \boldsymbol\cR_{\bf n}, B(\cE))$ such that \ $\varphi=\boldsymbol\cP\mu_\varphi$;
\item[(iii)] there exists a
$k$-tuple   ${\bf V}=(V_1,\ldots, V_k)$, $V_i=(V_{i,1},\ldots, V_{i,n_i})$, of doubly commuting row isometries on  a Hilbert space $\cK$, and  bounded linear operators $W_1,W_2:\cE\to \cK$, such that
$$
\varphi({\bf X})=  (W_1^*\otimes I)\left[ C_{\bf X}({\bf V})^*
C_{\bf X}({\bf V}) \right](W_2\otimes I)
$$
  and
$W_1^*{\bf V}_{\boldsymbol\alpha}^* {\bf V}_{\boldsymbol\beta}W_2=0$ if
 ${\bf R}_{\boldsymbol\alpha}^* {\bf R}_{\boldsymbol\beta}$   is not equal to $ {\bf R}_{\boldsymbol\gamma}$   or ${\bf R}_{\boldsymbol\gamma}^*$ for some ${\boldsymbol\gamma}\in {\bf F}_{\bf n}^+$.
\end{enumerate}
\end{theorem}

\begin{proof}
Define the map $\Psi:\text{\rm CB}_0\,( \boldsymbol\cR_{\bf n}^* \boldsymbol\cR_{\bf n}, B(\cE))\to
{\text{\bf RH}}^1({\bf B_n})$   by $\Psi(\mu):=\boldsymbol\cP\mu$. To prove injectivity of
$\Psi$, let $\mu_1,\mu_2$ be in $\text{\rm CB}_0\,( \boldsymbol\cR_{\bf n}^* \boldsymbol\cR_{\bf n}, B(\cE))$ such that $\Psi(\mu_1)=\Psi(\mu_2)$.  Due to the
uniqueness  of the representation of a free $k$-pluriharmonic function
and the definition of the noncommutative Poisson transform of a
completely bounded map on $ \boldsymbol\cR_{\bf n}^* \boldsymbol\cR_{\bf n}$, we deduce that
$\mu_1({\bf R}_{\boldsymbol\alpha})=\mu_2({\bf R}_{\boldsymbol\alpha})$ and
$\mu_1({\bf R}_{\boldsymbol\alpha}^*)=\mu_2({\bf R}_{\boldsymbol\alpha}^*)$ for  $\boldsymbol\alpha\in \FF_{n_1}^+\times \cdots \times \FF_{n_k}^+$, and
$\mu_1({\bf R}_{\boldsymbol\alpha}^* {\bf R}_{\boldsymbol\beta})=\mu_2({\bf R}_{\boldsymbol\alpha}^* {\bf R}_{\boldsymbol\beta})=0$ if ${\bf R}_{\boldsymbol\alpha}^* {\bf R}_{\boldsymbol\beta}$ is not equal to ${\bf R}_{\boldsymbol\gamma}$ or ${\bf R}_{\boldsymbol\gamma}^*$ for some ${\boldsymbol\gamma}\in {\bf F}_{\bf n}^+$.
Hence, we deduce that  $\mu_1=\mu_2$.

According to  Theorem \ref{pluri-measure}, for any $\varphi\in {\text{\bf RH}}^1({\bf B_n})$, there is   a completely bounded linear  map $\mu_\varphi \in \text{\rm CB}(
\boldsymbol\cR_{\bf n}^* \boldsymbol\cR_{\bf n}, B(\cE))$ such that \ $\varphi=\boldsymbol\cP\mu_\varphi$ and $\|\varphi\|_1=\|\mu_\varphi\|_{cb}$. This proves that
  the map
$\Psi$ is surjective and  $\|\boldsymbol\cP\mu_\varphi\|_1=\|\mu_\varphi\|_{cb}$.   Therefore, item (i) is equivalent to (ii).

Now, the latter equivalence  and Theorem
\ref{Poisson-factor} imply
\begin{equation}
\label{HBP} \varphi({\bf X} )=(\boldsymbol\cP\mu_\varphi)({\bf X})=\widehat \mu_\varphi(C_{\bf X}^* C_{\bf X}),\qquad {\bf X}\in {\bf B_n}(\cH),
\end{equation}
where
$$ C_{\bf X} := (I_{\otimes_{i=1}^kF^2(H_{n_i})} \otimes \boldsymbol\Delta_{\bf X}(I)^{1/2})\prod_{i=1}^k
 \left(I-{\bf R}_{i,1}\otimes X_{i,1}^*-\cdots -{\bf R}_{i,n_i}\otimes X_{i,n_i}^*\right)^{-1}.
$$
 Due to  Wittstock's extension theorem
\cite{W2}, there exists a completely bounded map
$\Phi:C^*({\bf R} )\to B(\cE)$ that extends $\mu_\varphi$ and
$\|\mu_\varphi\|_{cb}=\|\Phi\|_{cb}$. According to Theorem 8.4 from
\cite{Pa-book},  there exists a Hilbert space $\cK$, a
$*$-representation $\pi:C^*({\bf R} )\to B(\cK)$, and bounded
operators $W_1,W_2:\cE\to \cK$,   with $\|\Phi\|=\|W_1\|\|W_2\|$
such that
\begin{equation}
\label{pi**}
\Phi(A)=W_1^*\pi(A) W_2,\qquad A\in C^*({\bf R} ).
\end{equation}
Set $V_{i,j}:=\pi({\bf R}_{i,j})$ for $i\in \{1,\ldots, k\}$ and $j\in \{1,\ldots, n_i\}$ and note that    ${\bf V}=(V_1,\ldots, V_n)$, is a $k$-tuple  of doubly commuting row isometries $V_i=(V_{i,1},\ldots, V_{i,n_i})$.
  Using
now  relations \eqref{HBP} and \eqref{pi**}, one can deduce item (iii).
The proof of  the  implication (iii)$\implies$ (i) is similar to the proof of the same implication from Theorem \ref{pluri-measure}.
\end{proof}

Consider now the subspace of free holomorphic functions
${\bf H}^1({\bf B_n}):=Hol({\bf B_n})\bigcap {\text{\bf PH}}^1( {\bf B_n})$ together
with the norm $\|\cdot\|_1$. Using  Theorem \ref{Har-1}, we can obtain the following   weak  analogue of the F. and M.~Riesz
Theorem \cite{H} in our setting.

\begin{corollary}\label{FMR}
$\left({\bf H}^1({\bf B_n}), \|\cdot \|_1\right)$ is a Banach space which
can be identified with the annihilator of $\boldsymbol\cR_n$ in $\text{\rm
CB}_0\,(\boldsymbol\cR_{\bf n}^* \boldsymbol\cR_{\bf n}, B(\cE))$, i.e.,
$$
(\boldsymbol\cR_n)^{\perp}:=\{\mu\in \text{\rm CB}_0\,( \boldsymbol\cR_{\bf n}^* \boldsymbol\cR_{\bf n},
B(\cE)): \ \mu({\bf R}_{\boldsymbol\alpha})=0 \text{ for all } \boldsymbol\alpha\in {\bf F}_{\bf n}^+,  |\boldsymbol\alpha|\geq 1\}.
$$
Moreover,  for each $f\in {\bf H}^1({\bf B_n})$, there is a unique
completely bounded linear map $\mu_f\in (\boldsymbol\cR_n)^{\perp}$ such that
$f=\boldsymbol\cP\mu_f$.
\end{corollary}

Given a  a
completely  bounded linear  map   $\mu:\boldsymbol\cR_{\bf n}^* \boldsymbol\cR_{\bf n}\to
B(\cE)$,
we introduce  the {\it noncommutative Fantappi\` e transform }  of  $\mu$ to be the map $\boldsymbol\cF\mu: {\bf B}_{\bf n}(\cH) \to
B(\cE)\otimes_{min}B(\cH)$ defined by
$$
(\boldsymbol\cF\mu)({\bf X} ):= \widehat\mu\left[\prod_{i=1}^k(I-{\bf R}_{i,1}^*\otimes X_{i,1}-\cdots
-{\bf R}_{i,n_i}^*\otimes X_{i,n_i})^{-1}\right]
$$
for ${\bf X}:=(X_1,\ldots, X_k)\in {\bf B}_{\bf n}(\cH)$. We remark that the
noncommutative Fantappi\` e transform  is a   linear map and
$\boldsymbol\cF\mu$ is a free holomorphic function in the open polyball
${\bf B}_{\bf n}$,  with coefficients in $B(\cE)$.

Let $\mu: \boldsymbol\cR_{\bf n}^* \boldsymbol\cR_{\bf n}\to B(\cE)$ be a completely positive linear map. We introduce the {\it noncommutative Herglotz-Riesz
transform} of $\mu$ on the regular polyball to be the map ${\bf H}\mu: {\bf B}_{\bf n}(\cH) \to B(\cE)\otimes_{min}B(\cH)$  defined by
$$
({\bf H}\mu)({\bf X} ):=\widehat\mu\left[ 2\prod_{i=1}^k(I-{\bf R}_{i,1}^*\otimes X_{i,1}-\cdots
-{\bf R}_{i,n_i}^*\otimes X_{i,n_i})^{-1}-I\right]
$$
for ${\bf X}:=(X_1,\ldots, X_k)\in {\bf B}_{\bf n}(\cH)$.
Note that
$({\bf H}\mu)({\bf X} )=2(\boldsymbol\cF\mu)({\bf X} )-\mu(I)\otimes I$.

\begin{theorem}\label{Herglotz}
Let $f$ be a function from the polyball ${\bf B_n}(\cH)$ to  $B(\cE)\otimes_{min}B(\cH)$. Then the following statements are equivalent.
\begin{enumerate}
\item[(i)] $f$ is a free holomorphic function  with $\Re f \geq 0$ and  the linear maps $\{\nu_{\Re f_r}\}_{r\in [0,1)}$ associated with $\Re f$  are completely positive.

\item[(ii)] The function $f$ admits a Herglotz-Riesz representation
$$f({\bf X})=({\bf H}\mu)({\bf X} )+i\Im f(0),
$$
where
  $\mu:C^*({\bf R})\to B(\cE)$ is a completely positive linear map with the property that
    $\mu({\bf R}_{\boldsymbol\alpha}^* {\bf R}_{\boldsymbol\beta})=0$   if \
 ${\bf R}_{\boldsymbol\alpha}^* {\bf R}_{\boldsymbol\beta}$ is not equal to ${\bf R}_{\boldsymbol\gamma}$ or ${\bf R}_{\boldsymbol\gamma}^*$ for some ${\boldsymbol\gamma}\in {\bf F}_{\bf n}^+$.

\item[(iii)] There exist a
$k$-tuple   ${\bf V}=(V_1,\ldots, V_k)$, $V_i=(V_{i,1},\ldots, V_{i,n_i})$, of doubly commuting row isometries on  a Hilbert space $\cK$, and  a bounded linear operator $W:\cE\to \cK$, such that
$$
f({\bf X})=(W^*\otimes I)\left[ 2\prod_{i=1}^k(I-V_{i,1}^*\otimes X_{i,1}-\cdots
-V_{i,n_i}^*\otimes X_{i,n_i})^{-1}-I\right](W\otimes I)+i \Im f(0)
$$
and $W^* {\bf V}_{\boldsymbol\alpha}^*{\bf V}_{\boldsymbol\beta }W=0$   if \
${\bf R}_{\boldsymbol\alpha}^* {\bf R}_{\boldsymbol\beta}$ is not equal to ${\bf R}_{\boldsymbol\gamma}$ or ${\bf R}_{\boldsymbol\gamma}^*$ for some $\boldsymbol\gamma\in {\bf F}_{\bf n}^+$.
\end{enumerate}

\end{theorem}
\begin{proof} We prove that (i)$\implies$(ii).
Let $f$ have the representation
$f({\bf X})=\sum_{{\boldsymbol \alpha}\in {\bf F}_{\bf n}^+} A_{(\boldsymbol\alpha)}\otimes {\bf X}_{\boldsymbol \alpha}$.
Due to Corollary \ref{cp}, there exists a completely positive linear map $\mu:C^*({\bf R})\to B(\cE)$ such that $\Re f=\boldsymbol\cP\mu$.
Consequently, we have
$$\mu(I)=\frac{1}{2}(A_{({\bf g})}+ A_{({\bf g})}^*),\quad \mu({\bf R}_{\widetilde{{\boldsymbol \alpha}}})=\frac{1}{2} A^*_{({\boldsymbol \alpha})},
\quad
\mu({\bf R}^*_{\widetilde{{\boldsymbol \alpha}}})=\frac{1}{2}A_{({\boldsymbol \alpha})},\qquad  \text{ for all } \boldsymbol\alpha\in {\bf F}_{\bf n}^+, |\boldsymbol\alpha|\geq 1,
$$
and
$\mu({\bf R}^*_{\widetilde{{\boldsymbol \alpha}}}{\bf R}_{\widetilde{{\boldsymbol \beta}}})=0$   if \
${\bf R}_{\boldsymbol\alpha}^* {\bf R}_{\boldsymbol\beta}$ is not equal to ${\bf R}_{\boldsymbol\gamma}$ or ${\bf R}_{\boldsymbol\gamma}^*$ for some $\boldsymbol \gamma\in {\bf F}_{\bf n}^+$.
Using the definition of the  Herglotz-Riesz transform, we obtain
\begin{equation*}
\begin{split}
({\bf H}\mu)({\bf X} )&:=\widehat\mu\left[ 2\prod_{i=1}^k(I-{\bf R}_{i,1}^*\otimes X_{i,1}-\cdots
-{\bf R}_{i,n_i}^*\otimes X_{i,n_i})^{-1}-I\right]\\
&=
\sum_{{\boldsymbol \alpha}\in {\bf F}_{\bf n}^+} A_{(\boldsymbol\alpha)}\otimes {\bf X}_{\boldsymbol \alpha}+
A_{({\bf g})}\otimes I- \frac{1}{2}(A_{({\bf g})}+ A_{({\bf g})}^*)\otimes I\\
&=f({\bf X})-\frac{1}{2}(A^*_{({\bf g})}- A_{({\bf g})})\otimes I\\
&=f({\bf X})-i\Im f(0),
\end{split}
\end{equation*}
which proves item (ii). Now we prove that (ii) implies (i). Assume that item (ii) holds.
Then
$$
f({\bf X})=2(\boldsymbol\cF\mu)({\bf X} )-\mu(I)\otimes I-i\Im f(0)
$$
is a free holomorphic function on the polyball ${\bf B_n}$. Taking into account that $\mu({\bf R}_{\boldsymbol\alpha}^* {\bf R}_{\boldsymbol\beta})=0$
 if \
${\bf R}_{\boldsymbol\alpha}^* {\bf R}_{\boldsymbol\beta}$ is not equal to ${\bf R}_{\boldsymbol\gamma}$ or ${\bf R}_{\boldsymbol\gamma}^*$ for some $\boldsymbol\gamma\in {\bf F}_{\bf n}^+$, and using Theorem \ref{Poisson-factor}, we deduce that
\begin{equation*}
\begin{split}
&\frac{1}{2} (f({\bf X})+f({\bf X})^*)=\frac{1}{2}\left(({\bf H}\mu)({\bf X} )+({\bf H}\mu)({\bf X} )^*\right)\\
&=\widehat\mu\left[ \prod_{i=1}^k(I-{\bf R}_{i,1}^*\otimes X_{i,1}-\cdots
-{\bf R}_{i,n_i}^*\otimes X_{i,n_i})^{-1}-I+\prod_{i=1}^k(I-{\bf R}_{i,1}\otimes X_{i,1}^*-\cdots
-{\bf R}_{i,n_i}\otimes X_{i,n_i}^*)^{-1}\right]\\
&=\widehat\mu\left[ \boldsymbol\cP({\bf R},{\bf X})\right]\geq 0.
\end{split}
\end{equation*}
Therefore, $\Re f$ is free $k$-pluriharmonic function such that $\Re f=\boldsymbol\cP \mu$. Due to Corollary \ref{cp}, we deduce that
the linear maps $\{\nu_{\Re f_r}\}_{r\in [0,1)}$ associated with $\Re f$  are completely positive.

Now, we prove the implication (ii) $\implies$ (iii). Assume that (ii) holds.
According to Stinespring's representation theorem \cite{St}, there is a Hilbert space $\cK$, a $*$-representation $\pi: C^*({\bf R})\to B(\cK)$ and a bounded
$W:\cE\to \cK$ with $\|\mu(I)\|=\|W\|^2$ such that
$\mu(A)=W^*\pi(A)W$ for all $A\in C^*({\bf R})$. Setting $V_{i,j}:=\pi ({\bf R}_{ij})$ for all $i\in \{1,\ldots, k\}$ and $j\in \{1,\ldots, n_i\}$, it is clear that the $k$-tuple ${\bf V}=(V_1,\ldots, V_k)$, $V_i=(V_{i,1},\ldots, V_{i,n_i})$, consists of doubly commuting row isometries.
Note that,  if \
${\bf R}_{\boldsymbol\alpha}^* {\bf R}_{\boldsymbol\beta}$ is not equal to ${\bf R}_{\boldsymbol\gamma}$ or ${\bf R}_{\boldsymbol\gamma}^*$ for some $\boldsymbol\gamma\in {\bf F}_{\bf n}^+$, then
$$
W^*{\bf V}_{\boldsymbol\alpha}^* {\bf V}_{\boldsymbol\beta}W=
W^*\pi({\bf R}_{\boldsymbol\alpha}^* {\bf R}_{\boldsymbol\beta})W=
\mu({\bf R}_{\boldsymbol\alpha}^* {\bf R}_{\boldsymbol\beta})=0.
$$
Now, one can easily see that the relation $f({\bf X})=({\bf H}\mu)({\bf X} )+i\Im f(0)$ leads to the representation in item (iii), which completes the proof.
It remains to prove that (iii) $\implies $ (ii). To this end, assume that (iii) holds.
Since  the $k$-tuple   ${\bf V}=(V_1,\ldots, V_k)$, $V_i=(V_{i,1},\ldots, V_{i,n_i})$, consists of  doubly commuting row isometries on  a Hilbert space $\cK$, the noncommutative von Neumann inequality \cite{Po-poisson} implies that the map $\pi:C^*({\bf R})\to B(\cE)$ defined by
$$\pi({\bf R}_{\boldsymbol\alpha}{\bf R}_{\boldsymbol\beta}^*):=
{\bf V}_{\boldsymbol\alpha}{\bf V}_{\boldsymbol\beta}^*,\qquad \boldsymbol\alpha,\boldsymbol\beta\in {\bf F}_{\bf n}^+,
$$
is a $*$-representation of $C^*({\bf R})$. Define $\mu:C^*({\bf R})\to B(\cE)$  by setting $\mu(A):=W^* \pi(A) W$. It is clear that $\mu$ is a completely positive linear map and item (iii) implies
$$
f({\bf X})=\widehat\mu\left[ 2\prod_{i=1}^k(I-{\bf R}_{i,1}^*\otimes X_{i,1}-\cdots
-{\bf R}_{i,n_i}^*\otimes X_{i,n_i})^{-1}-I\right] +i\Im f(0).
$$
Note also that
$$\mu({\bf R}_{\boldsymbol\alpha}^* {\bf R}_{\boldsymbol\beta})=W^*\pi({\bf R}_{\boldsymbol\alpha}^* {\bf R}_{\boldsymbol\beta})W=
W^*{\bf V}_{\boldsymbol\alpha}^* {\bf V}_{\boldsymbol\beta}W=0
$$
if \
${\bf R}_{\boldsymbol\alpha}^* {\bf R}_{\boldsymbol\beta}$ is not equal to ${\bf R}_{\boldsymbol\gamma}$ or ${\bf R}_{\boldsymbol\gamma}^*$ for some $\boldsymbol\gamma\in {\bf F}_{\bf n}^+$.
This shows that item (ii) holds.
The proof is complete.
\end{proof}
We remark that in the particular  case when $n_1=\cdots =n_k=1$,  we obtain an operator-valued extention of Kor\' annyi-Puk\' anszky   integral representation  for holomorphic functions with positive real part on polydisks \cite{KP}.

In what follows, we say that $f$ has a Herglotz-Riesz representation if item (ii) of Theorem \ref{Herglotz} is satisfied.

\begin{theorem} Let $f:{\bf B_n}(\cH)\to B(\cE)\otimes_{min}B(\cH)$ be a function,   where ${\bf n}=(n_1,\ldots,n_k)\in \NN^k$. If $f$   admits a Herglotz-Riesz representation, then
$f$ is a free holomorphic function  with $\Re f \geq 0$.

Conversely,
  if $f$ is  a free holomorphic function such that $\Re f\geq 0$, then there is a unique completely positive linear map
$\mu:\boldsymbol\cR_{\bf n}^* + \boldsymbol\cR_{\bf n} \to B(\cE)$ such that
$$
f({\bf Y})=({\bf H}\mu)(k{\bf Y})+i \Im f(0), \qquad {\bf Y}\in \frac{1}{k} {\bf B_n}(\cH).
$$
Moreover,
$$
f({\bf X})=2\sum_{{\boldsymbol \alpha}\in {\bf F}_{\bf n}^+} k^{|{\boldsymbol \alpha}|} \mu({\bf R}_{\widetilde{{\boldsymbol \alpha}}}^*)\otimes {\bf X}_{\boldsymbol \alpha}-\mu(I)\otimes I,\qquad {\bf X}\in {\bf B_n}(\cH).
$$
\end{theorem}
\begin{proof} The direct implication was proved in Theorem \ref{Herglotz}.
We prove the converse.  Assume that $f$ has the representation
\begin{equation}
\label{free-hol}
f({\bf X})=\sum_{{\boldsymbol \alpha}\in {\bf F}_{\bf n}^+} A_{(\boldsymbol\alpha)}\otimes {\bf X}_{\boldsymbol \alpha},\qquad {\bf X}\in {\bf B_n}(\cH).
\end{equation}
First we consider  the case when $\frac{1}{2}(A_{({\bf g})}+A_{({\bf g})}^*)=I_\cE$. Since $\Re f\geq 0$ and $\Re f (0)=I$, Theorem \ref{pluri-structure} shows that there is  a $k$-tuple  ${\bf V}=(V_1,\ldots, V_k)$ of commuting row isometries on a space $\cK\supset \cE$ such that
$$
 \Re f({\bf X})=
 \sum_{(\boldsymbol \sigma, \boldsymbol \omega)\in {\bf F}_{\bf n}^+\times {\bf F}_{\bf n}^+}
 P_\cE{\bf V}_{\widetilde{\boldsymbol\alpha}}^*{\bf V}_{\widetilde{\boldsymbol\beta}}|_\cE\otimes
 {\bf X}_{\boldsymbol \alpha} {\bf X}_{\boldsymbol\beta}^*,
 $$
The uniqueness of the representation of  free $k$-pluriharmonic functions  on ${\bf B_n}$ implies
\begin{equation}
\label{VV}
P_\cE{\bf V}_{\widetilde{\boldsymbol\alpha}}^*{\bf V}_{\widetilde{\boldsymbol\beta}}|_\cE
=
\begin{cases} \frac{1}{2}A_{(\boldsymbol\alpha)}
  &\quad \text{ if } \boldsymbol\alpha\in {\bf F}_{\bf n}^+, \boldsymbol\beta={\bf g}\\
   \frac{1}{2}A_{(\boldsymbol\beta)}^*
  &\quad \text{ if } \boldsymbol\beta\in {\bf F}_{\bf n}^+, \boldsymbol\alpha={\bf g}\\
 0& \quad \text{ otherwise}.
 \end{cases}
\end{equation}
Set $T_{i,j}:=\frac{1}{k} V_{i,j}$ for $i\in \{1,\ldots, k\}$ and $j\in \{1,\ldots, n_i\}$.
According to Proposition 1.9 from \cite{Po-Berezin-poly}, the $k$-tuple
${\bf T}:=(T_1,\ldots, T_k)$, with $T_i:=(T_{i,1},\ldots, T_{i,n_i})$, is in the closed polyball
${\bf B_n}(\cK)$. Using Theorem 7.2 from \cite{Po-Berezin-poly}, we find a $k$-tuple ${\bf W}:=(W_1,\ldots, W_k)$ of doubly commuting row isometries on a Hilbert space $\cG\supset \cK$ such that
$W_{i,j}^*|_\cK=T_{i,j}^*$  for all  $i\in \{1,\ldots, k\}$ and $j\in \{1,\ldots, n_i\}$.
Define the linear map $\mu:C^*({\bf R})\to B(\cE)$ by setting
$$
\mu({\bf R}_{\widetilde{\boldsymbol\beta}}{\bf R}_{\widetilde{\boldsymbol\alpha}}^*)=
P_\cE \left[P_\cK {\bf W}_{\widetilde{\boldsymbol\beta}}{\bf W}_{\widetilde{\boldsymbol\alpha}}^*|_\cK \right]|_\cE, \qquad \boldsymbol\alpha, \boldsymbol\beta\in {\bf F}_{\bf n}^+.
$$
Note that $\mu$ is a completely positive linear map with the property that
$\mu({\bf R}_{\widetilde{\boldsymbol\beta}})= \frac{1}{2k^{|\boldsymbol\beta|}}A_{(\boldsymbol\beta)}^*
$
and $\mu({\bf R}_{\widetilde{\boldsymbol\alpha}}^*)= \frac{1}{2k^{|\boldsymbol\alpha|}}A_{(\boldsymbol\alpha)}
$
if  $\boldsymbol\alpha, \boldsymbol\beta\in {\bf F}_{\bf n}^+$ with $\boldsymbol\alpha\neq {\bf g}$ and $\boldsymbol\beta\neq {\bf g}$, and $\mu(I)=I_\cE$.
Consequently, relations  \eqref{free-hol} and \eqref{VV} imply
$$
f({\bf X})=2\sum_{{\boldsymbol \alpha}\in {\bf F}_{\bf n}^+} k^{|{\boldsymbol \alpha}|} \mu({\bf R}_{\widetilde{{\boldsymbol \alpha}}}^*)\otimes {\bf X}_{\boldsymbol \alpha}-\mu(I)\otimes I,\qquad {\bf X}\in {\bf B_n}(\cH).
$$
Setting ${\bf Y}:=\frac{1}{k} {\bf X}$, we deduce that
\begin{equation*}
\begin{split}
f({\bf Y})&=2\sum_{{\boldsymbol \alpha}\in {\bf F}_{\bf n}^+}   \mu({\bf R}_{\widetilde{{\boldsymbol \alpha}}}^*)\otimes k^{|{\boldsymbol \alpha}|}{\bf Y}_{\boldsymbol \alpha}-\mu(I)\otimes I\\
&=
\widehat\mu\left[ 2\prod_{i=1}^k(I-{\bf R}_{i,1}^*\otimes kY_{i,1}-\cdots
-{\bf R}_{i,n_i}^*\otimes kY_{i,n_i})^{-1}-I\right]\\
&=({\bf H}\mu)(k{\bf Y})
\end{split}
\end{equation*}
for any ${\bf Y}\in \frac{1}{k} {\bf B_n}(\cH)$, which completes the proof
when $A_{({\bf g})}=I_\cE$.
Now, we consider the case when  $C_{({\bf g})}:=\frac{1}{2}(A_{({\bf g})}+A_{({\bf g})}^*)\geq 0$. For each $\epsilon>0$, define the free holomorphic function
 $$g_\epsilon:= (C_{({\bf g})}+\epsilon I)^{-1/2}[f+\epsilon I_\cE\otimes I_\cH](C_{({\bf g})}+\epsilon I)^{-1/2}
 $$
  and note that $\Re g_\epsilon (0)= I$.
  Applying the first part of the proof to $g_\epsilon$, we find  a completely positive linear map
  $\mu_\epsilon:C^*({\bf R})\to B(\cE)$
  with the property that
$$
\mu_\epsilon({\bf R}_{\widetilde{\boldsymbol\beta}})= \frac{1}{2k^{|\boldsymbol\beta|}}(C_{({\bf g})}+\epsilon I)^{-1/2}A_{(\boldsymbol\beta)}^*(C_{({\bf g})}+\epsilon I)^{-1/2}
$$
and
$$\mu_\epsilon({\bf R}_{\widetilde{\boldsymbol\alpha}}^*)= \frac{1}{2k^{|\boldsymbol\alpha|}}(C_{({\bf g})}+\epsilon I)^{-1/2}A_{(\boldsymbol\alpha)}(C_{({\bf g})}+\epsilon I)^{-1/2}
$$
if  $\boldsymbol\alpha, \boldsymbol\beta\in {\bf F}_{\bf n}^+$ with $\boldsymbol\alpha\neq {\bf g}$ and $\boldsymbol\beta\neq {\bf g}$, and $\mu_\epsilon(I)=I_\cE$.
Setting $\nu_\epsilon(\xi):=(C_{({\bf g})}+\epsilon I)^{1/2}\mu_\epsilon(\xi)(C_{({\bf g})}+\epsilon I)^{1/2}$ for all $\xi\in C^*({\bf R})$,  one can easily see that $\nu_\epsilon$ is a completely positive linear map with the property that
$\nu_\epsilon({\bf R}_{\widetilde{\boldsymbol\beta}})= \frac{1}{2k^{|\boldsymbol\beta|}}A_{(\boldsymbol\beta)}^*
$
and $\nu_\epsilon({\bf R}_{\widetilde{\boldsymbol\alpha}}^*)= \frac{1}{2k^{|\boldsymbol\alpha|}}A_{(\boldsymbol\alpha)}
$
if  $\boldsymbol\alpha, \boldsymbol\beta\in {\bf F}_{\bf n}^+$ with $\boldsymbol\alpha\neq {\bf g}$ and $\boldsymbol\beta\neq {\bf g}$, and $\nu_\epsilon(I)=C_{({\bf g})}+\epsilon I_\cE$.
Following the proof of Theorem \ref{pluri-measure} and  Corollary \ref{cp}, we find a completely positive linear map
$\nu:C^*({\bf R})\to B(\cE)$ such that $\nu(\xi)=\text{\rm WOT-}\lim_{\epsilon_k\to 0} \nu_{\epsilon_k}(\xi)$ for $\xi\in C^*({\bf R})$, where $\{\epsilon_k\}$ is a sequence of positive numbers converging to  zero. Consequently, we have
$\nu({\bf R}_{\widetilde{\boldsymbol\beta}})= \frac{1}{2k^{|\boldsymbol\beta|}}A_{(\boldsymbol\beta)}^*
$
and $\nu({\bf R}_{\widetilde{\boldsymbol\alpha}}^*)= \frac{1}{2k^{|\boldsymbol\alpha|}}A_{(\boldsymbol\alpha)}
$
if  $\boldsymbol\alpha, \boldsymbol\beta\in {\bf F}_{\bf n}^+$ with $\boldsymbol\alpha\neq {\bf g}$ and $\boldsymbol\beta\neq {\bf g}$, and $\nu(I)=C_{({\bf g})}$. As in the first part of this proof, one can easily see that
$$
f({\bf Y})=({\bf H}\nu)(k{\bf Y})+i \Im f(0), \qquad {\bf Y}\in \frac{1}{k} {\bf B_n}(\cH).
$$
and
$$
f({\bf X})=2\sum_{{\boldsymbol \alpha}\in {\bf F}_{\bf n}^+} k^{|{\boldsymbol \alpha}|} \nu({\bf R}_{\widetilde{{\boldsymbol \alpha}}}^*)\otimes {\bf X}_{\boldsymbol \alpha}-\nu(I)\otimes I,\qquad {\bf X}\in {\bf B_n}(\cH).
$$
The proof is complete.
\end{proof}

\bigskip

       %

      \end{document}